\documentclass[11pt,a4paper]{article}
\usepackage{multirow}
\pdfoutput=1

\usepackage{amssymb}
\usepackage{amsmath}
\usepackage{amsfonts}
\usepackage{bbm}
\usepackage{amsthm}
\usepackage{mathrsfs}
\usepackage{hyperref}
\usepackage{color}
\usepackage[margin=2.41cm]{geometry}
\usepackage[all,cmtip]{xy}
\usepackage[utf8]{inputenc}
\usepackage{graphicx}
\usepackage{varwidth}
\usepackage{comment}
\usepackage{enumitem}

\usepackage{bm}
\usepackage{mathtools}

\usepackage{upgreek}
\usepackage{rotating}

\usepackage{tikz}
\usetikzlibrary{cd,nfold,matrix,arrows,calc,decorations.pathmorphing,fit,shapes.geometric,decorations.pathreplacing,positioning}
\tikzset{Rightarrow/.style={double equal sign distance,>={Implies},->},
triple/.style={-,preaction={draw,Rightarrow}},
quadruple/.style={preaction={draw,Rightarrow,shorten >=0pt},shorten >=1pt,-,double,double
distance=0.2pt}}

\usepackage{quiver}

\definecolor{darkred}{rgb}{0.8,0.1,0.1}
\hypersetup{
	colorlinks=true,         
	urlcolor=darkred,
	linkcolor=darkred,
	citecolor=blue,
}

\theoremstyle{plain}
\newtheorem{theo}{Theorem}[section]
\newtheorem{lem}[theo]{Lemma}
\newtheorem{propo}[theo]{Proposition}
\newtheorem{cor}[theo]{Corollary}

\theoremstyle{definition}
\newtheorem{defi}[theo]{Definition}

\newtheorem{exx}[theo]{Example}
\newenvironment{ex}
{\pushQED{\qed}\exx}
{\popQED\endexx}

\newtheorem{remm}[theo]{Remark}
\newenvironment{rem}
{\pushQED{\qed}\remm}
{\popQED\endremm}

\newtheorem{constrr}[theo]{Construction}
\newenvironment{constr}
{\pushQED{\qed}\constrr}
{\popQED\endconstrr}

\numberwithin{equation}{section}

\def\nn{\nonumber}

\def\KZ{\mathrm{KZ}}

\def\bbK{\mathbb{K}}
\def\bbR{\mathbb{R}}
\def\bbC{\mathbb{C}}
\def\bbN{\mathbb{N}}

\def\hom{\underline{\mathrm{hom}}}

\def\Sym{\mathrm{Sym}}

\def\id{\mathrm{id}}
\def\Id{\mathrm{Id}}
\def\ID{\mathrm{ID}}

\def\sf{\mathrm{sf}}
\def\ps{\mathrm{ps}}

\def\dd{\mathrm{d}}

\def\L{\mathrm{cl}}

\def\hh{\mathrm{h}}

\def\1{I}
\def\oone{\mathbbm{1}}

\def\Vec{\mathbf{Vec}}
\def\Ch{\mathbf{Ch}}

\def\dgMod{\mathbf{dgMod}}

\def\CC{\mathbf{C}}
\def\DD{\mathbf{D}}

\def\EE{\mathbf{E}}

\def\LL{\mathbf{L}}
\def\RR{\mathbf{R}}

\def\dgCat{\mathbf{dgCat}}

\def\ad{\mathrm{ad}}

\def\H{\mathcal{H}}
\def\L{\mathcal{L}}
\def\R{\mathcal{R}}
\def\O{\mathcal{O}}

\def\T{\mathsf{T}}

\def\Pol{\mathrm{Pol}}

\newcommand\und[1]{\underline{#1}}

\makeatletter
\let\@fnsymbol\@alph
\makeatother

\title{%
Syllepses from 3-shifted Poisson structures and second-order integration of infinitesimal 2-braidings}

\author{%
Cameron James Deverall Kemp\vspace{4mm}\\
{\small School of Mathematical Sciences, University of Nottingham,}\\
{\small University Park, Nottingham NG7 2RD, United Kingdom.}\vspace{4mm}\\
{\small \begin{tabular}{ll}
Email: & \href{mailto:cameron.kemp@nottingham.ac.uk}{\texttt{cameron.kemp@nottingham.ac.uk}}
\vspace{2mm}
\end{tabular}
}
}

\date{May 2025}

\begin{document}

\maketitle

\begin{abstract}
\noindent This paper follows on from ``Infinitesimal 2-braidings from 2-shifted Poisson structures". It is demonstrated that the hexagonators appearing at second order satisfy the requisite axioms of a braided monoidal cochain 2-category provided that the strict infinitesimal 2-braiding is totally symmetric and coherent (in Cirio and Faria Martins' sense). We show that those infinitesimal 2-braidings induced by 2-shifted Poisson structures are indeed totally symmetric and we relate coherency to the third-weight component of the Maurer-Cartan equation that a 2-shifted Poisson structure must satisfy. Furthermore, we show that 3-shifted Poisson structures and ``coboundary" 2-shifted Poisson structures induce syllepses.
\end{abstract}
\vspace{-1mm}

\vspace{-2mm}

\vspace{-2mm}

\tableofcontents

\subsection{Conventions and notations}
Throughout this paper we take our ground field $\bbK$ to be a generic field of characteristic 0 and our base algebra $R$ to be a generic associative unital $\bbK$-algebra.

If we invoke the phrase ``truncation of the homs" then we are referring to the truncated tensor product $\tilde{\otimes}_R$ of two $[-1,0]$-concentrated cochain complexes $U,V\in\Ch_R^{[-1,0]}$, see \cite[(A.8b)]{Us}. In particular, this monoidal product of the base enrichment category allows us to move the hom differential $\partial$ across a product of two homotopies in a $\Ch_R^{[-1,0]}$-category. Explicitly, for two homotopies $\phi$ and $\phi'$, we have 
\begin{equation}
(\partial\phi)\phi'=\phi\,\partial\phi'\quad.
\end{equation}
Boldface is used to denote a category or highlight the term being defined whereas italics are used for emphasis. For example, a \textbf{cochain 2-category} is a $\Ch_R^{[-1,0]}$-category and ``infinitesimal 2-braidings" are defined to be \textit{pseudo}natural rather than 2-natural.

We sometimes use the notation $\equiv$ to stress that the equality holds between pseudonatural transformations.

\section{Introduction}
It is well established that braided monoidal linear categories play an active role in the study of: 3D TQFTs\footnote{For example, the Reshetikhin-Turaev construction uses modular tensor categories.} \cite{Turaev,Virelizier,Renzi}, rational 2D CFTs \cite{Fuchs,Huang}, (2+1)D topological phases of matter\footnote{Braided fusion categories model the fusion and braiding statistics of quasi-particles.} \cite{Kong} and, more specifically, topological quantum computing \cite{Delaney}. More generally, they have seen applications to broader topics within Topology and Computation \cite{Stay}. 

There are many known results and conjectures concerning the application of braided monoidal 2-categories to, for example, 4D TQFTs \cite{Neuchl}. See \cite[Section 2.1]{Schommer} for a clear and concise overview of the history of braided monoidal 2-categories and the terms involved therein. In particular, the 2-categorical context means that the braiding includes data of ``hexagonators" and one has room to study additional structures on top of the braiding such as ``syllepses"; these too have applications, e.g., in (3+1)D topological phases of matter \cite{Decoppet}.

The relevant braided monoidal linear (2-)categories we are concerned with are constructed as deformation quantisations of symmetric strict monoidal linear (2-)categories. In this vein, it was Cartier \cite{Cartier} who showed how to categorify Drinfeld's algebraic deformation quantisation results \cite{Drinfeld90}, i.e., Cartier translated it into the problem of integrating a symmetric infinitesimal braiding on a symmetric strict monoidal $\bbC$-linear category to construct a braided monoidal $\bbC[[\hbar]]$-linear category, see \cite[Theorem XX.6.1]{Kassel}. Briefly, naturality of the infinitesimal braiding $t:\otimes\Rightarrow\otimes$ guarantees the \textbf{four-term relations}:
\begin{equation}\label{eq:four-term relations}
[t_{ij},t_{ik}+t_{jk}]=0=[t_{jk},t_{ij}+t_{ik}]\qquad,\qquad|\{i,j,k\}|=3\quad,
\end{equation}
see \cite[Lemma XX.4.3]{Kassel}. Such relations \eqref{eq:four-term relations} imply flatness of the \textbf{Knizhnik-Zamolodchikov (KZ) connection}, 
\begin{equation}\label{eq:KZ connection}
\nabla_{\mathrm{KZ}}:=\hbar\sum_{1\leq i<j\leq n+1}\left(\frac{\dd z_i-\dd z_j}{z_i-z_j}\right)t_{ij}\quad,
\end{equation}
on the \textbf{configuration space of $n+1$ indistinguishable particles on the complex line}, 
\begin{equation}\label{eq:complex configuration space}
Y_{n+1}:=\{(z_1,\ldots,z_{n+1})\in\bbC^{n+1}|\mathrm{~if~}i\neq j\mathrm{~then~}z_i\neq z_j\}\quad,
\end{equation}
see \cite[Proposition XIX.2.1]{Kassel}. The Drinfeld KZ series $\Phi_\KZ(t_{12},t_{23})$ is then constructed as the non-singular factor of a parallel transport with respect to $\nabla_\mathrm{KZ}$ along a straightforward path in $Y_3$, see \cite[Theorem 20]{BRW}, and flatness is essential to proving that the pentagon (with associator $\alpha:=\Phi_\KZ(t_{12},t_{23})$) and hexagon (with braiding $\sigma:=\gamma\circ e^{\hbar\pi it}$ where $\gamma$ is the symmetric braiding) axioms are indeed satisfied.

With the goal of finding the appropriate 2-categorical analogue of the KZ connection, Cirio and Faria Martins (CM) \cite{Joao1,Joao,Joao2} considered the higher gauge theoretic concept of a ``2-connection"\footnote{See \cite{Baez} for a lucid introduction.}. Specifically, \cite[Theorem 10]{Joao1} finds necessary and sufficient conditions for their, by construction, fake flat KZ 2-connection to be 2-flat and \cite[Theorem 13]{Joao1} gives sufficient conditions for the 2-connection to descend to $X_n:=Y_n/\mathrm{S}_n$. They then abstracted from these conditions the notion of a ``Free infinitesimal 2-Yang-Baxter operator in a differential crossed module" in \cite[Definition 26]{Joao2}. They reversed the order of this investigation in their latter work \cite{Joao}, i.e., they \textit{began} with the definition of a ``strict infinitesimal 2-braiding" \cite[Definition 14]{Joao}. Importantly, defining infinitesimal 2-braidings as \textit{pseudo}natural endomorphisms of the monoidal product gave rise to obstructions to the four-term relations and these obstructions were described in \cite[(21)]{Joao}. They then provide the simple notions of a ``coherent" infinitesimal 2-braiding \cite[Definition 17]{Joao} and ``totally symmetric" infinitesimal 2-braiding \cite[Definition 18]{Joao}. Assuming the infinitesimal 2-braiding to be strict, the KZ 2-connection they construct in \cite[(58)]{Joao} is fake flat. Moreover, \cite[Theorem 23]{Joao} proves that it is 2-flat if the infinitesimal 2-braiding is also coherent and \cite[Theorem 24]{Joao} proves that it descends to $X_n:=Y_n/\mathrm{S}_n$ if the infinitesimal 2-braiding is totally symmetric. 

With $\alpha:=\Phi_\KZ(t_{12},t_{23})$ as ansatz associator and $\sigma:=\gamma\circ e^{\hbar\pi it}$ as ansatz braiding, it was found in \cite[(3.44)]{Us} that the hexagon axioms are obstructed at second order for a symmetric strict $t$. We prove herein that, assuming $t$ is totally symmetric and strict, these obstructions are hexagonators if and only if $t$ is coherent. This not only provides good indication that CM have found incredibly powerful notions for the problem of 2-categorical deformation quantisation but it also encourages the study of the 2-holonomy of their KZ 2-connection for the express purposes of constructing hexagonator series (to all orders in $\hbar$) whose second-order term coincides with such ``infinitesimal hexagonators". We will address this problem in a future work.

\subsection{Summary}
Section \ref{sec:Setup} introduces the objects of central interest, i.e., braided strictly-unital monoidal $\Ch_R^{[-1,0]}$-categories, see Definition \ref{def:bra str un mon cat}. In order to contextualise this definition, we will first need to recall essential concepts in \cite[Appendix A.2]{Us} surrounding the (closed) symmetric strict monoidal tricategory $\dgCat_R^{[-1,0],\ps}$ of $\Ch_R^{[-1,0]}$-categories. That being said, we slightly simplify terminology and notation. For instance: we will call them ``pseudonatural transformations" rather than ``$\Ch_R^{[-1,0]}$-enriched pseudo-natural transformations", we denote the homotopy components as $\xi_f$ rather than $\xi_{U,U'}[f]$, we will denote the $\circ$-composition (see Constructions \ref{con:vert comp pseudo} and \ref{con:vertcomp mods}) by juxtaposition, we introduce compact notation for whiskerings \eqref{eq:whiskering pseudo by functor} and \eqref{subeq:whisker mod by pseudo}, etc.  
 
Subsection \ref{subsec:tricat universe} introduces the important concept of ``pseudonatural isomorphisms" (see Definition \ref{def:psnat auto}) and mentions that they are preserved by juxtaposing, whiskering and monoidally composing. We also collect various structures of $\dgCat_R^{[-1,0],\ps}$ that will be necessary for definitions and constructions appearing in later sections. For example, the horizontal $*$ and vertical $\circ$ composition of pseudonatural transformations does not satisfy the exchange law\footnote{It is this fact \textit{alone} that makes $\dgCat_R^{[-1,0],\ps}$ a \textit{weak} 3-category.} but rather this is relaxed by what we call the ``exchanger" (see Construction \ref{con:comp constraint}) and we mention in Construction \ref{con:add mods} that we can $R$-linearly add pseudonatural transformations and modifications; the ``4-term relationators" of Subsection \ref{subsec:4T relationators} are constructed from both of these structures.

Subsection \ref{subs:brastrunmon cats} specifies the general definition of a braided monoidal bicategory in \cite[Appendix C]{Schommer} to our context where: we work with $\Ch^{[-1,0]}_R$-categories instead of bicategories, the monoidal product $\otimes$ is a $\Ch^{[-1,0]}_R$-functor rather than a pseudofunctor, the associator and braiding are pseudonatural isomorphisms instead of ``pseudonatural equivalences", the unitors and 2-unitors are all identities. The latter part of Definition \ref{def:strictly-unital monoidal Ch-cat} fixes the notions of a ``pentagonal strictly-unital monoidal $\Ch^{[-1,0]}_R$-category" and a ``strict monoidal $\Ch^{[-1,0]}_R$-category". Definition \ref{def:hexagonal braiding} introduces the concept of a ``hexagonally-braided" strict monoidal $\Ch^{[-1,0]}_R$-category and Remark \ref{rem:symstrict mon ChCat defi} discusses the even more special case of a symmetric strict monoidal $\Ch^{[-1,0]}_R$-category. 

Definition \ref{def:syllepsis} specifies the general notion of a symmetric syllepsis on a braided monoidal bicategory \cite[Definition C.5]{Schommer} to our context of braided strictly-unital monoidal $\Ch^{[-1,0]}_R$-categories. In Subsubsection \ref{subsub:inf syllep on symstrmon}, Remark \ref{rem:symmetrically-sylleptic symmetric strict} analyses the scenario of symmetric syllepses on symmetric strict monoidal $\Ch^{[-1,0]}_R$-categories and Definition \ref{def:infinitesimal syllepsis} introduces our notion of an ``infinitesimal syllepsis". To finish Section \ref{sec:syllepses}, Subsection \ref{subs:syllep from 3-shifted} shows that 3-shifted Poisson structures induce infinitesimal syllepses in much the same way that 2-shifted Poisson structures induce infinitesimal 2-braidings, i.e., we substitute $\pi_{n=3}^{(2)}$ for $\pi_{n=2}^{(2)}$ in \cite[(3.25a)]{Us}.

Section \ref{sec:inf2bra and coboundary syllepses} recalls the definition of a $\gamma$-equivariant semi-strict infinitesimal 2-braiding as in \cite[Definitions 3.2 and 3.6]{Us}. Actually, we will call them ``symmetric strict infinitesimal 2-braidings", see Definitions \ref{def: strict t} and \ref{def:sym t}. Definition \ref{def:Breen t} introduces the notion of a ``Breen infinitesimal 2-braiding" which is meant to formalise the order $h$ relation coming from the strict hexagonal Breen polytope axiom \eqref{eqn:strict Breen axiom} that a hexagonally-braided strict monoidal $\Ch_{\bbK[h]/(h^2)}^{[-1,0]}$-category must satisfy. We also introduce, in Subsubsection \ref{subsub:coboundary inf syllep}, our notion of a ``coboundary infinitesimal syllepsis" and show that these are induced by 2-shifted Poisson structures whose weight 2 term is coboundary. Subsection \ref{subsec:index notation} recalls the subscript index notation $t_{ij}:\otimes^n\Rightarrow\otimes^n$ from \cite[(3.32)]{Us} and then uses this notation to introduce our notions of ``left/right pre-hexagonators", upon choosing Drinfeld's KZ series $\Phi_\KZ(t_{12},t_{23})$ as our ansatz associator and $\gamma e^{\frac{h}{2}t}$ as our ansatz braiding.

Section \ref{sec:2nd order} is about showing that the hexagonators admit nontrivial second-order terms and satisfy the requisite axioms of a braided pentagonal strictly-unital monoidal $\Ch_{\bbK[h]/(h^3)}^{[-1,0]}$-category. Assuming a symmetric strict infinitesimal 2-braiding $t$, Subsubsection \ref{subsub:inf hex} constructs the ``infinitesimal hexagonators" \eqref{eq:inf hexagonators} from the 4-term relationators. 

Subsection \ref{subsec:verifying axioms} recalls CM's \cite{Joao} important notions of ``coherency" (see Definition \ref{def: coherency}) and ``total symmetry" (see Definition \ref{def: total sym}) of an infinitesimal 2-braiding. Remark \ref{rem:our t is coherent} delves into the difficulty of showing that those infinitesimal 2-braidings induced by 2-shifted Poisson structures \cite[Theorem 3.10]{Us} are indeed coherent\footnote{We will address this problem head-on in a future work where we will generalise Cartan's formula \cite[Proposition 3.6]{Poisson} to the context of commutative differential graded algebras (CDGAs) and/or construct the infinitesimal 2-braiding on the symmetric monoidal cochain 2-category $\mathsf{Ho}_2(\mathbf{wRep}_\bbK L)$ of weak representations of the $\mathbb{L}_\infty$-algebra $L$, equivariant maps and homotopies modulo 2-homotopies.}.

We demonstrate in Propositions \ref{propo: SM1/2 hold}, \ref{prop:SM3/4 hold} and \ref{propo: h^2 Breen polytope} that a coherent totally symmetric strict infinitesimal 2-braiding provides a second-order deformation quantisation of a symmetric strict monoidal $\Ch^{[-1,0]}_\bbK$-category to a braided pentagonal strictly-unital monoidal $\Ch^{[-1,0]}_{\bbK[h]/(h^3)}$-category. We then show, in Remark \ref{rem:h^2 coboundary syllepsis}, that the coboundary symmetric syllepsis of Remark \ref{rem:coboundary 2-shifted induces syllep} is also a symmetric syllepsis on the aforementioned braided pentagonal strictly-unital monoidal $\Ch^{[-1,0]}_{\bbK[h]/(h^3)}$-category.

We include Appendix \ref{app:Third-order data} for three reasons: the first is that we want to show that the pentagonator $\Pi$ is \textit{actually} nontrivial and this nontriviality begins at order $h^3$. The second reason is that we wish to determine the order $h^3$ term of the right pre-hexagonator \eqref{eq:symstr right pre-hex}. Even though we do not wish to go through the tedious exercise of showing that these third-order expansions (of the pentagonator and hexagonators) satisfy the requisite axioms, the third reason for the appendix is that we wish to show that the third-order expansion of the coboundary syllepsis satisfies its requisite axioms.

\section{Setup}\label{sec:Setup}
See \cite[Chapters 4, 11 and 12]{2D-cats} for the utmost general context concerning all the definitions and constructions appearing throughout this section.

\subsection{The closed symmetric strict monoidal tricategory of cochain 2-categories}\label{subsec:tricat universe}
This subsection begins from \cite[Appendix A.2]{Us} and only recalls the necessary details. For example, we will not recall the horizontal composition $*$ of modifications for we will not use such a composition throughout the entire paper.
\subsubsection{Pseudonatural transformations, modifications and their compositions}
In particular, we start from \cite[Definition A.5]{Us}.
\begin{defi}\label{def:pseudonatural}
Given $\Ch_R^{[-1,0]}$-functors $F,G : \CC\to\DD$, a \textbf{pseudonatural
transformation} $\xi : F\Rightarrow G$ consists of the following two pieces of data:
\begin{itemize}
\item[(i)] For each 0-cell $U\in \CC$, a 1-cell $\xi_U\in\DD[F(U),G(U)]^0$ in the $\hom$-object $\DD[F(U),G(U)]\in\Ch_R^{[-1,0]}$ which we call a \textbf{cochain component}.

\item[(ii)] For each pair of 0-cells $U,U'\in \CC$, an $R$-linear map whose output we call a \textbf{homotopy component},
\begin{flalign}\label{eqn:doubleindexedcomponents}
\xi_{(-)}\,:\, \CC[U,U']^0~\longrightarrow~\DD[F(U),G(U')]^{-1}\quad,
\end{flalign}
such that for all 1-cells $f\in\CC[U,U']^0$,
\begin{subequations}\label{eqn: dubindex ind as homotopy}
\begin{flalign}
G(f)\, \xi_U - \xi_{U'}\,F(f)\,=\partial\xi_f\quad,
\end{flalign}
and for all 2-cells $q\in \CC[U,U']^{-1}$,
\begin{flalign}\label{eq:naturality of homotopy components}
G(q)\, \xi_U - \xi_{U'}\,F(q)\,=\, \xi_{\partial q}\quad.
\end{flalign}
\end{subequations}
\end{itemize}
These two pieces of data have to satisfy the following two axioms:
\begin{itemize}
\item[(1)] For all $U\in \CC$,
\begin{subequations}
\begin{flalign}\label{eqn: homotopy kills unit}
\xi_{1_U} \,=0\quad.
\end{flalign}
\item[(2)] For all $f\in \CC[U,U']^0$ and $f'\in \CC[U',U'']^0$,
\begin{flalign}\label{eqn:dubindex splits prods}
\xi_{f' f}\,=\,\,\xi_{f'}\, F(f) + G(f')\,\xi_f\quad.
\end{flalign}
\end{subequations}
\end{itemize}
A \textbf{pseudonatural endomorphism} is one of the form $\upsilon : F\Rightarrow F$.
\end{defi}
\begin{rem}
Let us note that a pseudonatural transformation with trivial homotopy components, $\xi_f=0$, is simply a $\Ch_R^{[-1,0]}$-natural transformation.
The \textbf{zero pseudonatural transformation} $0:F\Rightarrow G$ is such an example.
\end{rem}
\begin{ex}\label{ex:identity pseudonatural transformation}
The \textbf{identity pseudonatural transformation}
$\Id_F : F\Rightarrow F : \CC\to \DD$, defined by $(\Id_F)_U =1_{F(U)}$ and
$(\Id_{F})_f=0$, is both a pseudonatural endomorphism and $\Ch_R^{[-1,0]}$-natural transformation, i.e., it is a $\Ch_R^{[-1,0]}$-natural endomorphism. We sometimes denote such identities as $1:=\Id_F$ when the context makes it clear what the $\Ch^{[-1,0]}_R$-functor $F$ is.
\end{ex}
\begin{defi}\label{defi: modification}
Given two pseudonatural transformations $\xi,\xi'\in\dgCat_R^{[-1,0],\ps}[\CC,\DD](F,G)$, a \textbf{modification} $\Xi:\xi\Rrightarrow\xi':F\Rightarrow G:\CC\rightarrow\DD$ consists of the following datum:
\begin{enumerate}
\item[(i)] For each 0-cell $U\in\CC$, a 2-cell $\Xi_U\in\DD[F(U),G(U)]^{-1}$ such that 
\begin{equation}\label{eqn: mod is homot between pseudos}
\xi_U-\xi'_U=\partial\Xi_U\quad.
\end{equation}
\end{enumerate}
This datum is subject to the single axiom:
\begin{enumerate}
\item For every 1-cell $f\in\CC[U,U']^0$,
\begin{equation}\label{eqn:mod single condition}
\Xi_{U'}F(f)+\xi_f=\xi'_f+G(f)\Xi_U\quad.
\end{equation}
\end{enumerate}
An \textbf{endomodification} $\Xi\in\mathrm{Mod}_F$ is a modification of the form $\Xi:\xi\Rrightarrow\xi':F\Rightarrow F$, i.e., it is a modification between pseudonatural endomorphisms and \textit{not} an endomorphic modification $\Theta:\theta\Rrightarrow\theta:F\Rightarrow G$. We denote by $\ID_\xi:=\mathbf{0}:\xi\Rrightarrow\xi$ the \textbf{zero modification} which is an instance of an endomorphic modification. 
\end{defi}

\begin{constr}\label{con:vert comp pseudo}
The \textbf{vertical composite pseudonatural transformation}
\begin{subequations}\label{eqn: ver comp of pseudos}
\begin{equation}
\begin{tikzcd}
	{\CC} && {\DD}
	\arrow[""{name=0, anchor=center, inner sep=0}, "G"{description}, from=1-1, to=1-3]
	\arrow[""{name=1, anchor=center, inner sep=0}, "F", curve={height=-32pt}, from=1-1, to=1-3]
	\arrow[""{name=2, anchor=center, inner sep=0}, "H"',curve={height=32pt}, from=1-1, to=1-3]
	\arrow["\xi"', shorten <=5pt, shorten >=5pt, Rightarrow, from=1, to=0]
	\arrow["\theta"', shorten <=5pt, shorten >=5pt, Rightarrow, from=0, to=2]
\end{tikzcd}~~\stackrel{\circ}{\longmapsto}~~
\begin{tikzcd}
	{\CC} && {\DD}
	\arrow[""{name=1, anchor=center, inner sep=0}, "F", curve={height=-32pt}, from=1-1, to=1-3]
	\arrow[""{name=2, anchor=center, inner sep=0}, "H"', curve={height=32pt}, from=1-1, to=1-3]
	\arrow["\theta\xi"', shorten <=5pt, shorten >=5pt, Rightarrow, from=1, to=2]
\end{tikzcd}
\end{equation}
in the $2$-category $\dgCat_R^{[-1,0],\ps}\big[\CC,\DD\big]$ is defined by:
\begin{alignat}{2}
(\theta\xi)_U\,&:=\, \theta_U\,\xi_U\quad&&,\\
(\theta \xi)_f\,&:=\, \theta_f\,\xi_U + \theta_{U'}\,\xi_f\quad&&.\label{eq:hom components of vercomp pseudos}
\end{alignat}
\end{subequations}
This composition
is associative and unital with respect to the identity 1.
\end{constr}
\begin{defi}\label{def:psnat auto}
A pseudonatural transformation $\xi:F\Rightarrow G:\CC\rightarrow\DD$ is a \textbf{pseudonatural isomorphism} if there exists another pseudonatural transformation $\xi^{-1}:G\Rightarrow F:\CC\rightarrow\DD$ such that $\xi^{-1}\xi=\Id_F$ and $\xi\,\xi^{-1}=\Id_G$. A \textbf{pseudonatural automorphism} is a pseudonatural endomorphism which is also a pseudonatural isomorphism.
\end{defi}
It is obvious from the above definition that pseudonatural isomorphisms are closed under vertical composition. 
\begin{rem}\label{rem:pseudonatural endo is auto iff}
The formulae for the vertical composition/unit in Construction \ref{con:vert comp pseudo} tells us that a pseudonatural transformation $\xi:F\Rightarrow G:\CC\rightarrow\DD$ is a pseudonatural isomorphism if and only if the cochain components $\xi_U\in\DD[F(U),G(U)]^0$ are isomorphisms $\xi_U:F(U)\cong G(U)$. In other words, it can be easily checked that setting:
\begin{subequations}\label{subeq:inverse pseudo iso formulae}
\begin{alignat}{2}
\xi^{-1}_U:=&(\xi_U)^{-1}:G(U)\cong F(U)\quad&&,\\
\xi^{-1}_f:=&-(\xi_{U'})^{-1}\xi_f(\xi_U)^{-1}\in\DD\left[G(U),F(U')\right]^{-1}\quad&&,
\end{alignat}
\end{subequations}
produces a well-defined pseudonatural transformation $\xi^{-1}:G\Rightarrow F:\CC\rightarrow\DD$ which, by construction, is the \textbf{inverse} to $\xi$. 
\end{rem}
\begin{constr}
The \textbf{horizontal composite pseudonatural transformation}
\begin{subequations}
\begin{equation}
\begin{tikzcd}
	{\CC} && {\DD} && {\EE}
	\arrow[""{name=0, anchor=center, inner sep=0}, "F", curve={height=-32pt}, from=1-1, to=1-3]
	\arrow[""{name=1, anchor=center, inner sep=0}, "G"', curve={height=32pt}, from=1-1, to=1-3]
	\arrow[""{name=2, anchor=center, inner sep=0}, "{F'}", curve={height=-32pt}, from=1-3, to=1-5]
	\arrow[""{name=3, anchor=center, inner sep=0}, "{G'}"', curve={height=32pt}, from=1-3, to=1-5]
	\arrow["\xi\,"', shorten <=6pt, shorten >=6pt, Rightarrow, from=0, to=1]
	\arrow["{\,\upsilon}"', shorten <=6pt, shorten >=6pt, Rightarrow, from=2, to=3]
\end{tikzcd}
~~\stackrel{\ast}{\longmapsto}~~
\begin{tikzcd}
	{\CC} && {\EE}
	\arrow[""{name=0, anchor=center, inner sep=0}, "F' F", curve={height=-32pt}, from=1-1, to=1-3]
	\arrow[""{name=1, anchor=center, inner sep=0}, "G' G"', curve={height=32pt}, from=1-1, to=1-3]
	\arrow["\upsilon\ast\xi\,"', shorten <=6pt, shorten >=6pt, Rightarrow, from=0, to=1]
\end{tikzcd}
\end{equation}
in the tricategory $\dgCat_R^{[-1,0],\ps}$ is defined by
\begin{alignat}{2}
(\upsilon\ast\xi)_U\,&:=\,\upsilon_{G(U)}\,F'(\xi_U)\quad&&,\\
(\upsilon\ast\xi)_f\,&:=\,\upsilon_{G(f)}\,F'(\xi_U)
+\upsilon_{G(U')}\,F'\big(\xi_f\big)\quad&&.
\end{alignat}
\end{subequations}
This composition
is associative and unital with respect to the particular identity pseudonatural transformations
$\Id_{\id_\CC}:\id_\CC\Rightarrow\id_\CC:\CC\to\CC$. 
\end{constr}
\begin{rem}
The horizontal composition of two identity pseudonatural transformations is another identity pseudonatural transformation, i.e.,
\begin{equation}\label{eq:horcomp is strictly unitary}
\Id_{F'}*\Id_F=\Id_{F'F}\quad.
\end{equation}
This fact that $\ast$ is \textit{strictly unitary} will be implicit in many constructions throughout.
\end{rem}
\begin{rem}\label{rem:whiskering pseudo by functor}
\textbf{Whiskering} pseudonatural transformations by $\Ch^{[-1,0]}_R$-functors works as usual:
\begin{equation}\label{eq:whiskering pseudo by functor}
(\upsilon*\Id_G)_U=\upsilon_{G(U)}\,\,,\,\,(\upsilon*\Id_G)_f=\upsilon_{G(f)}\quad,\quad(\Id_{F'}*\xi)_U=F'(\xi_U)\,\,,\,\,(\Id_{F'}*\xi)_f=F'(\xi_f)\,,
\end{equation}
hence pseudonatural isomorphisms are closed under whiskering. We will denote \eqref{eq:whiskering pseudo by functor} by
\begin{equation}\label{eq:notation for whiskering pseudo}
\upsilon_G:=\upsilon*\Id_G\qquad,\qquad F'(\xi):=\Id_{F'}*\xi\quad.
\end{equation}
Associativity of horizontal composition gives us $H(\upsilon_F)=H(\upsilon)_F$.
\end{rem} 
\begin{constr}\label{con:comp constraint}
Given any composable diagram of the form
\begin{subequations}\label{eqn:compositioncoherences}
\begin{equation}
\begin{tikzcd}
	{\CC} && {\DD} && {\EE}
	\arrow[""{name=0, anchor=center, inner sep=0}, "F", curve={height=-32pt}, from=1-1, to=1-3]
	\arrow[""{name=1, anchor=center, inner sep=0}, "G"{description}, from=1-1, to=1-3]
	\arrow[""{name=2, anchor=center, inner sep=0}, "H"', curve={height=32pt}, from=1-1, to=1-3]
	\arrow[""{name=3, anchor=center, inner sep=0}, "{F'}", curve={height=-32pt}, from=1-3, to=1-5]
	\arrow[""{name=4, anchor=center, inner sep=0}, "{H'}"', curve={height=32pt}, from=1-3, to=1-5]
	\arrow[""{name=5, anchor=center, inner sep=0}, "{G'}"{description}, from=1-3, to=1-5]
	\arrow["\xi"', shorten <=5pt, shorten >=5pt, Rightarrow, from=0, to=1]
	\arrow["\theta"', shorten <=5pt, shorten >=5pt, Rightarrow, from=1, to=2]
	\arrow["{\xi'}"', shorten <=5pt, shorten >=5pt, Rightarrow, from=3, to=5]
	\arrow["{\theta'}"', shorten <=5pt, shorten >=5pt, Rightarrow, from=5, to=4]
\end{tikzcd}\quad,
\end{equation}
the \textbf{exchanger} is the modification
\begin{flalign}
\ast^2_{(\theta',\theta),(\xi',\xi)}\,:\,(\theta'\ast\theta)(\xi'\ast\xi)~\Rrightarrow~\theta'\xi'\ast\theta\xi
\end{flalign}
whose components are defined by
\begin{flalign}
\big(\ast^2_{(\theta',\theta),(\xi',\xi)}\big)_U\,:=\,\theta'_{H(U)}\,\xi'_{\theta_U}F'(\xi_U)\quad.
\end{flalign}
\end{subequations}
\end{constr} 
\begin{rem}\label{rem:whiskering and exchanger}
The exchanger \eqref{eqn:compositioncoherences} is trivial if $\xi'$ is a $\Ch_R^{[-1,0]}$-natural transformation or $\theta=1$, because $\xi'_{1_{G(U)}}=0$, hence whiskering is functorial, i.e.:
\begin{subequations}
\begin{equation}
\theta'_F\xi'_F=(\theta'\xi')_F\qquad,\qquad F'(\theta)F'(\xi)=F'(\theta\xi)\quad.
\end{equation}
In addition, strict unitarity \eqref{eq:horcomp is strictly unitary} gives us:
\begin{equation}
(\xi'_G)^{-1}:=(\xi'*\Id_G)^{-1}=\xi'^{-1}*\Id_G=:\xi'^{-1}_G\quad,\quad\left(F'(\xi)\right)^{-1}:=(\Id_{F'}*\xi)^{-1}=\Id_{F'}*\xi^{-1}=:F'(\xi^{-1})\,.
\end{equation}
\end{subequations}
\end{rem}
\begin{constr}\label{con:lat comp mods}
The \textbf{lateral composite modification}
\begin{subequations}\label{eqn:lat comp mods}
\begin{equation}
\begin{tikzcd}
	{\CC} &&& {\DD}
	\arrow[""{name=0, anchor=center, inner sep=0}, "G"', curve={height=36pt}, from=1-1, to=1-4]
	\arrow[""{name=1, anchor=center, inner sep=0}, "F", curve={height=-36pt}, from=1-1, to=1-4]
	\arrow[""{name=2, anchor=center, inner sep=0}, "{\,\xi''}", curve={height=-28pt}, shorten <=8pt, shorten >=8pt, Rightarrow, from=1, to=0]
	\arrow[""{name=3, anchor=center, inner sep=0}, "\xi\,"', curve={height=28pt}, shorten <=8pt, shorten >=8pt, Rightarrow, from=1, to=0]
	\arrow[""{name=4, anchor=center, inner sep=0}, "{\xi'}"{description}, shorten <=6pt, shorten >=6pt, Rightarrow, from=1, to=0]
	\arrow["\Xi", shorten <=4pt, shorten >=4pt, triple, from=3, to=4]
	\arrow["{\Xi'}", shorten <=4pt, shorten >=4pt, triple, from=4, to=2]
\end{tikzcd}
~~\stackrel{\cdot}{\longmapsto}~~
\begin{tikzcd}
	{\CC} &&& {\DD}
	\arrow[""{name=0, anchor=center, inner sep=0}, "G"', curve={height=36pt}, from=1-1, to=1-4]
	\arrow[""{name=1, anchor=center, inner sep=0}, "F", curve={height=-36pt}, from=1-1, to=1-4]
	\arrow[""{name=2, anchor=center, inner sep=0}, "{\,\xi''}", curve={height=-28pt}, shorten <=8pt, shorten >=8pt, Rightarrow, from=1, to=0]
	\arrow[""{name=3, anchor=center, inner sep=0}, "\xi\,"', curve={height=28pt}, shorten <=8pt, shorten >=8pt, Rightarrow, from=1, to=0]
	\arrow["\Xi'\cdot\Xi", shorten <=4pt, shorten >=4pt, triple, from=3, to=2]
\end{tikzcd}
\end{equation}
in the category $\dgCat_R^{[-1,0],\ps}\big[\CC,\DD\big](F,G)$ is defined by
\begin{flalign}
(\Xi'\cdot\Xi)_U \,:=\, \Xi'_U + \Xi_U\quad.
\end{flalign}
\end{subequations}
This composition is associative, unital
with respect to $\mathbf{0}$ and invertible\footnote{Every modification is thus an isomorphic modification and $\ast$ is \textit{pseudo}functorial. Furthermore, $\dgCat_R^{[-1,0],\ps}$ is actually a weak $(3,2)$-category.} with respect to the \textbf{reverse} modifications $\overleftarrow{\Xi}:\xi'\Rrightarrow\xi:F\Rightarrow G:\CC\rightarrow\DD$ defined by
\begin{equation}\label{interior inverse modifications}
\overleftarrow{\Xi}_U:=-\Xi_U\quad.
\end{equation}
\end{constr}
\begin{constr}\label{con:vertcomp mods}
The \textbf{vertical composite modification}
\begin{subequations}\label{eqn: def ver comp mod}
\begin{equation}
\begin{tikzcd}
	{\CC} && {\DD}
	\arrow[""{name=0, anchor=center, inner sep=0}, "G"{description}, from=1-1, to=1-3]
	\arrow[""{name=1, anchor=center, inner sep=0}, "H"', curve={height=36pt}, from=1-1, to=1-3]
	\arrow[""{name=2, anchor=center, inner sep=0}, "F", curve={height=-36pt}, from=1-1, to=1-3]
	\arrow[""{name=3, anchor=center, inner sep=0}, "{\,\xi'}", curve={height=-16pt}, shorten <=5pt, shorten >=5pt, Rightarrow, from=2, to=0]
	\arrow[""{name=4, anchor=center, inner sep=0}, "\xi\,"', curve={height=16pt}, shorten <=5pt, shorten >=5pt, Rightarrow, from=2, to=0]
	\arrow[""{name=5, anchor=center, inner sep=0}, "\theta\,"', curve={height=16pt}, shorten <=5pt, shorten >=5pt, Rightarrow, from=0, to=1]
	\arrow[""{name=6, anchor=center, inner sep=0}, "{\,\theta'}", curve={height=-16pt}, shorten <=5pt, shorten >=5pt, Rightarrow, from=0, to=1]
	\arrow["\Xi", shorten <=5pt, shorten >=5pt, triple, from=4, to=3]
	\arrow["\Theta", shorten <=5pt, shorten >=5pt, triple, from=5, to=6]
\end{tikzcd}
~~\stackrel{\circ}{\longmapsto}~~
\begin{tikzcd}
	{\CC} &&& {\DD}
	\arrow[""{name=1, anchor=center, inner sep=0}, "H"', curve={height=36pt}, from=1-1, to=1-4]
	\arrow[""{name=2, anchor=center, inner sep=0}, "F", curve={height=-36pt}, from=1-1, to=1-4]
	\arrow[""{name=3, anchor=center, inner sep=0}, "{\,\theta'   \xi'}", curve={height=-20pt}, shorten <=5pt, shorten >=5pt, Rightarrow, from=2, to=1]
	\arrow[""{name=4, anchor=center, inner sep=0}, "\theta  \xi\,"', curve={height=20pt}, shorten <=5pt, shorten >=5pt, Rightarrow, from=2, to=1]
	\arrow["\Theta  \Xi", shorten <=5pt, shorten >=5pt, triple, from=4, to=3]
\end{tikzcd}
\end{equation}
in the $2$-category $\dgCat_R^{[-1,0],\ps}\big[\CC,\DD\big]$ is defined by
\begin{flalign}\label{eq:vertcomp mods}
(\Theta \Xi)_U\,:=\, \Theta_U\xi'_U + \theta_U\Xi_U\quad.
\end{flalign}
\end{subequations} 
This composition is associative and unital
with respect to the particular identity modifications
$\ID_{\Id_F} : \Id_F\Rrightarrow \Id_F : F\Rightarrow F : \CC\to \DD$.
We denote the \textbf{post-whiskering} by
\begin{subequations}\label{subeq:whisker mod by pseudo}
\begin{equation}\label{eq:post-whisker mod by pseudo}
\theta\Xi:=\ID_\theta\,\Xi
\end{equation}
and the \textbf{pre-whiskering} by
\begin{equation}\label{eq:pre-whisker mod by pseudo}
\Theta\xi':=\Theta\,\ID_{\xi'}\quad.
\end{equation}
Using this notation, we can write 
\begin{equation}\label{eq:nice vercomp mods}
\Theta\Xi:=\Theta\xi'\cdot\theta\Xi\quad.
\end{equation}
\end{subequations}
\end{constr}
\begin{defi}\label{def:isomodification}
If $\xi,\xi':F\Rightarrow G$ are pseudonatural isomorphisms then a modification between them $\Xi:\xi\Rrightarrow\xi'$ is called an \textbf{isomodification}. An \textbf{automodification} $\Xi\in\mathrm{Mod}_F^\times$ is a modification between pseudonatural automorphisms.
\end{defi}
It is obvious from the above definition that isomodifications are closed under vertical composition and whiskering by a pseudonatural isomorphism.
\begin{rem}\label{rem:inverse isomodification}
If $\Xi:\xi\Rrightarrow\xi':F\Rightarrow G$ is an isomodification then
\begin{equation}\label{eq:inverse mod}
\Xi^{-1}:=\xi^{-1}\overleftarrow{\Xi}\xi'^{-1}:\xi^{-1}\Rrightarrow\xi'^{-1}:G\Rightarrow F
\end{equation}
clearly defines the \textbf{inverse} isomodification.
\end{rem}
\subsubsection{The closed symmetric strict monoidal structure}
We now recall the closed symmetric strict monoidal structure on the tricategory $\dgCat_R^{[-1,0],\ps}$. Again, we only recall those aspects of the structure that will be relevant throughout, e.g., we will not recall the formula for the monoidal composite modification. We will be brief in describing the closedness; we will describe how to $R$-linearly add pseudonatural transformations/modifications and differentiate modifications to get pseudonatural transformations but we will not go through the tedious effort of showing that the inner hom 3-functor is right adjoint to the monoidal composite 3-functor. 
\begin{constr}
The $\Ch_R^{[-1,0]}$-category $\CC\boxtimes\DD$ is defined by: 
\begin{enumerate}
\item[(i)] Objects are simply 2-tuples which we denote by juxtaposition, i.e., $UV\in\CC\boxtimes\DD$ where $U\in\CC$ and $V\in\DD$.
\item[(ii)] The hom 2-term complexes are given by the truncated tensor product, i.e., 
\begin{equation}
f\boxtimes g:=  \begin{cases}
     0\,, & |f|=|g|=-1\\
     f\otimes_R g\,, & \mathrm{otherwise}
    \end{cases}:UV\rightarrow U'V'\quad.
\end{equation}
\item[(iii)] Cochain composition is given by the middle-four exchange, i.e.,
\begin{equation}
(f'\boxtimes g')(f\boxtimes g)=(f'f)\boxtimes(g'g)
\end{equation}
has no sign for passing $f$ through $g'$ because if they both were of degree $-1$ then the truncation would kill this composite.
\item[(iv)] Cochain units are simply given by mapping $\oone\in\bbK\subseteq R$ to $1_U\boxtimes1_V:=1_U\otimes_\bbR1_V=1_{UV}$.
\end{enumerate}
\end{constr}
To a pair of $\Ch_R^{[-1,0]}$-functors, $F : \CC\to \CC'$ and $G : \DD\to\DD' $, we assign the 
usual $\Ch_R^{[-1,0]}$-functor, 
$F\boxtimes G : \CC\boxtimes \DD\to \CC'\boxtimes \DD'$,
from \cite[Section 1.4]{Kelly}.
\begin{constr}\label{con:monprod pseudos}
The \textbf{monoidal composite pseudonatural transformation}
\begin{subequations}
\begin{equation}
\begin{tikzcd}
	{\CC} && {\CC'}
	\arrow[""{name=0, anchor=center, inner sep=0}, "F", curve={height=-32pt}, from=1-1, to=1-3]
	\arrow[""{name=1, anchor=center, inner sep=0}, "{F'}"', curve={height=32pt}, from=1-1, to=1-3]
	\arrow["\xi\,"', shorten <=6pt, shorten >=6pt, Rightarrow, from=0, to=1]
\end{tikzcd}
~~
\begin{tikzcd}
	{\DD} && {\DD'}
	\arrow[""{name=0, anchor=center, inner sep=0}, "G", curve={height=-32pt}, from=1-1, to=1-3]
	\arrow[""{name=1, anchor=center, inner sep=0}, "{G'}"', curve={height=32pt}, from=1-1, to=1-3]
	\arrow["\theta\,"', shorten <=6pt, shorten >=6pt, Rightarrow, from=0, to=1]
\end{tikzcd}
~~\stackrel{\boxtimes}{\longmapsto}~~
\begin{tikzcd}
	{\CC\boxtimes \DD} && {\CC'\boxtimes \DD'}
	\arrow[""{name=0, anchor=center, inner sep=0}, "F\,\boxtimes\,G", curve={height=-32pt}, from=1-1, to=1-3]
	\arrow[""{name=1, anchor=center, inner sep=0}, "{F'\,\boxtimes\,G'}"', curve={height=32pt}, from=1-1, to=1-3]
	\arrow["\xi\,\boxtimes\,\theta\,"', shorten <=6pt, shorten >=6pt, Rightarrow, from=0, to=1]
\end{tikzcd}
\end{equation}
in the tricategory $\dgCat_R^{[-1,0],\ps}$ is defined by
\begin{alignat}{2}
(\xi\boxtimes \theta)_{UV}\,&:=\,\xi_U\boxtimes\theta_V\quad&&,\\
(\xi\boxtimes \theta)_{f\,\boxtimes \,g}\,&:=\,
\xi_f\boxtimes G'(g)\theta_V + \xi_{U'}F(f)\boxtimes \theta_g\quad&&.
\end{alignat}
\end{subequations}
\end{constr}

\begin{rem}\label{rem:boxtimes is a strictly associative 3-functor}
The monoidal composition $\boxtimes$ is a 3-functor meaning that the exchange law and preservation of identities is upheld for all three different levels of morphisms\footnote{Thus pseudonatural isomorphisms and isomodifications are preserved under monoidal composition $\boxtimes$.\label{fnote:pseudo preserved by moncomp}} in the tricategory $\dgCat_R^{[-1,0],\ps}$, this fact will become relevant, e.g., when discussing the pentagon axiom holding at order $h^2$ in Lemma \ref{lem:pent axiom upheld}. This monoidal composition is strictly associative because of the truncation and because we regard $\Ch^{[-1,0]}_R$ as a strict monoidal category. The strict monoidal unit is the $\Ch_R^{[-1,0]}$-category $\RR\,\in\,  \dgCat_R^{[-1,0],\ps}$ consisting of: a single object $\star$, morphisms
$\RR(\star,\star):= R$ and the cochain composition given by multiplication of scalars. The symmetric braiding $\tau$ has components $\tau_{\CC,\DD}:\CC\boxtimes\DD\cong\DD\boxtimes\CC$ as the obvious $\Ch^{[-1,0]}_R$-isomorphisms, which are clearly natural and satisfy the hexagon axiom
\begin{equation}\label{eq:tau satisfies hex}
\tau_{\CC,(\DD\,\boxtimes\,\EE)}=(\id_\DD\boxtimes\tau_{\CC,\EE})(\tau_{\CC,\DD}\boxtimes\id_\EE)\quad.
\end{equation}
\end{rem}
To state various definitions we will need the notion of $R$-linear addition $+$ of pseudonatural transformations and modifications. Note that in the latter case $R$-linear addition of modifications is actually distinct from lateral composition $\cdot$, e.g., scaling a modification by\footnote{That is, $\frac{1}{25}$ of the unit of $R$.} $\frac{1}{25}$ has no correlate as a lateral composite.  
\begin{constr}\label{con:add mods}
Suppose we have modifications $\Xi:\xi\Rrightarrow\xi':F\Rightarrow G$ and $\Theta:\theta\Rrightarrow\theta':F\Rightarrow G$ then, for $r,r'\in R$, we define
\begin{subequations}
\begin{equation}
r\Xi+r'\Theta:r\xi+r'\theta\Rrightarrow z\xi'+w\theta':F\Rightarrow G
\end{equation}
by
\begin{equation}
(r\xi+r'\theta)_U:=r\xi_U+r'\theta_U\,\,,\,\,
(r\xi+r'\theta)_f:=r\xi_f+r'\theta_f\qquad,\qquad(r\Xi+r'\Theta)_U:=r\Xi_U+r'\Theta_U\,.
\end{equation}
\end{subequations}
\end{constr}
\begin{rem}
Vertical composition $\circ$, of pseudonatural transformations \eqref{eqn: ver comp of pseudos} and modifications \eqref{eqn: def ver comp mod}, is $R$-bilinear. The addition of pseudonatural transformations is unital with respect to the zero pseudonatural transformation $0:F\Rightarrow G$ and every pseudonatural transformation $\xi:F\Rightarrow G$ admits an inverse with respect to this addition given by $-\xi:F\Rightarrow G$. The particular zero modification $\bm{0}:0\Rrightarrow0:F\Rightarrow G$ is the unit of addition for the modifications and the modification $\Xi:\xi\Rrightarrow\xi'$ admits $-\Xi:-\xi\Rrightarrow-\xi'$ as an inverse for this addition.
\end{rem}
\begin{defi}\label{def:coboundary of mod}
Given a modification $\Xi:\xi\Rrightarrow\xi':F\Rightarrow G$, we define the \textbf{coboundary}\footnote{Cf. the single datum of a modification \eqref{eqn: mod is homot between pseudos}.}
\begin{equation}\label{eq:coboundary pseudonat}
\partial\Xi:\equiv\xi-\xi':F\Rightarrow G\qquad.
\end{equation}
\end{defi}

\subsection{Braided strictly-unital monoidal cochain 2-categories}\label{subs:brastrunmon cats}
As already mentioned, the associator and braiding will be defined as pseudonatural isomorphisms but we already know that pseudonatural isomorphisms are closed under vertical and monoidal (see Footnote \ref{fnote:pseudo preserved by moncomp}) composition as well as whiskering (see Remark \ref{rem:whiskering pseudo by functor}). This then implies that the pentagonator \eqref{eq:Pentagonator} and hexagonators \eqref{eq:left hexagonator}/\eqref{eq:right hexagonator} will automatically be isomodifications and admit inverses as in Remark \ref{rem:inverse isomodification}.
 
\subsubsection{The associator and pentagonator}
The following definition is our specification of the definition of a monoidal bicategory as found in \cite[Definition C.1]{Schommer}.
\begin{defi}\label{def:strictly-unital monoidal Ch-cat}
A \textbf{strictly-unital monoidal $\Ch^{[-1,0]}_R$-category} $(\CC,\otimes,I,\alpha,\Pi)$ consists of a quintuple of data:
\begin{enumerate}
\item[(i)] $\CC\in\dgCat_R^{[-1,0],\ps}$.
\item[(ii)] A $\Ch^{[-1,0]}_R$-functor $\otimes:\CC\boxtimes\CC\rightarrow\CC$ called the \textbf{monoidal product} whose object map is denoted by bracketing the juxtaposition, i.e.,
\begin{subequations}
\begin{equation}
U\times V\overset{\boxtimes}{\longmapsto}UV\overset{\otimes}{\longmapsto}(UV)\quad.
\end{equation}
Again, the truncation of the homs (and the fact that $\otimes$ preserves degrees) means that the exchange never incurs a sign, i.e.,
\begin{equation}
(f'\otimes g')(f\otimes g)=(f'f)\otimes(g'g)\quad.
\end{equation}
\end{subequations}
\item[(iii)] A $\Ch^{[-1,0]}_R$-functor $\eta:\RR\rightarrow\CC$ called the \textbf{monoidal unit} given by $\eta(\star)=I\in\CC$.
\item[(iv)] A pseudonatural isomorphism $\alpha:\otimes\,(\otimes\,\boxtimes\,\id_\CC)\Rightarrow\otimes\,(\id_\CC\,\boxtimes\,\otimes):\CC\boxtimes\CC\boxtimes\CC\rightarrow\CC$ called the \textbf{associator}.
\item[(v)] An isomodification\footnote{Note that we implicitly use the fact that $\boxtimes$ is a 3-functor to write things such as $$\otimes\,(\id_\CC\boxtimes\otimes)\,(\otimes\boxtimes\id_\mathbf{C\,\boxtimes\,C})=\otimes\,(\otimes\boxtimes\otimes)=\otimes\,(\otimes\boxtimes\id_\CC)\,(\id_\mathbf{C\,\boxtimes\,C}\boxtimes\otimes)\quad.$$},
\[\begin{tikzcd}[column sep=0.1in,row sep=0.2in]
	&& {\otimes\,(\otimes\boxtimes\id_\CC)\,(\otimes\boxtimes\id_\mathbf{C\,\boxtimes\,C})} \\
	{\otimes\,(\otimes\boxtimes\id_\CC)\,(\id_\CC\boxtimes\otimes\boxtimes\id_\CC)} &&&& {\otimes\,(\id_\CC\boxtimes\otimes)\,(\otimes\boxtimes\id_\mathbf{C\,\boxtimes\,C})} \\
	\\
	{\otimes\,(\id_\CC\boxtimes\otimes)\,(\id_\CC\boxtimes\otimes\boxtimes\id_\CC)} &&&& {\otimes\,(\id_\CC\boxtimes\otimes)\,(\id_\mathbf{C\,\boxtimes\,C}\boxtimes\otimes)}
	\arrow["{\otimes\,(\alpha\,\boxtimes\,1)}"', Rightarrow, from=1-3, to=2-1]
	\arrow["{\alpha_{\otimes\,\boxtimes\,\id_{\mathbf{C\,\boxtimes\,C}}}}", Rightarrow, from=1-3, to=2-5]
	\arrow["{\alpha_{\id_\CC\,\boxtimes\,\otimes\,\boxtimes\,\id_\CC}}"', Rightarrow, from=2-1, to=4-1]
	\arrow["{\alpha_{\id_{\mathbf{C\,\boxtimes\,C}}\,\boxtimes\,\otimes}}", Rightarrow, from=2-5, to=4-5]
	\arrow["\Pi", shorten <=25pt, shorten >=15pt, Rightarrow, scaling nfold=3, from=4-1, to=2-5]
	\arrow["{\otimes\,(1\,\boxtimes\,\alpha)}"', Rightarrow, from=4-1, to=4-5]
\end{tikzcd}\]
called the \textbf{pentagonator},
\begin{subequations}
\begin{equation}\label{eq:Pentagonator}
\Pi:\left[\otimes(1\boxtimes\alpha)\right]\alpha_{\id_\CC\,\boxtimes\,\otimes\,\boxtimes\,\id_\CC}\left[\otimes(\alpha\boxtimes1)\right]\Rrightarrow\alpha_{\id_{\mathbf{C\boxtimes C}}\,\boxtimes\,\otimes}\,\alpha_{\otimes\,\boxtimes\,\id_{\mathbf{C\boxtimes C}}}\,.
\end{equation}
For every $UVWX\in\CC^{\,\boxtimes\,4}$, we must have
\begin{equation}\label{eq:pentagon as cochain}
(1_U\otimes\alpha_{VWX})\alpha_{U(VW)X}(\alpha_{UVW}\otimes1_X)-\alpha_{UV(WX)}\alpha_{(UV)WX}=\partial\Pi_{UVWX}\quad.
\end{equation}
For every $f\boxtimes g\boxtimes k\boxtimes x\in\CC^{\,\boxtimes\,4}[UVWX,U'V'W'X']^0$, we must have
\begin{flalign}
&\Pi_{U'V'W'X'}\left[\big((f\otimes g)\otimes k\big)\otimes x\right]+(f\otimes\alpha_{g\,\boxtimes\,k\,\boxtimes\,x})\alpha_{U(VW)X}(\alpha_{UVW}\otimes1_X)\nn\\&+(1_{U'}\otimes\alpha_{V'W'X'})\alpha_{f\,\boxtimes\,(g\,\otimes\,k)\,\boxtimes\,x}(\alpha_{UVW}\otimes1_X)+(1_{U'}\otimes\alpha_{V'W'X'})\alpha_{U'(V'W')X'}(\alpha_{f\,\boxtimes\,g\,\boxtimes\,k}\otimes x)
\nn\\&=\nn\\
&\alpha_{f\,\boxtimes\,g\,\boxtimes\,(k\,\otimes\,x)}\alpha_{(UV)WX}+\alpha_{U'V'(W'X')}\alpha_{(f\,\otimes\,g)\,\boxtimes\,k\,\boxtimes\,x}+\left[f\otimes\big(g\otimes(k\otimes x)\big)\right]\Pi_{UVWX}\label{eq:pentagon as homotopy}
\end{flalign}
\end{subequations}
\end{enumerate}
This data is required to satisfy the following four axioms:

1. \textbf{Associahedron}, i.e., the following two pasting diagrams in $\CC$ are equivalent:
\[\begin{tikzcd}
	{\big((((UV)W)X)Y\big)} &&& {\big(((UV)W)(XY)\big)} &&& {\big((UV)(W(XY))\big)} \\
	\\
	{\big(((U(VW))X)Y\big)} &&& {\big((U(VW))(XY)\big)} \\
	&&& {\big(U((VW)(XY))\big)} &&& {\big(U(V(W(XY)))\big)} \\
	{\big((U((VW)X))Y\big)} &&& {\big(U(((VW)X)Y)\big)} \\
	\\
	{\big((U(V(WX)))Y\big)} &&& {\big(U((V(WX)))Y)\big)} &&& {\big(U(V((WX)Y))\big)}
	\arrow["{\alpha_{((UV)W)XY}}", from=1-1, to=1-4]
	\arrow["{(\alpha_{UVW}\otimes1_X)\otimes1_Y}"', from=1-1, to=3-1]
	\arrow["{\alpha_{(UV)W(XY)}}", from=1-4, to=1-7]
	\arrow["{\alpha_{UVW}\otimes1_{(XY)}}", from=1-4, to=3-4]
	\arrow["{\alpha_{UV(W(XY))}}", from=1-7, to=4-7]
	\arrow["{-\alpha_{\alpha_{UVW}\,\boxtimes\,1_{XY}}}", Rightarrow, from=3-1, to=1-4]
	\arrow["{\alpha_{(U(VW))XY}}", from=3-1, to=3-4]
	\arrow["{\alpha_{U(VW)X}\otimes1_Y}"', from=3-1, to=5-1]
	\arrow["{\alpha_{U(VW)(XY)}}"', from=3-4, to=4-4]
	\arrow["{\Pi_{UVW(XY)}}"{description}, Rightarrow, from=4-4, to=1-7]
	\arrow["{1_U\otimes\alpha_{VW(XY)}}", from=4-4, to=4-7]
	\arrow["{\Pi_{U(VW)XY}}"{description}, curve={height=-18pt}, Rightarrow, from=5-1, to=3-4]
	\arrow["{\alpha_{U((VW)X)Y}}", from=5-1, to=5-4]
	\arrow["{(1_U\otimes\alpha_{VWX})\otimes1_Y}"', from=5-1, to=7-1]
	\arrow["{1_U\otimes\alpha_{(VW)XY}}", from=5-4, to=4-4]
	\arrow["{1_U\otimes(\alpha_{VWX}\otimes1_Y)}", from=5-4, to=7-4]
	\arrow["{-\alpha_{1_U\,\boxtimes\,\alpha_{VWX}\,\boxtimes\,1_Y}}"{description}, Rightarrow, from=7-1, to=5-4]
	\arrow["{\alpha_{U(V(WX))Y}}"', from=7-1, to=7-4]
	\arrow["{1_U\otimes\alpha_{V(WX)Y}}"', from=7-4, to=7-7]
	\arrow["{1_{1_U}\otimes\Pi_{VWXY}}"', Rightarrow, from=7-7, to=4-4]
	\arrow["{1_U\otimes(1_V\otimes\alpha_{WXY})}"', from=7-7, to=4-7]
\end{tikzcd}\]
and
\[\begin{tikzcd}
	{\big((((UV)W)X)Y\big)} &&& {\big(((UV)W)(XY)\big)} &&& {\big((UV)(W(XY))\big)} \\
	\\
	{\big(((U(VW))X)Y\big)} &&& {\big((UV)((WX)Y)\big)} \\
	&&&&&& {\big(U(V(W(XY)))\big)} \\
	{\big((U((VW)X))Y\big)} &&& {\big(((UV)(WX))Y\big)} \\
	\\
	{\big((U(V(WX)))Y\big)} &&& {\big(U((V(WX)))Y)\big)} &&& {\big(U(V((WX)Y))\big)}
	\arrow["{\alpha_{((UV)W)XY}}", from=1-1, to=1-4]
	\arrow["{(\alpha_{UVW}\otimes1_X)\otimes1_Y}"', from=1-1, to=3-1]
	\arrow[""{name=0, anchor=center, inner sep=0}, "{\alpha_{(UV)WX}\otimes1_Y}"', from=1-1, to=5-4]
	\arrow["{\alpha_{(UV)W(XY)}}", from=1-4, to=1-7]
	\arrow[""{name=1, anchor=center, inner sep=0}, "{\alpha_{UV(W(XY))}}", from=1-7, to=4-7]
	\arrow[""{name=2, anchor=center, inner sep=0}, "{\alpha_{U(VW)X}\otimes1_Y}"', from=3-1, to=5-1]
	\arrow["{1_{(UV)}\otimes\alpha_{WXY}}"{description}, from=3-4, to=1-7]
	\arrow[""{name=3, anchor=center, inner sep=0}, "{\alpha_{UV((WX)Y)}}"{description}, from=3-4, to=7-7]
	\arrow["{(1_U\otimes\alpha_{VWX})\otimes1_Y}"', from=5-1, to=7-1]
	\arrow["{\alpha_{(UV)(WX)Y}}"{description}, from=5-4, to=3-4]
	\arrow["{\alpha_{UV(WX)}\otimes1_Y}"{description}, from=5-4, to=7-1]
	\arrow["{\alpha_{U(V(WX))Y}}"', from=7-1, to=7-4]
	\arrow["{1_U\otimes\alpha_{V(WX)Y}}"', from=7-4, to=7-7]
	\arrow["{1_U\otimes(1_V\otimes\alpha_{WXY})}"', from=7-7, to=4-7]
	\arrow["{\Pi_{(UV)WXY}}"{description}, Rightarrow, from=0, to=1-4]
	\arrow["{\Pi_{UVWX}\otimes1_{1_Y}}"{description}, Rightarrow, from=2, to=5-4]
	\arrow["{\alpha_{1_{UV}\,\boxtimes\,\alpha_{WXY}}}"{description}, curve={height=-30pt}, shorten <=10pt, shorten >=10pt, Rightarrow, from=3, to=1]
	\arrow["{\Pi_{UV(WX)Y}}"{description}, shorten >=7pt, Rightarrow, from=7-4, to=3]
\end{tikzcd}\]
As an equality between homotopies, this reads,
\begin{align}
&(1_U\otimes\alpha_{VW(XY)})\Pi_{U(VW)XY}\left([\alpha_{UVW}\otimes1_X]\otimes1_Y\right)-\left(1_U\otimes\alpha_{VW(XY)}\right)\alpha_{U(VW)(XY)}\alpha_{\alpha_{UVW}\,\boxtimes\,1_{XY}}\nn\\&+(1_U\otimes\Pi_{VWXY})\alpha_{U((VW)X)Y}\left(\alpha_{U(VW)X}[\alpha_{UVW}\otimes1_X]\otimes1_Y\right)+\Pi_{UVW(XY)}\alpha_{((UV)W)XY}\nn\\&-\left(1_U\otimes[1_V\otimes\alpha_{WXY}]\alpha_{V(WX)Y}\right)\alpha_{1_U\,\boxtimes\,\alpha_{VWX}\,\boxtimes\,1_Y}\left(\alpha_{U(VW)X}[\alpha_{UVW}\otimes1_X]\otimes1_Y\right)\nn\\&=\label{eq:associahedron}\\&\left(1_U\otimes[1_V\otimes\alpha_{WXY}]\alpha_{V(WX)Y}\right)\alpha_{U(V(WX))Y}\left(\Pi_{UVWX}\otimes1_Y\right)+\alpha_{UV(W(XY))}\Pi_{(UV)WXY}\nn\\&+\left(1_U\otimes[1_V\otimes\alpha_{WXY}]\right)\Pi_{UV(WX)Y}\left(\alpha_{(UV)WX}\otimes1_Y\right)+\alpha_{1_{UV}\,\boxtimes\,\alpha_{WXY}}\alpha_{(UV)(WX)Y}(\alpha_{(UV)WX}\otimes1_Y)\nn
\end{align}
2. \textbf{Unitality}, i.e., the monoidal product $\otimes$ is a retraction of both the left unit $\eta\boxtimes\id_\CC$ and the right unit $\id_\CC\boxtimes\eta$,
\begin{subequations}
\begin{equation}\label{eq:Unitality of monoidal product}
\begin{tikzcd}
	{\CC\boxtimes\CC} \\
	\\
	{\RR\boxtimes\CC=\CC=\CC\boxtimes\RR}
	\arrow["\otimes"{description}, from=1-1, to=3-1]
	\arrow["{\id_\CC\,\boxtimes\,\eta}"', shift right=5, curve={height=18pt}, from=3-1, to=1-1]
	\arrow["{\eta\,\boxtimes\,\id_\CC}", shift left=5, curve={height=-18pt}, from=3-1, to=1-1]
\end{tikzcd}\qquad,\qquad\otimes\,(\eta\boxtimes\id_\CC)=\id_\CC=\otimes\,(\id_\CC\boxtimes \eta)\quad.
\end{equation}
In terms of objects and morphisms, this demands:
\begin{equation}\label{eq:unitality in terms of object/morphism}
(I U)=U=(UI)\qquad,\qquad1_I\otimes f=f=f\otimes1_I\quad.
\end{equation}
\end{subequations}
3. \textbf{Triangle}; the associator $\alpha$ is a retraction of the middle insertion of the unit $\Id_{\id_\CC\,\boxtimes\,\eta\,\boxtimes\,\id_\CC}$,
\begin{subequations}
\begin{equation}
\alpha_{\id_\CC\,\boxtimes\,\eta\,\boxtimes\,\id_\CC}=\Id_\otimes:\otimes\,(\otimes\boxtimes\id_\CC)\,(\id_\CC\boxtimes\eta\boxtimes\id_\CC)=\otimes:\CC\boxtimes\CC\rightarrow\CC\quad.
\end{equation}
In terms of objects and morphisms, this demands:
\begin{equation}\label{eq:triangle as cochain and homotopy}
\alpha_{UI V}=1_{(UV)}\qquad,\qquad\alpha_{f\,\boxtimes\,1_I\,\boxtimes\,g}=0\quad.
\end{equation}
\end{subequations}
4. \textbf{Prism}; the pentagonator $\Pi$ is a retraction of both middle insertions of the unit, 
\begin{equation}\label{eq:Prism}
\Pi_{UI WX}=0=\Pi_{UVI X}\quad.
\end{equation}
If the pentagon axiom is upheld and the pentagonator $\Pi$ is trivial then we say $(\CC,\otimes,I,\alpha)$ is a \textbf{pentagonal strictly-unital monoidal $\Ch^{[-1,0]}_R$-category}.
If, in addition, the monoidal product is associative\footnote{In terms of objects and morphisms, this reads as:
\begin{equation}
((UV)W)=:(UVW):=(U(VW))\qquad,\qquad(f\otimes g)\otimes k=:f\otimes g\otimes k:=f\otimes(g\otimes k)\quad.
\end{equation}} $\otimes(\otimes\boxtimes\id_\CC)=\otimes(\id_\CC\boxtimes\otimes)$ and the associator is trivial $\alpha=\Id_{\otimes\,(\otimes\,\boxtimes\,\id_\CC)}$ then we say $(\CC,\otimes,I)$ is a \textbf{strict monoidal $\Ch^{[-1,0]}_R$-category}.
\end{defi}
\begin{rem}
In the familiar context of an ordinary monoidal 1-category $(\CC,\otimes,I,\alpha,\lambda,\rho)$, the triangle axiom
\begin{subequations}
\begin{equation}
\begin{tikzcd}
	{(U\otimes I)\otimes W} && {U\otimes(I\otimes W)} \\
	& {U\otimes W}
	\arrow["{\alpha_{UI W}}", from=1-1, to=1-3]
	\arrow["{\rho_U\,\otimes\,1_W}"', from=1-1, to=2-2]
	\arrow["{1_U\,\otimes\,\lambda_W}", from=1-3, to=2-2]
\end{tikzcd}
\end{equation}
implies both the ``left" triangle identity
\begin{equation}
\begin{tikzcd}
	{(I\otimes V)\otimes W} && {I\otimes(V\otimes W)} \\
	& {V\otimes W}
	\arrow["{\alpha_{I VW}}", from=1-1, to=1-3]
	\arrow["{\lambda_V\,\otimes\,1_W}"', from=1-1, to=2-2]
	\arrow["{\lambda_{V\otimes W}}", from=1-3, to=2-2]
\end{tikzcd}
\end{equation}
and the ``right" triangle identity
\begin{equation}
\begin{tikzcd}
	{(U\otimes V)\otimes I} && {U\otimes(V\otimes I)} \\
	& {U\otimes V}
	\arrow["{\alpha_{UVI}}", from=1-1, to=1-3]
	\arrow["{\rho_{U\otimes V}}"', from=1-1, to=2-2]
	\arrow["{1_U\,\otimes\,\rho_V}", from=1-3, to=2-2]
\end{tikzcd}
\end{equation}
\end{subequations}
The proof of this usually relies on the naturality of the associator; for example, see the diagrammatic proof of \cite[Lemma XI.2.2]{Kassel}. In the case of Definition \ref{def:strictly-unital monoidal Ch-cat}, the triangle axiom \eqref{eq:triangle as cochain and homotopy} still provides these left and right triangle identities despite the fact that the associator $\alpha$ is a \textit{pseudo}natural isomorphism. This is in large part because of strict unitality \eqref{eq:unitality in terms of object/morphism} and the fact that the pentagon axiom holds if one of the middle tensorands is the monoidal unit \eqref{eq:Prism}. Let us show this for the left triangle identity by plugging in $U=V=I$ into \eqref{eq:pentagon as cochain},
\begin{subequations}
\begin{equation}
\alpha_{II(WX)}\alpha_{(II)WX}=(1_I\otimes\alpha_{I WX})\alpha_{I(I W)X}(\alpha_{II W}\otimes1_X)\quad\implies\quad\alpha_{I WX}=1_{(WX)}\quad,
\end{equation}
and $f=g=1_I$ into \eqref{eq:pentagon as homotopy},
\begin{equation}
\alpha_{(1_I\,\otimes\,1_I)\,\boxtimes\,k\,\boxtimes\,x}=1_I\otimes\alpha_{1_I\,\boxtimes\,k\,\boxtimes\,x}+\alpha_{1_I\,\boxtimes\,(1_I\,\otimes\,k)\,\boxtimes\,x}\quad\implies\quad\alpha_{1_I\,\boxtimes\,k\,\boxtimes\,x}=0\quad.
\end{equation}
\end{subequations}
By plugging $W=X=I$ into \eqref{eq:pentagon as cochain} and $k=x=1_I$ into \eqref{eq:pentagon as homotopy} one shows the right triangle identity also holds hence the three 2-unitors in \cite[Definition C.1]{Schommer} are all trivial. 

Furthermore, we can plug $U=V=I$ into \eqref{eq:associahedron}. We use the fact that all three aforementioned triangle identities hold together with \eqref{eq:Prism} to get
\begin{equation}
2\Pi_{I WXY}=\Pi_{I WXY}\qquad\implies\qquad\Pi_{I WXY}=0\quad.
\end{equation}
Plugging $X=Y=I$ into \eqref{eq:associahedron} would likewise give use that $\Pi_{UVWI}=0$. 

In summary, the prismatic axioms \cite[Axiom SM2 of Definition C.1]{Schommer} indeed hold because: both the left and right unitor are identities, all three 2-unitors are trivial, the pentagonator is trivial if one of its subscripts is $I$, the associator homotopy components of the unitors (which are unit 1-cells) vanishes as in \eqref{eqn: homotopy kills unit}.
\end{rem}
\subsubsection{The braiding and hexagonators}
The following definition is our specification of the definition of a braiding on a monoidal bicategory as found in \cite[Definition C.2]{Schommer}.
\begin{defi}\label{def:bra str un mon cat}
A \textbf{braided strictly-unital monoidal $\Ch_R^{[-1,0]}$-category} consists of a strictly-unital monoidal $\Ch_R^{[-1,0]}$-category $(\CC,\otimes,I,\alpha,\Pi)$ together with a triple $(\sigma,\H^L,\H^R)$ of data where:
\begin{enumerate}
\item[(i)] The \textbf{braiding} $\sigma:\otimes\Rightarrow\otimes\,\tau_{\CC,\CC}:\CC\boxtimes\CC\rightarrow\CC$ is a pseudonatural isomorphism.
\item[(ii)] The \textbf{left hexagonator}\footnote{Note that we implicitly use the fact that $\tau$ is a natural transformation to write things such as $$\tau_{\CC,\CC}\,(\id_\CC\boxtimes\otimes)=(\otimes\boxtimes\id_\CC)\tau_{\CC,(\CC\,\boxtimes\,\CC)}$$ and we also use the fact that $\tau$ is a symmetric braiding hence satisfies the hexagon axiom \eqref{eq:tau satisfies hex}.},
\[\begin{tikzcd}[column sep=0.1in,row sep=0.2in]
	& {\otimes(\id_{\CC}\boxtimes\otimes)} && {\otimes(\otimes\boxtimes\id_{\CC})\tau_{\mathbf{C,(C\,\boxtimes\,C)}}} \\
	\\
	{\otimes(\otimes\boxtimes\id_{\CC})} &&&& {\otimes(\id_{\CC}\boxtimes\otimes)\tau_{\mathbf{C,(C\,\boxtimes\,C)}}} \\
	\\
	& {\otimes(\otimes\tau_{\mathbf{C,C}}\boxtimes\id_{\CC})} && {\otimes(\id_{\CC}\boxtimes\otimes)(\tau_{\mathbf{C,C}}\boxtimes\id_{\CC})}
	\arrow[""{name=0, anchor=center, inner sep=0}, "{\sigma_{\id_{\CC}\,\boxtimes\,\otimes}}", Rightarrow, from=1-2, to=1-4]
	\arrow["{\alpha_{\tau_{\mathbf{C,(C\,\boxtimes\,C)}}}}", Rightarrow, from=1-4, to=3-5]
	\arrow["\alpha", Rightarrow, from=3-1, to=1-2]
	\arrow["{\otimes(\sigma\,\boxtimes\,1)}"', Rightarrow, from=3-1, to=5-2]
	\arrow[""{name=1, anchor=center, inner sep=0}, "{\alpha_{\tau_{\mathbf{C,C}}\,\boxtimes\,\id_{\CC}}}", curve={height=-15pt}, Rightarrow, from=5-2, to=5-4]
	\arrow["{\otimes(1\,\boxtimes\,\sigma)_{\tau_{\mathbf{C,C}}\,\boxtimes\,\id_{\CC}}}"', Rightarrow, from=5-4, to=3-5]
	\arrow["{\H^L}", shorten <=20pt, shorten >=20pt, Rightarrow, scaling nfold=3, from=0, to=1]
\end{tikzcd}\]
\begin{subequations}
is an isomodification,
\begin{equation}\label{eq:left hexagonator}
\H^L:\alpha_{\tau_{\mathbf{C,(C\boxtimes C)}}}\,\sigma_{\id_{\CC}\,\boxtimes\,\otimes}\,\alpha\Rrightarrow\left(\big[\otimes(1\boxtimes\sigma)\big]\alpha\right)_{\tau_{\mathbf{C,C}}\,\boxtimes\,\id_{\CC}}\left[\otimes(\sigma\boxtimes1)\right]\,.
\end{equation}
The following diagrams in $\CC$ will appear throughout the enumeration of the axioms:
\begin{equation}\label{eq:left hexagonator as 2-cell}
\begin{tikzcd}
	& {(U(VW))} && {((VW)U)} \\
	{((UV)W)} &&&& {(V(WU))} \\
	& {((VU)W)} && {(V(UW))}
	\arrow[""{name=0, anchor=center, inner sep=0}, "{\sigma_{U(VW)}}", from=1-2, to=1-4]
	\arrow["{\alpha_{VWU}}", from=1-4, to=2-5]
	\arrow["{\alpha_{UVW}}", from=2-1, to=1-2]
	\arrow["{\sigma_{UV}\,\otimes\,1_W}"', from=2-1, to=3-2]
	\arrow[""{name=1, anchor=center, inner sep=0}, "{\alpha_{VUW}}"', from=3-2, to=3-4]
	\arrow["{1_V\,\otimes\,\sigma_{UW}}"', from=3-4, to=2-5]
	\arrow["{\H^L_{UVW}}", shorten <=13pt, shorten >=13pt, Rightarrow, from=0, to=1]
\end{tikzcd}
\end{equation}
\begin{equation}
\begin{tikzcd}
	& {(U(VW))} && {((VW)U)} \\
	{((UV)W)} &&&& {(V(WU))} \\
	& {((VU)W)} && {(V(UW))}
	\arrow["{\alpha^{-1}_{UVW}}"', from=1-2, to=2-1]
	\arrow[""{name=0, anchor=center, inner sep=0}, "{\sigma^{-1}_{U(VW)}}"', from=1-4, to=1-2]
	\arrow["{\alpha^{-1}_{VWU}}"', from=2-5, to=1-4]
	\arrow["{1_V\,\otimes\,\sigma^{-1}_{UW}}", from=2-5, to=3-4]
	\arrow["{\sigma^{-1}_{UV}\,\otimes\,1_W}", from=3-2, to=2-1]
	\arrow[""{name=1, anchor=center, inner sep=0}, "{\alpha^{-1}_{VUW}}", from=3-4, to=3-2]
	\arrow["{(\H^L)^{-1}_{UVW}}", shorten <=13pt, shorten >=13pt, Rightarrow, from=0, to=1]
\end{tikzcd}
\end{equation}
\end{subequations}
\item[(iii)] The \textbf{right hexagonator},
\[\begin{tikzcd}[column sep=0.1in,row sep=0.2in]
	& {\otimes(\otimes\boxtimes\id_{\CC})} && {\otimes(\id_\CC\boxtimes\otimes)\tau_{\mathbf{(C\,\boxtimes\,C),C}}} \\
	\\
	{\otimes(\id_\CC\boxtimes\otimes)} &&&& {\otimes(\otimes\boxtimes\id_{\CC})\tau_{\mathbf{(C\,\boxtimes\,C),C}}} \\
	\\
	& {\otimes(\id_\CC\boxtimes\otimes\tau_{\mathbf{C,C}})} && {\otimes(\otimes\boxtimes\id_\CC)(\id_\CC\boxtimes\tau_{\mathbf{C,C}})}
	\arrow[""{name=0, anchor=center, inner sep=0}, "{\sigma_{\otimes\,\boxtimes\,\id_\CC}}", Rightarrow, from=1-2, to=1-4]
	\arrow["{\alpha^{-1}_{\tau_{\mathbf{(C\,\boxtimes\,C),C}}}}", Rightarrow, from=1-4, to=3-5]
	\arrow["\alpha^{-1}", Rightarrow, from=3-1, to=1-2]
	\arrow["{\otimes(1\,\boxtimes\,\sigma)}"', Rightarrow, from=3-1, to=5-2]
	\arrow[""{name=1, anchor=center, inner sep=0}, "{\alpha^{-1}_{\id_\CC\,\boxtimes\,\tau_{\mathbf{C,C}}}}", curve={height=-14pt}, Rightarrow, from=5-2, to=5-4]
	\arrow["{\otimes(\sigma\,\boxtimes\,1)_{\id_\CC\,\boxtimes\,\tau_{\mathbf{C,C}}}}"', Rightarrow, from=5-4, to=3-5]
	\arrow["{\H^R}", shorten <=20pt, shorten >=20pt, Rightarrow, scaling nfold=3, from=0, to=1]
\end{tikzcd}\]
\begin{subequations}
is an isomodification,
\begin{equation}\label{eq:right hexagonator}
\H^R:\alpha^{-1}_{\tau_{\mathbf{(C\,\boxtimes\,C),C}}}\,\sigma_{\otimes\,\boxtimes\,\id_\CC}\,\alpha^{-1}\Rrightarrow\left(\big[\otimes(\sigma\boxtimes1)\big]\alpha^{-1}\right)_{\id_\CC\,\boxtimes\,\tau_{\mathbf{C,C}}}\left[\otimes(1\boxtimes\sigma)\right]\,.
\end{equation}
The following diagrams in $\CC$ will also appear throughout the enumeration of the axioms:
\begin{equation}\label{eq:right hexagonator as 2-cell}
\begin{tikzcd}
	& {((UV)W)} && {(W(UV))} \\
	{(U(VW))} &&&& {((WU)V)} \\
	& {(U(WV))} && {((UW)V)}
	\arrow[""{name=0, anchor=center, inner sep=0}, "{\sigma_{(UV)W}}", from=1-2, to=1-4]
	\arrow["{\alpha^{-1}_{WUV}}", from=1-4, to=2-5]
	\arrow["{\alpha^{-1}_{UVW}}", from=2-1, to=1-2]
	\arrow["{1_U\otimes\sigma_{VW}}"', from=2-1, to=3-2]
	\arrow[""{name=1, anchor=center, inner sep=0}, "{\alpha^{-1}_{UWV}}"', from=3-2, to=3-4]
	\arrow["{\sigma_{UW}\otimes1_V}"', from=3-4, to=2-5]
	\arrow["{\H^R_{UVW}}", shorten <=13pt, shorten >=13pt, Rightarrow, from=0, to=1]
\end{tikzcd}
\end{equation}
\begin{equation}
\begin{tikzcd}
	& {((UV)W)} && {(W(UV))} \\
	{(U(VW))} &&&& {((WU)V)} \\
	& {(U(WV))} && {((UW)V)}
	\arrow["{\alpha_{UVW}}"', from=1-2, to=2-1]
	\arrow[""{name=0, anchor=center, inner sep=0}, "{\sigma^{-1}_{(UV)W}}"', from=1-4, to=1-2]
	\arrow["{\alpha_{WUV}}"', from=2-5, to=1-4]
	\arrow["{\sigma^{-1}_{UW}\otimes1_V}", from=2-5, to=3-4]
	\arrow["{1_U\otimes\sigma^{-1}_{VW}}", from=3-2, to=2-1]
	\arrow[""{name=1, anchor=center, inner sep=0}, "{\alpha_{UWV}}", from=3-4, to=3-2]
	\arrow["{(\H^R)^{-1}_{UVW}}", shorten <=13pt, shorten >=13pt, Rightarrow, from=0, to=1]
\end{tikzcd}
\end{equation}
\end{subequations}
\end{enumerate}
This data is required to satisfy the following four axioms:

1. \textbf{Left tetrahedron}, i.e., the following two pasting diagrams in $\CC$ are equivalent:
\[\begin{tikzcd}[column sep=0.15in,row sep=0.3in]
	&& {\big((U(VW))X\big)} \\
	& {\big(((UV)W)X\big)} && {\big(U((VW)X)\big)} \\
	{\big(((VU)W)X\big)} & {\big((UV)(WX)\big)} && {\big(U(V(WX))\big)} & {\big(((VW)X)U\big)} \\
	\\
	{\big((V(UW))X\big)} & {\big((VU)(WX)\big)} && {\big((V(WX))U\big)} & {\big((VW)(XU)\big)} \\
	\\
	{\big(V((UW)X)\big)} & {\big(V(U(WX))\big)} && {\big(V((WX)U)\big)} & {\big(V(W(XU))\big)}
	\arrow["{\alpha_{U(VW)X}}", from=1-3, to=2-4]
	\arrow["{\alpha_{UVW}\otimes1_X}", from=2-2, to=1-3]
	\arrow["{(\sigma_{UV}\otimes1_W)\otimes1_X}"', from=2-2, to=3-1]
	\arrow["{\alpha_{(UV)WX}}", from=2-2, to=3-2]
	\arrow["{1_U\otimes\alpha_{VWX}}"', from=2-4, to=3-4]
	\arrow["{\sigma_{U((VW)X)}}", from=2-4, to=3-5]
	\arrow["{\alpha_{VUW}\otimes1_X}"', from=3-1, to=5-1]
	\arrow[""{name=0, anchor=center, inner sep=0}, "{\alpha_{(VU)WX}}"{description}, from=3-1, to=5-2]
	\arrow[""{name=1, anchor=center, inner sep=0}, "{\alpha_{UV(WX)}}"', from=3-2, to=3-4]
	\arrow[""{name=2, anchor=center, inner sep=0}, "{\sigma_{UV}\otimes1_{(WX)}}", from=3-2, to=5-2]
	\arrow[""{name=3, anchor=center, inner sep=0}, "{\sigma_{U(V(WX))}}"', from=3-4, to=5-4]
	\arrow[""{name=4, anchor=center, inner sep=0}, "{\alpha_{VWX}\otimes1_U}"{description}, from=3-5, to=5-4]
	\arrow["{\alpha_{(VW)XU}}", from=3-5, to=5-5]
	\arrow[""{name=5, anchor=center, inner sep=0}, "{\alpha_{V(UW)X}}"', from=5-1, to=7-1]
	\arrow[""{name=6, anchor=center, inner sep=0}, "{\alpha_{VU(WX)}}", from=5-2, to=7-2]
	\arrow[""{name=7, anchor=center, inner sep=0}, "{\alpha_{V(WX)U}}"', from=5-4, to=7-4]
	\arrow[""{name=8, anchor=center, inner sep=0}, "{\alpha_{VW(XU)}}", from=5-5, to=7-5]
	\arrow["{1_V\otimes\alpha_{UWX}}", curve={height=-12pt}, from=7-1, to=7-2]
	\arrow[""{name=9, anchor=center, inner sep=0}, "{1_V\otimes\sigma_{U(WX)}}"', from=7-2, to=7-4]
	\arrow["{1_V\otimes\alpha_{WXU}}", curve={height=-12pt}, from=7-4, to=7-5]
	\arrow["{\Pi_{UVWX}}"', shorten <=4pt, shorten >=4pt, Rightarrow, from=1-3, to=1]
	\arrow["{\alpha_{\sigma_{UV}\,\boxtimes \,1_W\,\boxtimes \,1_X}}"'{pos=0.7}, shift right=3, Rightarrow, from=2, to=0]
	\arrow["{\H^L_{UV(WX)}}"', shorten <=17pt, shorten >=17pt, Rightarrow, from=1, to=9]
	\arrow["{\sigma_{1_U\,\boxtimes \,\alpha_{VWX}}}"'{pos=0.3}, shift right=3, Rightarrow, from=4, to=3]
	\arrow["{\overleftarrow{\Pi}_{VUWX}}"', shift right=3, Rightarrow, from=6, to=5]
	\arrow["{\overleftarrow{\Pi}_{VWXU}}"', shift right=5, Rightarrow, from=8, to=7]
\end{tikzcd}\]
and
\[\begin{tikzcd}[column sep=0.15in,row sep=0.3in]
	&& {\big((U(VW))X\big)} \\
	& {\big(((UV)W)X\big)} && {\big(U((VW)X)\big)} \\
	{\big(((VU)W)X\big)} && {\big(((VW)U)X\big)} && {\big(((VW)X)U\big)} \\
	& {\big((V(WU))X\big)} && {\big((VW)(UX)\big)} \\
	{\big((V(UW))X\big)} &&&& {\big((VW)(XU)\big)} \\
	& {\big(V((WU)X)\big)} && {\big(V(W(UX))\big)} \\
	{\big(V((UW)X)\big)} & {\big(V(U(WX))\big)} && {\big(V((WX)U)\big)} & {\big(V(W(XU))\big)}
	\arrow["{\alpha_{U(VW)X}}", from=1-3, to=2-4]
	\arrow[""{name=0, anchor=center, inner sep=0}, "{\sigma_{U(VW)}\otimes1_X}"{description}, from=1-3, to=3-3]
	\arrow["{\alpha_{UVW}\otimes1_X}", from=2-2, to=1-3]
	\arrow["{(\sigma_{UV}\otimes1_W)\otimes1_X}"', from=2-2, to=3-1]
	\arrow["{\sigma_{U((VW)X)}}", from=2-4, to=3-5]
	\arrow["{\H^L_{U(VW)X}}"{description}, Rightarrow, from=2-4, to=4-4]
	\arrow[""{name=1, anchor=center, inner sep=0}, "{\alpha_{VUW}\otimes1_X}"', from=3-1, to=5-1]
	\arrow["{\alpha_{VWU}\otimes1_X}"{description}, from=3-3, to=4-2]
	\arrow["{\alpha_{(VW)UX}}"{description}, from=3-3, to=4-4]
	\arrow["{\alpha_{(VW)XU}}", from=3-5, to=5-5]
	\arrow[""{name=2, anchor=center, inner sep=0}, "{\alpha_{V(WU)X}}"{description}, from=4-2, to=6-2]
	\arrow[""{name=3, anchor=center, inner sep=0}, "{1_{(VW)}\otimes\,\sigma_{UX}}"{description}, from=4-4, to=5-5]
	\arrow[""{name=4, anchor=center, inner sep=0}, "{\alpha_{VW(UX)}}"{description}, from=4-4, to=6-4]
	\arrow[""{name=5, anchor=center, inner sep=0}, "{(1_V\otimes\,\sigma_{UW})\otimes1_X}"{description}, from=5-1, to=4-2]
	\arrow["{\alpha_{V(UW)X}}"', from=5-1, to=7-1]
	\arrow["{\alpha_{VW(XU)}}", from=5-5, to=7-5]
	\arrow[""{name=6, anchor=center, inner sep=0}, "{1_V\otimes\,\alpha_{WUX}}", from=6-2, to=6-4]
	\arrow[""{name=7, anchor=center, inner sep=0}, "{1_V\otimes(1_W\otimes\,\sigma_{UX})}"{description}, from=6-4, to=7-5]
	\arrow[""{name=8, anchor=center, inner sep=0}, "{1_V\otimes(\sigma_{UW}\otimes1_X)}"{description}, from=7-1, to=6-2]
	\arrow["{1_V\otimes\,\alpha_{UWX}}"', curve={height=12pt}, from=7-1, to=7-2]
	\arrow[""{name=9, anchor=center, inner sep=0}, "{1_V\otimes\,\sigma_{U(WX)}}"', from=7-2, to=7-4]
	\arrow["{1_V\otimes\,\alpha_{WXU}}"', curve={height=12pt}, from=7-4, to=7-5]
	\arrow["{\H^L_{UVW}\otimes1_{1_X}}"{description}, shorten <=25pt, shorten >=25pt, Rightarrow, from=0, to=1]
	\arrow["{\overleftarrow{\Pi}_{VWUX}}"', shorten <=38pt, shorten >=38pt, Rightarrow, from=4, to=2]
	\arrow["{-\alpha_{1_V\,\boxtimes \,1_W\,\boxtimes \,\sigma_{UX}}}"{description}, shorten <=4pt, shorten >=4pt, Rightarrow, from=3, to=7]
	\arrow["{-\alpha_{1_V\,\boxtimes \,\sigma_{UW}\,\boxtimes \,1_X}}"{description}, shorten <=4pt, shorten >=4pt, Rightarrow, from=5, to=8]
	\arrow["{1_{1_V}\otimes\,\overleftarrow{\H^L}_{UWX}}"{description}, shorten <=2pt, shorten >=2pt, Rightarrow, from=6, to=9]
\end{tikzcd}\]
which reads explicitly as,
\begin{align}
&(1_V\otimes\alpha_{WXU})\alpha_{V(WX)U}\sigma_{1_U\,\boxtimes \,\alpha_{VWX}}\alpha_{U(VW)X}(\alpha_{UVW}\otimes1_X)-\Pi_{VWXU}\sigma_{U((VW)X)}\alpha_{U(VW)X}(\alpha_{UVW}\otimes1_X)\nn\\
&+(1_V\otimes\alpha_{WXU})\H^L_{UV(WX)}\alpha_{(UV)WX}+(1_V\otimes\alpha_{WXU}\sigma_{U(WX)})\alpha_{VU(WX)}\alpha_{\sigma_{UV}\,\boxtimes \,1_W\,\boxtimes \,1_X}\nn\\&+(1_V\otimes\alpha_{WXU})\alpha_{V(WX)U}\sigma_{U(V(WX))}\Pi_{UVWX}-(1_V\otimes\alpha_{WXU}\sigma_{U(WX)})\Pi_{VUWX}\big((\sigma_{UV}\otimes1_W)\otimes1_X\big)\nn\\&=\label{eq:left tetrahedron}\\
&\alpha_{VW(XU)}\H^L_{U(VW)X}(\alpha_{UVW}\otimes1_X)-\left(1_V\otimes(1_W\otimes\sigma_{UX})\alpha_{WUX}\right)\alpha_{1_V\,\boxtimes \,\sigma_{UW}\,\boxtimes \,1_X}\left(\alpha_{VUW}(\sigma_{UV}\otimes1_W)\otimes1_X\right)\nn\\&-\big(1_V\otimes(1_W\otimes\sigma_{UX})\big)\Pi_{VWUX}\big(\sigma_{U(VW)}\alpha_{UVW}\otimes1_X\big)-\alpha_{1_V\,\boxtimes \,1_W\,\boxtimes \,\sigma_{UX}}\alpha_{(VW)UX}(\sigma_{U(VW)}\alpha_{UVW}\otimes1_X)\nn\\&+\left(1_V\otimes(1_W\otimes\sigma_{UX})\alpha_{WUX}\right)\alpha_{V(WU)X}(\H^L_{UVW}\otimes1_X)-\left(1_V\otimes\H^L_{UWX}\right)\alpha_{V(UW)X}\left(\alpha_{VUW}(\sigma_{UV}\otimes1_W)\otimes1_X\right)\nn
\end{align}

2. \textbf{Right tetrahedron}, i.e., the following two pasting diagrams in $\CC$ are equivalent:
\[\begin{tikzcd}[column sep=0.15in,row sep=0.3in]
	&& {\big(U((VW)X)\big)} \\
	& {\big(U(V(WX))\big)} && {\big((U(VW))X\big)} \\
	{\big(U(V(XW))\big)} & {\big((UV)(WX)\big)} && {\big(((UV)W)X\big)} & {\big(X(U(VW))\big)} \\
	\\
	{\big(U((VX)W)\big)} & {\big((UV)(XW)\big)} && {\big(X((UV)W)\big)} & {\big((XU)(VW)\big)} \\
	\\
	{\big((U(VX))W\big)} & {\big(((UV)X)W\big)} && {\big((X(UV))W\big)} & {\big(((XU)V)W\big)}
	\arrow["{\alpha^{-1}_{U(VW)X}}", from=1-3, to=2-4]
	\arrow["{1_U\otimes\alpha^{-1}_{VWX}}", from=2-2, to=1-3]
	\arrow["{1_U\otimes(1_V\otimes\sigma_{WX})}"', from=2-2, to=3-1]
	\arrow["{\alpha^{-1}_{UV(WX)}}", from=2-2, to=3-2]
	\arrow["{\alpha^{-1}_{UVW}\otimes1_X}"', from=2-4, to=3-4]
	\arrow["{\sigma_{(U(VW))X}}", from=2-4, to=3-5]
	\arrow["{1_U\otimes\alpha^{-1}_{VXW}}"', from=3-1, to=5-1]
	\arrow[""{name=0, anchor=center, inner sep=0}, "{\alpha^{-1}_{UV(XW)}}"{description}, from=3-1, to=5-2]
	\arrow[""{name=1, anchor=center, inner sep=0}, "{\alpha^{-1}_{(UV)WX}}"', from=3-2, to=3-4]
	\arrow[""{name=2, anchor=center, inner sep=0}, "{1_{(UV)}\otimes\sigma_{WX}}", from=3-2, to=5-2]
	\arrow[""{name=3, anchor=center, inner sep=0}, "{\sigma_{((UV)W)X}}"', from=3-4, to=5-4]
	\arrow[""{name=4, anchor=center, inner sep=0}, "{1_X\otimes\alpha^{-1}_{UVW}}"{description}, from=3-5, to=5-4]
	\arrow["{\alpha^{-1}_{XU(VW)}}", from=3-5, to=5-5]
	\arrow[""{name=5, anchor=center, inner sep=0}, "{\alpha^{-1}_{U(VX)W}}"', from=5-1, to=7-1]
	\arrow[""{name=6, anchor=center, inner sep=0}, "{\alpha^{-1}_{(UV)XW}}", from=5-2, to=7-2]
	\arrow[""{name=7, anchor=center, inner sep=0}, "{\alpha^{-1}_{X(UV)W}}"', from=5-4, to=7-4]
	\arrow[""{name=8, anchor=center, inner sep=0}, "{\alpha^{-1}_{(XU)VW}}", from=5-5, to=7-5]
	\arrow["{\alpha^{-1}_{UVX}\otimes1_W}", curve={height=-12pt}, from=7-1, to=7-2]
	\arrow[""{name=9, anchor=center, inner sep=0}, "{\sigma_{(UV)X}\otimes1_W}"', from=7-2, to=7-4]
	\arrow["{\alpha^{-1}_{XUV}\otimes1_W}", curve={height=-12pt}, from=7-4, to=7-5]
	\arrow["{\Pi^{-1}_{UVWX}}"', shorten <=4pt, shorten >=4pt, Rightarrow, from=1-3, to=1]
	\arrow["{\alpha^{-1}_{1_{UV}\,\boxtimes\,\sigma_{WX}}}"'{pos=0.7}, shift right=3, Rightarrow, from=2, to=0]
	\arrow["{\H^R_{(UV)WX}}"', shorten <=17pt, shorten >=17pt, Rightarrow, from=1, to=9]
	\arrow["{\sigma_{\alpha^{-1}_{UVW}\,\boxtimes\,1_X}}"'{pos=0.3}, shift right=3, Rightarrow, from=4, to=3]
	\arrow["{\overleftarrow{\Pi^{-1}}_{UVXW}}"', shift right=3, Rightarrow, from=6, to=5]
	\arrow["{\overleftarrow{\Pi^{-1}}_{XUVW}}"', shift right, Rightarrow, from=8, to=7]
\end{tikzcd}\]
and 
\[\begin{tikzcd}[column sep=0.15in,row sep=0.3in]
	&& {\big(U((VW)X)\big)} \\
	& {\big(U(V(WX))\big)} && {\big((U(VW))X\big)} \\
	{\big(U(V(XW))\big)} && {\big(U(X(VW))\big)} && {\big(X(U(VW))\big)} \\
	& {\big(U((XV)W)\big)} && {\big((UX)(VW)\big)} \\
	{\big(U((VX)W)\big)} &&&& {\big((XU)(VW)\big)} \\
	& {\big((U(XV))W\big)} && {\big(((UX)V)W\big)} \\
	{\big((U(VX))W\big)} & {\big(((UV)X)W\big)} && {\big((X(UV))W\big)} & {\big(((XU)V)W\big)}
	\arrow["{\alpha^{-1}_{U(VW)X}}", from=1-3, to=2-4]
	\arrow[""{name=0, anchor=center, inner sep=0}, "{1_U\otimes\sigma_{(VW)X}}"{description}, from=1-3, to=3-3]
	\arrow["{1_U\otimes\alpha^{-1}_{VWX}}", from=2-2, to=1-3]
	\arrow["{1_U\otimes(1_V\otimes\sigma_{WX})}"', from=2-2, to=3-1]
	\arrow["{\sigma_{(U(VW))X}}", from=2-4, to=3-5]
	\arrow["{\overleftarrow{\H^R}_{U(VW)X}}"{description}, Rightarrow, from=2-4, to=4-4]
	\arrow[""{name=1, anchor=center, inner sep=0}, "{1_U\otimes\alpha^{-1}_{VXW}}"', from=3-1, to=5-1]
	\arrow["{1_U\otimes\alpha^{-1}_{XVW}}", from=3-3, to=4-2]
	\arrow["{\alpha^{-1}_{UX(VW)}}"', from=3-3, to=4-4]
	\arrow["{\alpha^{-1}_{XU(VW)}}", from=3-5, to=5-5]
	\arrow[""{name=2, anchor=center, inner sep=0}, "{\alpha^{-1}_{U(XV)W}}"{description}, from=4-2, to=6-2]
	\arrow[""{name=3, anchor=center, inner sep=0}, "{\sigma_{UX}\otimes1_{(VW)}}"{description}, from=4-4, to=5-5]
	\arrow[""{name=4, anchor=center, inner sep=0}, "{\alpha^{-1}_{(UX)VW}}"{description}, from=4-4, to=6-4]
	\arrow[""{name=5, anchor=center, inner sep=0}, "{1_U\otimes(\sigma_{VX}\otimes1_W)}"{description}, from=5-1, to=4-2]
	\arrow["{\alpha^{-1}_{U(VX)W}}"', from=5-1, to=7-1]
	\arrow["{\alpha^{-1}_{(XU)VW}}", from=5-5, to=7-5]
	\arrow[""{name=6, anchor=center, inner sep=0}, "{\alpha^{-1}_{UXV}\otimes1_W}", from=6-2, to=6-4]
	\arrow[""{name=7, anchor=center, inner sep=0}, "{(\sigma_{UX}\otimes1_V)\otimes1_W}"{description}, from=6-4, to=7-5]
	\arrow[""{name=8, anchor=center, inner sep=0}, "{(1_U\otimes\sigma_{VX})\otimes1_W}"{description}, from=7-1, to=6-2]
	\arrow["{\alpha^{-1}_{UVX}\otimes1_W}"', curve={height=12pt}, from=7-1, to=7-2]
	\arrow[""{name=9, anchor=center, inner sep=0}, "{\sigma_{(UV)X}\otimes1_W}"', from=7-2, to=7-4]
	\arrow["{\alpha^{-1}_{XUV}\otimes1_W}"', curve={height=12pt}, from=7-4, to=7-5]
	\arrow["{1_{1_U}\otimes\overleftarrow{\H^R}_{VWX}}"{description}, shorten <=25pt, shorten >=25pt, Rightarrow, from=0, to=1]
	\arrow["{\overleftarrow{\Pi^{-1}}_{UXVW}}"', shorten <=25pt, shorten >=25pt, Rightarrow, from=4, to=2]
	\arrow["{-\alpha^{-1}_{\sigma_{UX}\,\boxtimes\,1_{VW}}}"{description}, shorten <=9pt, shorten >=9pt, Rightarrow, from=3, to=7]
	\arrow["{-\alpha^{-1}_{1_U\,\boxtimes\,\sigma_{VX}\,\boxtimes\,1_W}}"{description}, shorten <=9pt, shorten >=9pt, Rightarrow, from=5, to=8]
	\arrow["{\overleftarrow{\H^R}_{UVX}\otimes1_{1_W}}"{description}, Rightarrow, from=6, to=9]
\end{tikzcd}\]
which reads explicitly as,
\begin{align}
&(\alpha^{-1}_{XUV}\otimes1_W)\alpha^{-1}_{X(UV)W}\sigma_{\alpha^{-1}_{UVW}\,\boxtimes\,1_X}\alpha^{-1}_{U(VW)X}(1_U\otimes\alpha^{-1}_{VWX})-\Pi^{-1}_{XUVW}\sigma_{(U(VW))X)}\alpha^{-1}_{U(VW)X}(1_U\otimes\alpha^{-1}_{VWX})\nn\\
&+(\alpha^{-1}_{XUV}\otimes1_W)\H^R_{(UV)WX}\alpha^{-1}_{UV(WX)}+(\alpha^{-1}_{XUV}\sigma_{(UV)X}\otimes1_W)\alpha^{-1}_{(UV)XW}\alpha^{-1}_{1_{UV}\,\boxtimes\,\sigma_{WX}}\nn\\&+(\alpha^{-1}_{XUV}\otimes1_W)\alpha^{-1}_{X(UV)W}\sigma_{((UV)W)X}\Pi^{-1}_{UVWX}-(\alpha^{-1}_{XUV}\sigma_{(UV)X}\otimes1_W)\Pi^{-1}_{UVXW}\big(1_U\otimes(1_V\otimes\sigma_{WX})\big)\nn\\&=\label{eq:right tetrahedron}\\
&-\alpha^{-1}_{(XU)VW}\H^R_{U(VW)X}(1_U\otimes\alpha^{-1}_{VWX})-\left((\sigma_{UX}\otimes1_V)\alpha^{-1}_{UXV}\otimes1_W\right)\alpha^{-1}_{1_U\,\boxtimes \,\sigma_{VX}\,\boxtimes \,1_W}\left(1_U\otimes\alpha^{-1}_{VXW}(1_V\otimes\sigma_{WX})\right)\nn\\&-\big((\sigma_{UX}\otimes1_V)\otimes1_W\big)\Pi^{-1}_{UXVW}\big(1_U\otimes\sigma_{(VW)X}\alpha^{-1}_{VWX}\big)-\alpha^{-1}_{\sigma_{UX}\,\boxtimes\,1_{VW}}\alpha^{-1}_{UX(VW)}\big(1_U\otimes\sigma_{(VW)X}\alpha^{-1}_{VWX}\big)\nn\\&-\left((\sigma_{UX}\otimes1_V)\alpha^{-1}_{UXV}\otimes1_W\right)\alpha^{-1}_{U(XV)W}(1_U\otimes\H^R_{VWX})-\left(\H^R_{UVX}\otimes1_W\right)\alpha^{-1}_{U(VX)W}\left(1_U\otimes\alpha^{-1}_{VXW}(1_V\otimes\sigma_{WX})\right)\nn
\end{align}
3. \textbf{Hexahedron}, i.e., the following two pasting diagrams in $\CC$ are equivalent:
\[\begin{tikzcd}[column sep=0.15in,row sep=0.3in]
	& {\big((U(VW))X\big)} && {\big(U((VW)X)\big)} \\
	{\big(((UV)W)X\big)} &&&& {\big(U(V(WX))\big)} \\
	{\big((W(UV))X\big)} && {\big((UV)(WX)\big)} && {\big(U((WX)V)\big)} \\
	\\
	\\
	\\
	{\big(W((UV)X)\big)} && {\big((WX)(UV)\big)} && {\big((U(WX))V\big)} \\
	{\big(W(X(UV))\big)} &&&& {\big(((WX)U)V\big)} \\
	& {\big(W((XU)V)\big)} && {\big((W(XU))V\big)}
	\arrow["{\alpha_{U(VW)X}}", from=1-2, to=1-4]
	\arrow["{\alpha^{-1}_{UVW}\otimes1_X}"', from=1-2, to=2-1]
	\arrow["{1_U\otimes\alpha_{VWX}}", from=1-4, to=2-5]
	\arrow["{\alpha^{-1}_{UV(WX)}\overleftarrow{\Pi}_{UVWX}(\alpha^{-1}_{UVW}\otimes1_X)}", shorten <=20pt, shorten >=20pt, Rightarrow, from=2-1, to=2-5]
	\arrow["{\sigma_{(UV)W}\otimes1_X}"', from=2-1, to=3-1]
	\arrow["{\alpha_{(UV)WX}}"', from=2-1, to=3-3]
	\arrow["{\alpha^{-1}_{UV(WX)}}", from=2-5, to=3-3]
	\arrow["{1_U\otimes\sigma_{V(WX)}}", from=2-5, to=3-5]
	\arrow[""{name=0, anchor=center, inner sep=0}, "{\alpha_{W(UV)X}}"', from=3-1, to=7-1]
	\arrow[""{name=1, anchor=center, inner sep=0}, "{\sigma_{(UV)(WX)}}"{description}, from=3-3, to=7-3]
	\arrow[""{name=2, anchor=center, inner sep=0}, "{\alpha^{-1}_{U(WX)V}}", from=3-5, to=7-5]
	\arrow["{1_W\otimes\sigma_{(UV)X}}"', from=7-1, to=8-1]
	\arrow["{\alpha_{WX(UV)}}"', from=7-3, to=8-1]
	\arrow["{\alpha^{-1}_{(WX)UV}}", from=7-3, to=8-5]
	\arrow["{\sigma_{U(WX)}\otimes1_V}", from=7-5, to=8-5]
	\arrow["{(1_W\otimes\alpha^{-1}_{XUV})\overleftarrow{\Pi}_{WXUV}\alpha^{-1}_{(WX)UV}}", shorten <=20pt, shorten >=20pt, Rightarrow, from=8-1, to=8-5]
	\arrow["{1_W\otimes\alpha^{-1}_{XUV}}"', from=8-1, to=9-2]
	\arrow["{\alpha_{WXU}\otimes1_V}", from=8-5, to=9-4]
	\arrow["{\alpha_{W(XU)V}}", from=9-4, to=9-2]
	\arrow["{\overleftarrow{\H^L}_{(UV)WX}}", shorten <=38pt, shorten >=38pt, Rightarrow, from=0, to=1]
	\arrow["{\H^R_{UV(WX)}}", shorten <=38pt, shorten >=38pt, Rightarrow, from=1, to=2]
\end{tikzcd}\]
and
\[\begin{tikzcd}[column sep=0.15in,row sep=0.3in]
	{\big(((UV)W)X\big)} & {\big((U(VW))X\big)} && {\big(U((VW)X)\big)} & {\big(U(V(WX))\big)} \\
	{\big((W(UV))X\big)} &&&& {\big(U((WX)V)\big)} \\
	& {\big((U(WV))X\big)} && {\big(U((WV)X)\big)} \\
	\\
	{\big(((WU)V)X\big)} & {\big(((UW)V)X\big)} & {\big((UW)(VX)\big)} & {\big(U(W(VX))\big)} & {\big(U(W(XV))\big)} \\
	\\
	& {\big((WU)(VX)\big)} & {} & {\big((UW)(XV)\big)} \\
	\\
	{\big(W(U(VX))\big)} & {\big(W(U(XV))\big)} & {\big((WU)(XV)\big)} & {\big(((WU)X)V\big)} & {\big(((UW)X)V\big)} \\
	\\
	& {\big(W((UX)V)\big)} && {\big((W(UX))V\big)} \\
	{\big(W((UV)X)\big)} &&&& {\big((U(WX))V\big)} \\
	{\big(W(X(UV))\big)} & {\big(W((XU)V)\big)} && {\big((W(XU))V\big)} & {\big(((WX)U)V\big)}
	\arrow["{\sigma_{(UV)W}\otimes1_X}"', from=1-1, to=2-1]
	\arrow[""{name=0, anchor=center, inner sep=0}, "{\alpha^{-1}_{UVW}\otimes1_X}"', curve={height=12pt}, from=1-2, to=1-1]
	\arrow["{\alpha_{U(VW)X}}", from=1-2, to=1-4]
	\arrow[""{name=1, anchor=center, inner sep=0}, "{(1_U\otimes\sigma_{VW})\otimes1_X}"{description}, from=1-2, to=3-2]
	\arrow[""{name=2, anchor=center, inner sep=0}, "{1_U\otimes\alpha_{VWX}}", curve={height=-12pt}, from=1-4, to=1-5]
	\arrow[""{name=3, anchor=center, inner sep=0}, "{1_U\otimes(\sigma_{VW}\otimes1_X)}"{description}, from=1-4, to=3-4]
	\arrow["{1_U\otimes\sigma_{V(WX)}}", from=1-5, to=2-5]
	\arrow["{\alpha^{-1}_{WUV}\otimes1_X}"{description}, from=2-1, to=5-1]
	\arrow[""{name=4, anchor=center, inner sep=0}, "{\alpha_{W(UV)X}}"{description, pos=0.6}, shift right=5, curve={height=30pt}, dashed, from=2-1, to=12-1]
	\arrow["{1_U\otimes\alpha_{WXV}}"{description}, from=2-5, to=5-5]
	\arrow[""{name=5, anchor=center, inner sep=0}, "{\alpha^{-1}_{U(WX)V}}"{description, pos=0.6}, shift left=5, curve={height=-30pt}, dashed, from=2-5, to=12-5]
	\arrow["{\alpha_{U(WV)X}}"{description}, from=3-2, to=3-4]
	\arrow[""{name=6, anchor=center, inner sep=0}, "{\alpha^{-1}_{UWV}\otimes1_X}"{description, pos=0.4}, from=3-2, to=5-2]
	\arrow[""{name=7, anchor=center, inner sep=0}, "{1_U\otimes\alpha_{WVX}}"{description, pos=0.4}, from=3-4, to=5-4]
	\arrow[""{name=8, anchor=center, inner sep=0}, "{\alpha_{(WU)VX}}"{description, pos=0.3}, from=5-1, to=7-2]
	\arrow[""{name=9, anchor=center, inner sep=0}, "{(\sigma_{UW}\otimes1_V)\otimes1_X}"', curve={height=12pt}, from=5-2, to=5-1]
	\arrow["{\alpha_{(UW)VX}}", curve={height=-12pt}, from=5-2, to=5-3]
	\arrow[""{name=10, anchor=center, inner sep=0}, "{\sigma_{UW}\otimes1_{(VX)}}"{description, pos=0.3}, from=5-3, to=7-2]
	\arrow[""{name=11, anchor=center, inner sep=0}, "{1_{(UW)}\otimes\sigma_{VX}}"{description, pos=0.3}, from=5-3, to=7-4]
	\arrow["{\alpha^{-1}_{UW(VX)}}"', curve={height=12pt}, from=5-4, to=5-3]
	\arrow[""{name=12, anchor=center, inner sep=0}, "{1_U\otimes(1_W\otimes\sigma_{VX})}", curve={height=-12pt}, from=5-4, to=5-5]
	\arrow[""{name=13, anchor=center, inner sep=0}, "{\alpha^{-1}_{UW(XV)}}"{description, pos=0.3}, from=5-5, to=7-4]
	\arrow[""{name=14, anchor=center, inner sep=0}, "{\alpha_{WU(VX)}}"{description, pos=0.7}, from=7-2, to=9-1]
	\arrow[""{name=15, anchor=center, inner sep=0}, "{1_{(WU)}\otimes\sigma_{VX}}"{description, pos=0.7}, from=7-2, to=9-3]
	\arrow[from=7-3, to=7-3, loop, in=55, out=125, distance=10mm]
	\arrow[""{name=16, anchor=center, inner sep=0}, "{\sigma_{UW}\otimes1_{(XV)}}"{description, pos=0.7}, from=7-4, to=9-3]
	\arrow[""{name=17, anchor=center, inner sep=0}, "{\alpha^{-1}_{(UW)XV}}"{description, pos=0.7}, from=7-4, to=9-5]
	\arrow[""{name=18, anchor=center, inner sep=0}, "{1_W\otimes(1_U\otimes\sigma_{VX})}"', curve={height=12pt}, from=9-1, to=9-2]
	\arrow["{1_W\otimes\alpha^{-1}_{UVX}}"{description}, from=9-1, to=12-1]
	\arrow[""{name=19, anchor=center, inner sep=0}, "{1_W\otimes\alpha^{-1}_{UXV}}"{description}, from=9-2, to=11-2]
	\arrow["{\alpha_{WU(XV)}}", curve={height=-12pt}, from=9-3, to=9-2]
	\arrow["{\alpha^{-1}_{(WU)XV}}"', curve={height=12pt}, from=9-3, to=9-4]
	\arrow[""{name=20, anchor=center, inner sep=0}, "{\alpha_{WUX}\otimes1_V}"{description}, from=9-4, to=11-4]
	\arrow[""{name=21, anchor=center, inner sep=0}, "{(\sigma_{UW}\otimes1_X)\otimes1_V}", curve={height=-12pt}, from=9-5, to=9-4]
	\arrow["{\alpha_{UWX}\otimes1_V}"{description}, from=9-5, to=12-5]
	\arrow[""{name=22, anchor=center, inner sep=0}, "{1_W\otimes(\sigma_{UX}\otimes1_V)}"{description}, from=11-2, to=13-2]
	\arrow["{\alpha_{W(UX)V}}"{description}, from=11-4, to=11-2]
	\arrow[""{name=23, anchor=center, inner sep=0}, "{(1_W\otimes\sigma_{UX})\otimes1_V}"{description}, from=11-4, to=13-4]
	\arrow["{1_W\otimes\sigma_{(UV)X}}"', from=12-1, to=13-1]
	\arrow["{\sigma_{U(WX)}\otimes1_V}", from=12-5, to=13-5]
	\arrow[""{name=24, anchor=center, inner sep=0}, "{1_W\otimes\alpha^{-1}_{XUV}}"', curve={height=12pt}, from=13-1, to=13-2]
	\arrow["{\alpha_{W(XU)V}}", from=13-4, to=13-2]
	\arrow[""{name=25, anchor=center, inner sep=0}, "{\alpha_{WXU}\otimes1_V}", curve={height=-12pt}, from=13-5, to=13-4]
	\arrow["{-\alpha_{1_U\,\boxtimes\,\sigma_{VW}\,\boxtimes\,1_X}}", shorten <=38pt, shorten >=38pt, Rightarrow, from=1, to=3]
	\arrow["{\H^R_{UVW}\otimes1_{1_X}}"{description}, shorten <=17pt, shorten >=17pt, Rightarrow, from=0, to=9]
	\arrow["{(1_W\otimes\alpha^{-1}_{UVX})\Pi_{WUVX}(\alpha^{-1}_{WUV}\otimes1_X)}"{pos=0.4}, Rightarrow, from=4, to=7-2]
	\arrow["{\alpha^{-1}_{UW(VX)}\overleftarrow{\Pi}_{UWVX}(\alpha^{-1}_{UWV}\otimes1_X)}", shorten <=38pt, shorten >=38pt, Rightarrow, from=6, to=7]
	\arrow["{-\alpha_{\sigma_{UW}\,\boxtimes\,1_{VX}}}", shift left=4, shorten <=13pt, shorten >=13pt, Rightarrow, from=8, to=10]
	\arrow["{\alpha^{-1}_{1_{UW}\,\boxtimes\,\sigma_{VX}}}", shift left=4, shorten <=13pt, shorten >=13pt, Rightarrow, from=11, to=13]
	\arrow["{1_{1_U}\otimes\overleftarrow{\H^L}_{VWX}}"{description}, shorten <=17pt, shorten >=17pt, Rightarrow, from=12, to=2]
	\arrow["{\alpha_{1_{WU}\,\boxtimes \,\sigma_{VX}}}", shift right=5, shorten <=13pt, shorten >=13pt, Rightarrow, from=14, to=15]
	\arrow["{-\alpha^{-1}_{\sigma_{UW}\,\boxtimes\,1_{XV}}}", shift right=5, shorten <=13pt, shorten >=13pt, Rightarrow, from=16, to=17]
	\arrow["{(\alpha_{UWX}\otimes1_V)\overleftarrow{\Pi^{-1}}_{UWXV}(1_U\otimes\alpha_{WXV})}"{pos=0.6}, shorten >=8pt, Rightarrow, from=7-4, to=5]
	\arrow["{(1_W\otimes\alpha^{-1}_{UXV})\overleftarrow{\Pi}_{WUXV}\alpha^{-1}_{(WU)XV}}"', shorten <=26pt, shorten >=26pt, Rightarrow, from=19, to=20]
	\arrow["{\overleftarrow{\H^L}_{UWX}\otimes1_{1_V}}"{description}, shorten <=17pt, shorten >=17pt, Rightarrow, from=21, to=25]
	\arrow["{\alpha_{1_W\,\boxtimes\,\sigma_{UX}\,\boxtimes\,1_V}}", shorten <=38pt, shorten >=38pt, Rightarrow, from=22, to=23]
	\arrow["{1_{1_W}\otimes\H^R_{UVX}}"{description}, shorten <=17pt, shorten >=17pt, Rightarrow, from=24, to=18]
\end{tikzcd}\]
which reads explicitly as,
\begin{footnotesize}
\begin{flalign}
&\alpha_{W(XU)V}(\alpha_{WXU}\otimes1_V)\left[\H^R_{UV(WX)}(1_U\otimes\alpha_{VWX})\alpha_{U(VW)X}-\alpha^{-1}_{(WX)UV}\sigma_{(UV)(WX)}\alpha^{-1}_{UV(WX)}\Pi_{UVWX}(\alpha^{-1}_{UVW}\otimes1_X)\right]\nn\\&-(1_W\otimes\alpha^{-1}_{XUV})\left[\Pi_{WXUV}\alpha^{-1}_{(WX)UV}\sigma_{(UV)(WX)}\alpha_{(UV)WX}+\H^L_{(UV)WX}\right](\alpha^{-1}_{UVW}\otimes1_X)\nn\\&=\label{eq:Hexahedron}\\&\left[(1_W\otimes\alpha^{-1}_{XUV}\sigma_{(UV)X}\alpha^{-1}_{UVX})\Pi_{WUVX}+(1_W\otimes\,\H^R_{UVX})\alpha_{WU(VX)}\alpha_{(WU)VX}\right](\alpha^{-1}_{WUV}\sigma_{(UV)W}\alpha^{-1}_{UVW}\otimes1_X)\nn\\&+[1_W\otimes(\sigma_{UX}\otimes1_V)\alpha^{-1}_{UXV}(1_U\otimes\sigma_{VX})]\alpha_{WU(VX)}\left[\alpha_{(WU)VX}(\H^R_{UVW}\otimes1_X)-\alpha_{\sigma_{UW}\boxtimes1_{VX}}\left(\alpha^{-1}_{UWV}(1_U\otimes\sigma_{VW})\otimes1_X\right)\right]\nn\\&+(1_W\otimes[\sigma_{UX}\otimes1_V]\alpha^{-1}_{UXV})\alpha_{1_{WU}\boxtimes\sigma_{VX}}(\sigma_{UW}\otimes1_{(VX)})\alpha_{(UW)VX}\left(\alpha^{-1}_{UWV}[1_U\otimes\sigma_{VW}]\otimes1_X\right)\nn\\&+\alpha_{1_W\boxtimes\sigma_{UX}\boxtimes1_V}(\alpha_{WUX}\otimes1_V)\alpha^{-1}_{(WU)XV}[\sigma_{UW}\otimes\sigma_{VX}]\alpha_{(UW)VX}\left[\alpha^{-1}_{UWV}[1_U\otimes\sigma_{VW}]\otimes1_X\right]\nn\\&-\left[1_W\otimes[\sigma_{UX}\otimes1_V]\alpha^{-1}_{UXV}\right]\Pi_{WUXV}\alpha^{-1}_{(WU)XV}[\sigma_{UW}\otimes\sigma_{VX}]\alpha_{(UW)VX}\left[\alpha^{-1}_{UWV}[1_U\otimes\sigma_{VW}]\otimes1_X\right]\nn\\&-\alpha_{W(XU)V}\left[(1_W\otimes\sigma_{UX})\alpha_{WUX}\otimes1_V\right]\alpha^{-1}_{(WU)XV}(\sigma_{UW}\otimes\sigma_{VX})\alpha^{-1}_{UW(VX)}\Pi_{UWVX}[\alpha^{-1}_{UWV}(1_U\otimes\sigma_{VW})\otimes1_X]\nn\\&-\alpha_{W(XU)V}\left[(1_W\otimes\sigma_{UX})\alpha_{WUX}\otimes1_V\right]\alpha^{-1}_{(WU)XV}(\sigma_{UW}\otimes\sigma_{VX})\alpha^{-1}_{UW(VX)}(1_U\otimes\alpha_{WVX})\alpha_{1_U\boxtimes\sigma_{VW}\boxtimes1_X}\nn\\&-\alpha_{W(XU)V}\left[(1_W\otimes\sigma_{UX})\alpha_{WUX}\otimes1_V\right]\alpha^{-1}_{\sigma_{UW}\boxtimes1_{XV}}(1_{(UW)}\otimes\sigma_{VX})\alpha^{-1}_{UW(VX)}[1_U\otimes\alpha_{WVX}(\sigma_{VW}\otimes1_X)]\alpha_{U(VW)X}\nn\\&+\alpha_{W(XU)V}\left[(1_W\otimes\sigma_{UX})\alpha_{WUX}(\sigma_{UW}\otimes1_X)\otimes1_V\right]\alpha^{-1}_{(UW)XV}\alpha^{-1}_{1_{UW}\boxtimes\sigma_{VX}}[1_U\otimes\alpha_{WVX}(\sigma_{VW}\otimes1_X)]\alpha_{U(VW)X}\nn\\&-\alpha_{W(XU)V}(\H^L_{UWX}\otimes1_V)\alpha^{-1}_{(UW)XV}\alpha^{-1}_{UW(XV)}[1_U\otimes(1_W\otimes\sigma_{VX})\alpha_{WVX}(\sigma_{VW}\otimes1_X)]\alpha_{U(VW)X}\nn\\&-\alpha_{W(XU)V}\left(\alpha_{WXU}\sigma_{U(WX)}\alpha_{UWX}\otimes1_V\right)\alpha^{-1}_{(UW)XV}\alpha^{-1}_{UW(XV)}(1_U\otimes\H^L_{VWX})\alpha_{U(VW)X}\nn\\&-\alpha_{W(XU)V}\left(\alpha_{WXU}\sigma_{U(WX)}\alpha_{UWX}\otimes1_V\right)\Pi^{-1}_{UWXV}(1_U\otimes\alpha_{WXV}\sigma_{V(WX)}\alpha_{VWX})\alpha_{U(VW)X}\nn
\end{flalign}    
\end{footnotesize}

4. \textbf{Breen polytope}, i.e., the following two pasting diagrams in $\CC$ are equivalent:
\[\begin{tikzcd}[column sep=0.15in,row sep=0.3in]
	& {\big((UV)W\big)} \\
	{\big((VU)W\big)} && {\big(U(VW)\big)} \\
	{\big(V(UW)\big)} && {\big(U(WV)\big)} \\
	{\big(V(WU)\big)} && {\big((UW)V\big)} \\
	{\big((VW)U\big)} && {\big((WU)V\big)} \\
	{\big((WV)U\big)} && {\big(W(UV)\big)} \\
	& {\big(W(VU)\big)}
	\arrow[""{name=0, anchor=center, inner sep=0}, "{\sigma_{UV}\otimes1_W}"', from=1-2, to=2-1]
	\arrow["{\alpha_{UVW}}", from=1-2, to=2-3]
	\arrow["{\alpha_{VUW}}"', from=2-1, to=3-1]
	\arrow["{1_U\otimes\sigma_{VW}}", from=2-3, to=3-3]
	\arrow[""{name=1, anchor=center, inner sep=0}, "{\sigma_{U(VW)}}"{description}, curve={height=12pt}, from=2-3, to=5-1]
	\arrow["{1_V\otimes\sigma_{UW}}"', from=3-1, to=4-1]
	\arrow["{\alpha^{-1}_{UWV}}", from=3-3, to=4-3]
	\arrow[""{name=2, anchor=center, inner sep=0}, "{\sigma_{U(WV)}}"{description}, curve={height=-12pt}, from=3-3, to=6-1]
	\arrow["{\alpha^{-1}_{VWU}}"', from=4-1, to=5-1]
	\arrow["{\sigma_{UW}\otimes1_V}", from=4-3, to=5-3]
	\arrow["{\sigma_{VW}\otimes1_U}"', from=5-1, to=6-1]
	\arrow["{\alpha_{WUV}}", from=5-3, to=6-3]
	\arrow["{\alpha_{WVU}}"', from=6-1, to=7-2]
	\arrow[""{name=3, anchor=center, inner sep=0}, "{1_W\otimes\sigma_{UV}}", from=6-3, to=7-2]
	\arrow["{\alpha^{-1}_{VWU}\overleftarrow{\H^L}_{UVW}}", shorten <=8pt, shorten >=8pt, Rightarrow, from=0, to=1]
	\arrow["{\sigma_{1_U\,\boxtimes\,\sigma_{VW}}}", shorten <=7pt, shorten >=7pt, Rightarrow, from=1, to=2]
	\arrow["{\H^L_{UWV}\alpha^{-1}_{UWV}}"', shorten <=8pt, Rightarrow, from=2, to=3]
\end{tikzcd}
\begin{tikzcd}
	& {\big((UV)W\big)} \\
	{\big((VU)W\big)} && {\big(U(VW)\big)} \\
	{\big(V(UW)\big)} && {\big(U(WV)\big)} \\
	{\big(V(WU)\big)} && {\big((UW)V\big)} \\
	{\big((VW)U\big)} && {\big((WU)V\big)} \\
	{\big((WV)U\big)} && {\big(W(UV)\big)} \\
	& {\big(W(VU)\big)}
	\arrow[""{name=0, anchor=center, inner sep=0}, "{\sigma_{UV}\otimes1_W}"', from=1-2, to=2-1]
	\arrow["{\alpha_{UVW}}", from=1-2, to=2-3]
	\arrow[""{name=1, anchor=center, inner sep=0}, "{\sigma_{(UV)W}}"{description, pos=0.4}, curve={height=12pt}, from=1-2, to=6-3]
	\arrow["{\alpha_{VUW}}"', from=2-1, to=3-1]
	\arrow[""{name=2, anchor=center, inner sep=0}, "{\sigma_{(VU)W}}"{description, pos=0.6}, shift left=2, curve={height=-20pt}, from=2-1, to=7-2]
	\arrow["{1_U\otimes\sigma_{VW}}", shift left=5, curve={height=-6pt}, from=2-3, to=3-3]
	\arrow["{1_V\otimes\sigma_{UW}}"', from=3-1, to=4-1]
	\arrow["{\alpha^{-1}_{UWV}}", from=3-3, to=4-3]
	\arrow["{\alpha^{-1}_{VWU}}"', from=4-1, to=5-1]
	\arrow["{\sigma_{UW}\otimes1_V}", from=4-3, to=5-3]
	\arrow["{\sigma_{VW}\otimes1_U}"', shift right=5, curve={height=6pt}, from=5-1, to=6-1]
	\arrow["{\alpha_{WUV}}", from=5-3, to=6-3]
	\arrow["{\alpha_{WVU}}"', from=6-1, to=7-2]
	\arrow[""{name=3, anchor=center, inner sep=0}, "{1_W\otimes\sigma_{UV}}", from=6-3, to=7-2]
	\arrow["{\alpha_{WUV}\H^R_{UVW}\alpha_{UVW}}"{description, pos=0.7}, shift left=5, curve={height=12pt}, shorten <=13pt, shorten >=4pt, Rightarrow, from=1, to=2-3]
	\arrow["{-\sigma_{\sigma_{UV}\,\boxtimes\,1_W}}"{description, pos=0.2}, shorten >=11pt, Rightarrow, from=0, to=3]
	\arrow["{\alpha_{WVU}\overleftarrow{\H^R}_{VUW}\alpha_{VUW}}"{description, pos=0.3}, shift right=4, shorten <=4pt, shorten >=17pt, Rightarrow, from=6-1, to=2]
\end{tikzcd}\]
which reads explicitly as, 
\begin{align}
&-\alpha_{WVU}(\sigma_{VW}\otimes1_U)\alpha^{-1}_{VWU}\H^L_{UVW}+\alpha_{WVU}\sigma_{1_U\,\boxtimes \,\sigma_{VW}}\alpha_{UVW}+\H^L_{UWV}\alpha^{-1}_{(UW)V}(1_U\otimes\sigma_{VW})\alpha_{UVW}\nn\\&=\label{eqn:Breen axiom}\\&-\alpha_{WVU}\H^R_{VUW}\alpha_{VUW}(\sigma_{UV}\otimes1_W)-\sigma_{\sigma_{UV}\,\boxtimes \,1_W}+(1_W\otimes\sigma_{UV})\alpha_{WUV}\H^R_{UVW}\alpha_{UVW}\quad.\nn
\end{align}
\end{defi}

\begin{defi}\label{def:hexagonal braiding}
A \textbf{hexagonal braiding} $\sigma$ on a strict monoidal $\Ch^{[-1,0]}_R$-category $(\CC,\otimes,I)$ is a pseudonatural isomorphism $\sigma:\otimes\Rightarrow\otimes\,\tau_{\CC,\CC}$ which satisfies the following three axioms:
\begin{enumerate}
\item \textbf{Strict left hexagon axiom},
\begin{subequations}
\begin{equation}\label{eq:strict Left hexagon axiom}
\sigma_{\id_{\CC}\,\boxtimes\,\otimes}\,\equiv\,\otimes\left([1\boxtimes\sigma]_{\tau_{\mathbf{C,C}}\,\boxtimes\,\id_\CC}(\sigma\boxtimes1)\right)\quad.
\end{equation}
\item \textbf{Strict right hexagon axiom},
\begin{equation}\label{eq:strict Right hexagon axiom}
\sigma_{\otimes\,\boxtimes\,\id_{\CC}}\,\equiv\,\otimes\left([\sigma\boxtimes1]_{\id_\CC\,\boxtimes\,\tau_{\mathbf{C,C}}}(1\boxtimes\sigma)\right)\quad.
\end{equation}
\item \textbf{Strict hexagonal Breen polytope},
\begin{equation}\label{eqn:strict Breen axiom}
\sigma_{1_U\,\boxtimes \,\sigma_{VW}}+\sigma_{\sigma_{UV}\,\boxtimes \,1_W}=0\quad.
\end{equation}
\end{subequations}
\end{enumerate}
\end{defi}
\begin{rem}\label{rem:symstrict mon ChCat defi}
If $\gamma:\otimes\Rightarrow\otimes\,\tau_{\CC,\CC}
$ is a $\Ch^{[-1,0]}_R$-natural \textbf{involution}\footnote{That is, a $\Ch^{[-1,0]}_R$-natural isomorphism which further satisfies \begin{equation}\label{involutory sigma}
\gamma_{\tau_{\CC,\CC}}\gamma=\Id_\otimes\qquad\implies\qquad\gamma_{VU}\gamma_{UV}=1_{(UV)}\quad.
\end{equation}}, then the involutive property of $\gamma$ together with the functoriality of $\otimes$ means that $\gamma$ satisfies the strict left hexagon axiom \eqref{eq:strict Left hexagon axiom},
\begin{subequations}
\begin{equation}
\gamma_{U(VW)}=(1_V\otimes\gamma_{UW})(\gamma_{UV}\otimes1_W)\quad,
\end{equation}
if and only if it satisfies the strict right hexagon axiom \eqref{eq:strict Right hexagon axiom},
\begin{equation}
\gamma_{(UV)W}=(\gamma_{UW}\otimes1_V)(1_U\otimes\gamma_{VW})\quad,
\end{equation}
\end{subequations}
in which case we call $(\CC,\otimes,I,\gamma)$ a \textbf{symmetric strict monoidal $\Ch^{[-1,0]}_R$-category}.
\end{rem}

\section{Syllepses}\label{sec:syllepses}
Following on from Subsection \ref{subs:brastrunmon cats}, we specify the general notion of a syllepsis on a braided monoidal bicategory as in \cite[Definition C.4]{Schommer} to our context of braided strictly-unital monoidal $\Ch^{[-1,0]}_R$-categories (see Definition \ref{def:bra str un mon cat}). The latter part of Definition \ref{def:syllepsis} translates the notion of a \textit{symmetric} syllepsis as in \cite[Definition C.5]{Schommer} to our context. We then introduce ``infinitesimal syllepses" on symmetric strict monoidal $\Ch^{[-1,0]}_R$-categories and show that 3-shifted Poisson structures give rise to them analogously to \cite[(3.25a)]{Us}. Remark \ref{rem:infsyllep easily integrates} discusses the curious fact that infinitesimal syllepses can be `integrated' to all orders simply by post-whiskering with the symmetric braiding $\gamma$.
\subsection{Symmetric syllepses on braided strictly-unital monoidal cochain 2-categories}\label{subsec:syllepsis defi}
\begin{defi}\label{def:syllepsis}
If $(\CC,\otimes,I,\alpha,\Pi,\sigma,\H^L,\H^R)$ is a braided strictly-unital monoidal $\Ch^{[-1,0]}_R$-category then a \textbf{syllepsis} $\Sigma$ consists of the following datum:
\begin{enumerate}
\item[(i)] A modification $\Sigma:\sigma\Rrightarrow\sigma^{-1}_{\tau_{\CC,\CC}}:\otimes\Rightarrow\otimes\,\tau_{\CC,\CC}:\CC\boxtimes\CC\rightarrow\CC$. Specifying Definition \ref{defi: modification}, this requires:
\begin{itemize}
\item $\forall U,V\in\CC$, we have the following homotopy $\Sigma_{UV}:\sigma_{UV}\Rightarrow\sigma^{-1}_{VU}:(UV)\rightarrow(VU)$, i.e.,
\begin{subequations}
\begin{equation}\label{eq:syllepsis is homotopy}
\partial\Sigma_{UV}=\sigma_{UV}-\sigma^{-1}_{VU}\quad.
\end{equation}
\item For all $f\in\CC[U,U']^0$ and $g\in\CC[V,V']^0$, 
\begin{equation}\label{eq:syllepsis is quasinatural}
\Sigma_{U'V'}(f\otimes g)+\sigma_{f\,\boxtimes\,g}=\sigma^{-1}_{g\,\boxtimes\,f}+(g\otimes f)\Sigma_{UV}\quad.
\end{equation}
\end{subequations}
\end{itemize}
\end{enumerate}
This datum is required to satisfy the following two axioms:

1. \textbf{Left factorisation}, i.e., for all $UVW\in\CC^{\,\boxtimes\,3}$, the component of the left hexagonator \eqref{eq:left hexagonator as 2-cell} is equal to the following pasting diagram in $\CC$,
\begin{subequations}
\begin{equation}
\begin{tikzcd}
	& {(U(VW))} &&& {((VW)U)} \\
	\\
	{((UV)W)} &&&&& {(V(WU))} \\
	\\
	& {((VU)W)} &&& {(V(UW))}
	\arrow[""{name=0, anchor=center, inner sep=0}, "{\sigma_{U(VW)}}", curve={height=-24pt}, from=1-2, to=1-5]
	\arrow[""{name=1, anchor=center, inner sep=0}, "{\sigma^{-1}_{(VW)U}}"', curve={height=18pt}, from=1-2, to=1-5]
	\arrow["{\alpha_{VWU}}", from=1-5, to=3-6]
	\arrow["{\alpha_{UVW}}", from=3-1, to=1-2]
	\arrow[""{name=2, anchor=center, inner sep=0}, "{\sigma_{UV}\otimes1_W}"', curve={height=30pt}, from=3-1, to=5-2]
	\arrow[""{name=3, anchor=center, inner sep=0}, "{\sigma_{VU}^{-1}\otimes1_W}", shift right=2, curve={height=-30pt}, from=3-1, to=5-2]
	\arrow[""{name=4, anchor=center, inner sep=0}, "{\alpha_{VUW}}"', from=5-2, to=5-5]
	\arrow[""{name=5, anchor=center, inner sep=0}, "{1_V\otimes\sigma_{UW}}"', curve={height=30pt}, from=5-5, to=3-6]
	\arrow[""{name=6, anchor=center, inner sep=0}, "{1_V\otimes\sigma_{WU}^{-1}}", shift right=2, curve={height=-30pt}, from=5-5, to=3-6]
	\arrow["{\Sigma_{U(VW)}}"{description}, Rightarrow, from=0, to=1]
	\arrow["{(\H^R)^{-1}_{VWU}}"{description, pos=0.6}, shorten <=15pt, Rightarrow, from=1, to=4]
	\arrow["{\overleftarrow{\Sigma}_{UV}\otimes1_W}"{description}, Rightarrow, from=3, to=2]
	\arrow["{1_V\otimes\overleftarrow{\Sigma}_{UW}}"{description}, Rightarrow, from=6, to=5]
\end{tikzcd}\qquad.
\end{equation}
In other words, 
\begin{equation}\label{eq:left factorisation}
\H^L_{UVW}=\alpha_{VWU}\Sigma_{U(VW)}\alpha_{UVW}+(\H^R)^{-1}_{VWU}-(1_V\otimes\sigma^{-1}_{WU})\alpha_{VUW}(\Sigma_{UV}\otimes1_W)-(1_V\otimes\Sigma_{UW})\alpha_{VUW}(\sigma_{UV}\otimes1_W)\,.
\end{equation}
\end{subequations}

2. \textbf{Right factorisation}, i.e., for all $UVW\in\CC^{\,\boxtimes\,3}$, the component of the right hexagonator \eqref{eq:right hexagonator as 2-cell} is equal to the following pasting diagram in $\CC$,
\begin{subequations}
\begin{equation}
\begin{tikzcd}
	& {((UV)W)} &&& {(W(UV))} \\
	\\
	{(U(VW))} &&&&& {((WU)V)} \\
	\\
	& {(U(WV))} &&& {((UW)V)}
	\arrow[""{name=0, anchor=center, inner sep=0}, "{\sigma_{(UV)W}}", curve={height=-24pt}, from=1-2, to=1-5]
	\arrow[""{name=1, anchor=center, inner sep=0}, "{\sigma_{W(UV)}^{-1}}"', curve={height=18pt}, from=1-2, to=1-5]
	\arrow["{\alpha_{WUV}^{-1}}", from=1-5, to=3-6]
	\arrow["{\alpha^{-1}_{UVW}}", from=3-1, to=1-2]
	\arrow[""{name=2, anchor=center, inner sep=0}, "{1_U\otimes\sigma_{VW}}"', curve={height=30pt}, from=3-1, to=5-2]
	\arrow[""{name=3, anchor=center, inner sep=0}, "{1_U\otimes\sigma^{-1}_{WV}}", shift right, curve={height=-30pt}, from=3-1, to=5-2]
	\arrow[""{name=4, anchor=center, inner sep=0}, "{\alpha^{-1}_{UWV}}"', from=5-2, to=5-5]
	\arrow[""{name=5, anchor=center, inner sep=0}, "{\sigma_{UW}\otimes1_V}"', curve={height=30pt}, from=5-5, to=3-6]
	\arrow[""{name=6, anchor=center, inner sep=0}, "{\sigma_{WU}^{-1}\otimes1_V}", shift right, curve={height=-30pt}, from=5-5, to=3-6]
	\arrow["{\Sigma_{(UV)W}}"{description}, Rightarrow, from=0, to=1]
	\arrow["{(\H^L)^{-1}_{WUV}}"{description, pos=0.6}, shorten <=15pt, Rightarrow, from=1, to=4]
	\arrow["{1_U\otimes\overleftarrow{\Sigma}_{VW}}"{description}, Rightarrow, from=3, to=2]
	\arrow["{\overleftarrow{\Sigma}_{UW}\otimes1_V}"{description}, Rightarrow, from=6, to=5]
\end{tikzcd}\quad.
\end{equation}
In other words, 
\begin{equation}\label{eq:right factorisation}
\H^R_{UVW}=\alpha^{-1}_{WUV}\Sigma_{(UV)W}\alpha^{-1}_{UVW}+(\H^L)^{-1}_{WUV}-(\sigma^{-1}_{WU}\otimes1_V)\alpha^{-1}_{UWV}(1_U\otimes\Sigma_{VW})-(\Sigma_{UW}\otimes1_V)\alpha^{-1}_{UWV}(1_U\otimes\sigma_{VW})\,.
\end{equation}
\end{subequations}
A syllepsis is said to be \textbf{symmetric} if, for all $UV\in\CC\boxtimes\CC$, we have the following equality between pasting diagrams in $\CC$, 
\begin{subequations}
\begin{equation}
\begin{tikzcd}
	{(UV)} &&&& {(VU)} \\
	\\
	\\
	\\
	{(VU)} &&&& {(UV)}
	\arrow[""{name=0, anchor=center, inner sep=0}, "{\sigma_{UV}}", from=1-1, to=1-5]
	\arrow["{\sigma_{UV}}"', from=1-1, to=5-1]
	\arrow[""{name=1, anchor=center, inner sep=0}, "{1_{(UV)}}"{description}, from=1-1, to=5-5]
	\arrow["{\sigma_{VU}}"', from=5-1, to=5-5]
	\arrow["{\sigma_{UV}}"', from=5-5, to=1-5]
	\arrow["{\sigma_{VU}\overleftarrow{\Sigma}_{UV}}"{description}, shorten <=6pt, Rightarrow, from=1, to=5-1]
	\arrow["{\Id_{\sigma_{UV}}}"{description, pos=0.3}, shorten >=10pt, Rightarrow, from=0, to=5-5]
\end{tikzcd}\quad\equiv\quad
\begin{tikzcd}
	{(UV)} &&&& {(VU)} \\
	\\
	\\
	\\
	{(VU)} &&&& {(UV)}
	\arrow[""{name=0, anchor=center, inner sep=0}, "{\sigma_{UV}}", from=1-1, to=1-5]
	\arrow["{\sigma_{UV}}"', from=1-1, to=5-1]
	\arrow[""{name=1, anchor=center, inner sep=0}, "{1_{(VU)}}"{description}, from=5-1, to=1-5]
	\arrow["{\sigma_{VU}}"', from=5-1, to=5-5]
	\arrow["{\sigma_{UV}}"', from=5-5, to=1-5]
	\arrow["{\Id_{\sigma_{UV}}}"{description, pos=0.3}, shorten <=10pt, Rightarrow, from=0, to=5-1]
	\arrow["{\sigma_{UV}\overleftarrow{\Sigma}_{VU}}"{description}, shorten <=13pt, Rightarrow, from=1, to=5-5]
\end{tikzcd}
\end{equation}
In other words,
\begin{equation}\label{eq:symmetric syllepsis}
\Sigma_{UV}=\sigma^{-1}_{VU}\Sigma_{VU}\sigma_{UV}\qquad.
\end{equation}
\end{subequations}
\end{defi}
\subsubsection{Infinitesimal syllepses on symmetric strict monoidal cochain 2-categories}\label{subsub:inf syllep on symstrmon}
\begin{rem}\label{rem:symmetrically-sylleptic symmetric strict}
Following on from Remark \ref{rem:symstrict mon ChCat defi}, a symmetrically-sylleptic symmetric strict monoidal $\Ch^{[-1,0]}_R$-category $(\CC,\otimes,I,\gamma,\Sigma)$ is such that the endomorphic modification $\Sigma:\gamma\Rrightarrow\gamma$ is ``primitive" and intertwines with the symmetric braiding $\gamma$. For clarity's sake, let us specialise Definition \ref{def:syllepsis} to the present context:
\begin{enumerate}
\item We have a family of coexact homotopies $\Sigma_{UV}\in\CC[(UV),(VU)]^{-1}$, i.e.,
\begin{subequations}
\begin{equation}\label{eq:syllep on symstr is coexact}
\partial\Sigma_{UV}=0\quad.
\end{equation}
\item The family of coexact homotopies is natural, i.e., for all $f\in\CC[U,U']^0$ and $g\in\CC[V,V']^0$, 
\begin{equation}\label{eq:syllep on symstr is natural}
\Sigma_{U'V'}(f\otimes g)=(g\otimes f)\Sigma_{UV}\quad.
\end{equation}
\item The coexact homotopies are \textbf{strictly left factorisable}, i.e., for all $UVW\in\CC^{\,\boxtimes\,3}$, 
\begin{equation}\label{eq:strictly left factorisable}
\Sigma_{U(VW)}=(1_V\otimes\gamma_{UW})(\Sigma_{UV}\otimes1_W)+(1_V\otimes\Sigma_{UW})(\gamma_{UV}\otimes1_W)\quad.
\end{equation}
\item The coexact homotopies are \textbf{strictly right factorisable}, i.e., for all $UVW\in\CC^{\,\boxtimes\,3}$, 
\begin{equation}\label{eq:strictly right factorisable}
\Sigma_{(UV)W}=(\gamma_{UW}\otimes1_V)(1_U\otimes\Sigma_{VW})+(\Sigma_{UW}\otimes1_V)(1_U\otimes\gamma_{VW})\quad.
\end{equation}
\item The coexact homotopies are $\gamma$-equivariant, i.e., for all $UV\in\CC\boxtimes\CC$,
\begin{equation}
\Sigma_{UV}=\gamma_{UV}\Sigma_{VU}\gamma_{UV}\quad.
\end{equation}
\end{subequations}
\end{enumerate}
\end{rem}

\begin{defi}\label{def:infinitesimal syllepsis}
An \textbf{infinitesimal syllepsis} $T$ on a symmetric strict monoidal $\Ch^{[-1,0]}_R$-category $(\CC,\otimes,I,\gamma)$ consists of the following datum:
\begin{enumerate}
\item[(i)] An endomodification $T:\Id_\otimes\Rrightarrow\Id_\otimes:\otimes\Rightarrow\otimes:\CC\boxtimes\CC\rightarrow\CC$. Again, this requires:
\begin{itemize}
\item A family of coexact endohomotopies $T_{UV}\in\CC[(UV),(UV)]^{-1}$, i.e.,
\begin{subequations}
\begin{equation}\label{eq:infsyllep on symstr is coexact}
\partial T_{UV}=0\quad.
\end{equation}
\item The family of coexact endohomotopies is natural, i.e.,
\begin{equation}\label{eq:infsyllep on symstr is natural}
T_{U'V'}(f\otimes g)=(f\otimes g)T_{UV}\quad.
\end{equation}
\end{subequations}
\end{itemize}
\end{enumerate}
This datum is required to satisfy the following two axioms:
\begin{enumerate}
\item \textbf{Left infinitesimal factorisation}, i.e., for all $UVW\in\CC^{\,\boxtimes\,3}$, 
\begin{subequations}
\begin{equation}\label{eq:left inf factorisation}
T_{U(VW)}=T_{UV}\otimes1_W+(\gamma_{VU}\otimes1_W)(1_V\otimes T_{UW})(\gamma_{UV}\otimes1_W)\quad.
\end{equation}
\item \textbf{Right infinitesimal factorisation}, i.e., for all $UVW\in\CC^{\,\boxtimes\,3}$, 
\begin{equation}\label{eq:right inf factorisation}
T_{(UV)W}=1_U\otimes T_{VW}+(1_U\otimes\gamma_{WV})(T_{UW}\otimes1_V)(1_U\otimes\gamma_{VW})\quad.
\end{equation}
\end{subequations}
\end{enumerate}
An infinitesimal syllepsis is said to be \textbf{symmetric} if, for all $UV\in\CC\boxtimes\CC$,
\begin{equation}\label{eq:symmetric infinitesimal syllepsis}
T_{UV}=\gamma_{VU}T_{VU}\gamma_{UV}\qquad.
\end{equation}
\end{defi}
\begin{rem}\label{rem:infsyllep easily integrates}
Recall that an infinitesimal braiding $t:\otimes\Rightarrow\otimes$ on a symmetric strict monoidal $\Vec_\bbK$-category $(\CC,\otimes,I,\gamma)$ does not admit the simple integration as the braided \textit{strict} monoidal $\Vec_{\bbK[[h]]}$-category $(\CC[[h]],\otimes,I,\sigma:=\gamma\,e^{\frac{h}{2}t})$ because, in general,$$e^{\frac{h}{2}(t_{13}+t_{23})}\neq e^{\frac{h}{2}t_{13}}e^{\frac{h}{2}t_{23}}\qquad,\qquad e^{\frac{h}{2}(t_{13}+t_{12})}\neq e^{\frac{h}{2}t_{13}}e^{\frac{h}{2}t_{12}}\quad.$$On the other hand, an infinitesimal syllepsis $T:\Id_\otimes\Rrightarrow\Id_\otimes$ on a symmetric strict monoidal $\Ch^{[-1,0]}_R$-category $(\CC,\otimes,I,\gamma)$ \textit{does}, in fact, admit an even simpler integration as the sylleptic symmetric strict monoidal $\Ch^{[-1,0]}_R$-category $(\CC,\otimes,I,\gamma,\Sigma:=\gamma T)$. This is because the strict factorisability relations \eqref{eq:strictly left factorisable}/\eqref{eq:strictly right factorisable} are already ``linear" whereas the strict hexagon axioms \eqref{eq:strict Left hexagon axiom}/\eqref{eq:strict Right hexagon axiom} are ``grouplike" and the infinitesimal hexagon relations \eqref{t_U(VW)}/\eqref{t_(UV)W} are linearisations thereof. 
\end{rem}
\subsection{Syllepses from 3-shifted Poisson structures}\label{subs:syllep from 3-shifted}
Recall the context of \cite[Subsection 3.2]{Us}. 
\begin{rem}
We know from \cite[Remark 2.7]{Us} that a 3-shifted Poisson structure $\pi_{n=3}$ on a semi-free CDGA $A$ is such that the weight 2 term
\begin{equation}
\pi^{(2)}_{n=3}\in\Sym^2_A\left(\T_A[-3-1]\right)^{3+2}\cong\Sym^2_A\left(\T_A[-1]\right)^{-1}
\end{equation}
satisfies $\dd_{\widehat{\Pol}}(\pi^{(2)}_{n=3})=0$ thus the $A$-linear cochain map from \cite[(3.22)]{Us},
\begin{equation}
\langle\,\cdot\,,\,\cdot\,\rangle :\Sym_A^2\big(\T_{\! A}[-1]\big)\otimes_A \Omega_A[1]\otimes_A\Omega_A[1]
\to A\quad,
\end{equation}
induces the coexact $A$-linear homotopy $\langle\pi^{(2)}_{n=3},\,\cdot\,\rangle\in{}_A\CC\left[\Omega_A[1]\otimes_A\Omega_A[1],A\right]^{-1}$.
\end{rem}
\begin{constr}\label{con:syllep from 3-shifted}
Let us also recall the $A$-linear cochain maps $\xi_M:M\to\Omega_A[1]\otimes_AM$ from \cite[(3.16c)]{Us}. Though the family of such maps is not natural but \textit{pseudo}natural, the following family of coexact $A$-linear homotopies
\begin{flalign}\label{eqn:TfromPoissoncomponents1}
\begin{gathered}
\xymatrix@C=9.5em@R=3em{
M\otimes_A N \ar[d]_-{\xi_M\otimes_A \xi_N} \ar[r]^-{T_{MN}}~&~(M\otimes_AN)[-1] \\
\Omega_A[1]\otimes_A M\otimes_A \Omega_A[1]\otimes_A N  \ar[r]_-{1_{\Omega_A[1]}\otimes_A\gamma_{M,\Omega_A[1]}\otimes_A1_N}~&~
\Omega_A[1]\otimes_A \Omega_A[1]\otimes_A M\otimes_A N \ar[u]_-{\langle\pi^{(2)}_{n=3},\,\cdot\,\rangle\otimes_A1_{M\otimes_AN}}
}
\end{gathered}
\end{flalign}
is actually natural because of the coexactness of $\langle\pi^{(2)}_{n=3},\,\cdot\,\rangle\in{}_A\CC\left[\Omega_A[1]\otimes_A\Omega_A[1],A\right]^{-1}$ and the truncation on the homs. It can also be shown to satisfy the infinitesimal factorisation relations \eqref{eq:left inf factorisation}/\eqref{eq:right inf factorisation} and the symmetry relation \eqref{eq:symmetric infinitesimal syllepsis} by analogy with \cite[Proof of Theorem 3.10]{Us}. 
\end{constr}
Construction \ref{con:syllep from 3-shifted} explicitly describes the symmetric infinitesimal syllepsis $T$ induced on the symmetric strict monoidal $\Ch^{[-1,0]}_\bbK$-category $\left(\mathsf{Ho}_2({}_A\dgMod^\sf),\otimes_A,A,\gamma\right)$ by the weight 2 term $\pi^{(2)}_{n=3}$ of a 3-shifted Poisson structure on the semi-free CDGA $A$, its components' explicit basis expression look analogous to \cite[(3.29a)]{Us}. Remark \ref{rem:infsyllep easily integrates} tells us that this symmetric infinitesimal syllepsis $T$ easily integrates to the symmetric syllepsis $\Sigma:=\gamma T$ on the symmetric strict monoidal $\Ch^{[-1,0]}_\bbK$-category $\left(\mathsf{Ho}_2({}_A\dgMod^\sf),\otimes_A,A,\gamma\right)$.


\section{Infinitesimal 2-braidings and coboundary infinitesimal syllepses}\label{sec:inf2bra and coboundary syllepses}
This section works in the opposite direction to that of \cite[Subsection 3.1]{Us}. That is, rather than starting with a first-order deformation of the symmetric braiding \cite[(3.4)]{Us} and then proceeding to derive the imposed relations on the pseudonatural endomorphism $t:\otimes\Rightarrow\otimes$, we start with the definition of a ``strict" infinitesimal 2-braiding and then proceed to show, in the proof to Theorem \ref{theo:hexbrastrmon iff strBreen}, that such a condition is necessary and sufficient for the deformed braiding to satisfy the strict left/right hexagon axiom \eqref{eq:strict Left hexagon axiom}/\eqref{eq:strict Right hexagon axiom}. However, our notion of a ``coboundary infinitesimal syllepsis" does not follow this latter order, i.e., we start with the first-order deformation quantisation then ask what relations a symmetric syllepsis of the form $\Sigma=h\gamma T$ (on the aforementioned hexagonally-braided strict monoidal $\Ch^{[-1,0]}_R$-category) imposes on $T$.
\subsection{Infinitesimal 2-braidings on symmetric strict monoidal cochain 2-categories}\label{subs:notation for t}
\begin{defi}\label{def: strict t}
An \textbf{infinitesimal 2-braiding} on a symmetric strict monoidal $\Ch^{[-1,0]}_R$-category $(\CC,\otimes,I,\gamma)$ is a pseudonatural endomorphism $t:\otimes\Rightarrow\otimes:\CC\,\boxtimes\,\CC\rightarrow\CC$. We say it is \textbf{strict} if it satisfies both the \textbf{left infinitesimal hexagon relation},
\begin{subequations}
\begin{equation}\label{t_U(VW)}
t_{\id_\CC\,\boxtimes\,\otimes}\equiv\otimes(t\boxtimes1)+\otimes\left([\gamma^{-1}\boxtimes1](1\boxtimes t)_{\tau_{\CC,\CC}\,\boxtimes\,\id_\CC}[\gamma\boxtimes1]\right)\quad,
\end{equation}
and the \textbf{right infinitesimal hexagon relation},
\begin{equation}\label{t_(UV)W}
t_{\otimes\,\boxtimes\,\id_\CC}\equiv\otimes(1\boxtimes t)+\otimes\left([1\boxtimes\gamma^{-1}](t\boxtimes1)_{\id_\CC\,\boxtimes\,\tau_{\CC,\CC}}[1\boxtimes\gamma]\right)\quad.
\end{equation}
\end{subequations}
\end{defi}
As in \cite[Remark 16]{Joao}, we say a $\emph{strict}$ infinitesimal 2-braiding to contrast with the more general scenario of \eqref{t_U(VW)} and \eqref{t_(UV)W} being replaced with modifications which satisfy higher coherence conditions\footnote{We will see discussion of hexagonators become relevant at \emph{second} order in the deformation parameter $h$.}. We wish to keep the adjective ``strict" rather than absorb it into the name of ``infinitesimal 2-braiding" because it will be interesting throughout to note which lemmas and propositions depend on strictness and which do not.
\begin{rem}\label{rem: t_U(VW) explicitly}
Reading \eqref{t_U(VW)} as an equality between the cochain components we reproduce the familiar left infinitesimal hexagon relation, namely for $U,V,W\in\CC$,
\begin{subequations}\label{subeq:left inf hex relations explicitly}
\begin{equation}\label{eq:cochain left infhex}
t_{U(VW)}=t_{UV}\otimes1_W+(\gamma_{VU}\otimes1_W)(1_V\otimes t_{UW})(\gamma_{UV}\otimes1_W)
\end{equation}
and as an equality between the homotopy components, for $f\in\CC[U,U']^0$, $g\in\CC[V,V']^0$ and $k\in\CC[W,W']^0$, we find\footnote{Recall that $\gamma$ is a $\Ch^{[-1,0]}_R$-natural transformation hence its homotopy components are trivial, i.e., $\gamma_{f\,\boxtimes \,g}=0$.} something analogous
\begin{equation}\label{eq:homotopy left infhex}
t_{f\,\boxtimes \,(g\otimes k)}=t_{f\,\boxtimes \,g}\otimes k+(\gamma_{V'U'}\otimes1_{W'})(g\otimes t_{f\,\boxtimes \,k})(\gamma_{UV}\otimes1_W)\quad.
\end{equation}
\end{subequations}
Likewise, reading \eqref{t_(UV)W} as an equality between the cochain components we reproduce the familiar right infinitesimal hexagon relation, 
\begin{subequations}\label{subeq:right inf hex relations explicitly}
\begin{equation}\label{eq:cochain right infhex}
t_{(UV)W}=1_U\otimes t_{VW}+(1_U\otimes\gamma_{WV})(t_{UW}\otimes1_V)(1_U\otimes\gamma_{VW})
\end{equation}
and as an equality between the homotopy components, we find
\begin{equation}\label{eq:homotopy right infhex}
t_{(f\otimes g)\,\boxtimes \,k}=f\otimes t_{g\,\boxtimes \,k}+(1_{U'}\otimes\gamma_{W'V'})(t_{f\,\boxtimes \,k}\otimes g)(1_U\otimes\gamma_{VW})\quad.
\end{equation}
\end{subequations}
\end{rem}
\begin{defi}\label{def:sym t}
An infinitesimal 2-braiding $\otimes\xRightarrow{t}\otimes$ is said to be \textbf{symmetric} if it intertwines with the symmetric braiding $\gamma$, i.e.,
\begin{equation}\label{intertwining}
\gamma\,t\,\equiv\,t_{\tau_{\CC,\CC}}\gamma\quad.
\end{equation}
Plainly: 
\begin{subequations}
\begin{alignat}{2}
\label{intertwine single-index}\gamma_{UV}t_{UV}=\,&t_{VU}\gamma_{UV}\quad&&,\\\label{intertwine homotopy}
\gamma_{U'V'}t_{f\,\boxtimes \,g}=\,&t_{g\,\boxtimes \,f}\gamma_{UV}\quad&&.
\end{alignat}
\end{subequations}
\end{defi}
\begin{lem}\label{lem:sym t then half conditions}
If an infinitesimal 2-braiding $\otimes\xRightarrow{t}\otimes$ is symmetric then \eqref{t_U(VW)}$\iff$\eqref{t_(UV)W}.
\end{lem}
\begin{proof}
We simply show that \eqref{eq:cochain left infhex}$\implies$\eqref{eq:cochain right infhex},
\begin{align*}
t_{(UV)W}=&\gamma_{W(UV)}t_{W(UV)}\gamma_{(UV)W}\\=&(1_U\otimes\gamma_{WV})(\gamma_{WU}\otimes1_V)\left[t_{WU}\otimes1_V+(\gamma_{UW}\otimes1_V)(1_U\otimes t_{WV})(\gamma_{WU}\otimes1_V)\right](\gamma_{UW}\otimes1_V)(1_U\otimes\gamma_{VW})\\=&1_U\otimes t_{VW}+(1_U\otimes\gamma_{WV})(t_{UW}\otimes1_V)(1_U\otimes\gamma_{VW})\quad,
\end{align*}
where: the first equality comes from the involutivity of $\gamma$ and the symmetry property \eqref{intertwine single-index}, the second equality comes from the hexagon axiom that $\gamma$ satisfies (as in Remark \ref{rem:symstrict mon ChCat defi}) together with the cochain components of the left infinitesimal hexagon relation \eqref{eq:cochain left infhex} and the last equality uses the functoriality of $\otimes$ together with, again, the involutivity of $\gamma$ and the symmetry property \eqref{intertwine single-index}. Analogous arguments show \eqref{eq:homotopy left infhex}$\implies$\eqref{eq:homotopy right infhex} upon using the homotopy components of the symmetry property \eqref{intertwine homotopy}. Obviously, showing \eqref{subeq:right inf hex relations explicitly}$\implies$\eqref{subeq:left inf hex relations explicitly} works the same way.
\end{proof}
\begin{defi}\label{def:Breen t}
An infinitesimal 2-braiding $\otimes\xRightarrow{t}\otimes$ is \textbf{Breen} if it satisfies
\begin{equation}\label{eq:Breen t}
\gamma_{U(WV)}t_{1_U\,\boxtimes\,\gamma_{VW}}+\gamma_{(VU)W}t_{\gamma_{UV}\,\boxtimes\,1_W}=0\quad.
\end{equation}
\end{defi}
\begin{ex}\label{ex:Poisson inf 2bra is Breen}
It is clear from \cite[(3.28) and (3.29)]{Us} that those symmetric strict infinitesimal 2-braidings induced by 2-shifted Poisson structures are also Breen.
\end{ex}
\begin{theo}\label{theo:hexbrastrmon iff strBreen}
An infinitesimal 2-braiding $t$ provides a first-order deformation quantisation of the symmetric strict monoidal $\Ch^{[-1,0]}_\bbK$-category $(\CC,\otimes,I,\gamma)$ to a hexagonally braided strict monoidal $\Ch^{[-1,0]}_{\bbK[h]/(h^2)}$-category 
\begin{equation}
\left(\CC_{\bbK[h]/(h^2)},\otimes,I,\sigma:\equiv\gamma[1+\tfrac{h}{2}t]\right)
\end{equation}
if and only if $t$ is a strict Breen infinitesimal 2-braiding.
\end{theo}
\begin{proof}
Let us show that the order $h$ relation of the strict left hexagon axiom \eqref{eq:strict Left hexagon axiom} is equivalent to the left infinitesimal hexagon relation \eqref{t_U(VW)}. The order $h$ relation of the strict left hexagon axiom \eqref{eq:strict Left hexagon axiom} reads as
\begin{subequations}
\begin{equation}
(\gamma t)_{\id_{\CC}\,\boxtimes\,\otimes}\,\equiv\,\otimes\left([1\boxtimes\gamma]_{\tau_{\mathbf{C,C}}\,\boxtimes\,\id_\CC}(\gamma t\boxtimes1)\right)+\otimes\left([1\boxtimes\gamma t]_{\tau_{\mathbf{C,C}}\,\boxtimes\,\id_\CC}(\gamma\boxtimes1)\right)\quad.
\end{equation}
Let us use Remark \ref{rem:whiskering and exchanger} and the fact that $\boxtimes$ is a 3-functor (as in Remark \ref{rem:boxtimes is a strictly associative 3-functor}) to rewrite this as
\begin{equation}
\gamma_{\id_{\CC}\,\boxtimes\,\otimes}\,t_{\id_{\CC}\,\boxtimes\,\otimes}\,\equiv\,\otimes\left([1\boxtimes\gamma]_{\tau_{\mathbf{C,C}}\,\boxtimes\,\id_\CC}(\gamma\boxtimes1)(t\boxtimes1)\right)+\otimes\left([1\boxtimes\gamma]_{\tau_{\mathbf{C,C}}\,\boxtimes\,\id_\CC}[1\boxtimes t]_{\tau_{\mathbf{C,C}}\,\boxtimes\,\id_\CC}(\gamma\boxtimes1)\right)\quad.
\end{equation}
Using Remark \ref{rem:whiskering and exchanger} again, we rewrite the RHS as 
\begin{equation}
\otimes\left([1\boxtimes\gamma]_{\tau_{\mathbf{C,C}}\,\boxtimes\,\id_\CC}(\gamma\boxtimes1)\right)\otimes(t\boxtimes1)+\otimes[1\boxtimes\gamma]_{\tau_{\mathbf{C,C}}\,\boxtimes\,\id_\CC}\otimes\left([1\boxtimes t]_{\tau_{\mathbf{C,C}}\,\boxtimes\,\id_\CC}(\gamma\boxtimes1)\right)\quad.
\end{equation}
Finally, we use the fact that $\gamma$ satisfies the strict left hexagon axiom 
\begin{equation}
\gamma_{\id_{\CC}\,\boxtimes\,\otimes}\,\equiv\,\otimes\left([1\boxtimes\gamma]_{\tau_{\mathbf{C,C}}\,\boxtimes\,\id_\CC}(\gamma\boxtimes1)\right)\,\equiv\,\otimes[1\boxtimes\gamma]_{\tau_{\mathbf{C,C}}\,\boxtimes\,\id_\CC}\otimes(\gamma\boxtimes1)
\end{equation}
\end{subequations}
together with, again, the functoriality of whiskering as in Remark \ref{rem:whiskering and exchanger}. Analogously, one shows that the strict right hexagon axiom \eqref{eq:strict Right hexagon axiom} is equivalent to the right infinitesimal hexagon relation \eqref{t_(UV)W}. Lastly, the Breen property \eqref{eq:Breen t} of the infinitesimal 2-braiding is obviously equivalent to the order $h$ relation coming from the strict hexagonal Breen axiom \eqref{eqn:strict Breen axiom}.
\end{proof}
\begin{rem}
The braiding $\sigma:\equiv\gamma[1+\tfrac{h}{2}t]$ is not necessarily involutive because, modulo $h^2$,
\begin{subequations}
\begin{equation}
\sigma_{VU}\sigma_{UV}-1_{(UV)}=(\gamma_{VU}+\tfrac{h}{2}\gamma_{VU}t_{VU})(\gamma_{UV}+\tfrac{h}{2}\gamma_{UV}t_{UV})-1_{(UV)}=\tfrac{h}{2}(t_{UV}+\gamma_{VU}t_{VU}\gamma_{UV})
\end{equation}
is, in general, non-zero unless $t$ is \textbf{skew-symmetric}, i.e.,
\begin{equation}
t+(\gamma\,t)_{\tau_{\CC,\CC}}\gamma=0\quad.
\end{equation}
\end{subequations}
In particular, if $t$ is a nontrivial \textit{symmetric} strict Breen infinitesimal 2-braiding then the deformed braiding is necessarily \textit{not} involutive.
\end{rem}
\subsubsection{Coboundary syllepses from coboundary 2-shifted Poisson structures}\label{subsub:coboundary inf syllep}
The following remark is a specification of Definition \ref{def:syllepsis} and makes use of the formulae for the cochain/homotopy components of an inverse pseudonatural isomorphism \eqref{subeq:inverse pseudo iso formulae}.
\begin{rem}\label{rem:coboundary 2-shifted induces syllep}
Given a symmetric strict Breen infinitesimal 2-braiding $t$, consider the hexagonally braided strict monoidal $\Ch^{[-1,0]}_{\bbK[h]/(h^2)}$-category $\left(\CC_{\bbK[h]/(h^2)},\otimes,I,\sigma:\equiv\gamma[1+\tfrac{h}{2}t]\right)$ of Theorem \ref{theo:hexbrastrmon iff strBreen}. A symmetric syllepsis $\Sigma:=h\gamma T$ is such that $T$ must satisfy the five relations:
\begin{enumerate}
\item For all $UV\in\CC\boxtimes\CC$, 
\begin{subequations}
\begin{equation}\label{eq:syllep on hexbra is homot}
\partial T_{UV}=t_{UV}\quad.
\end{equation}
\item For all $f\in\CC[U,U']^0$ and $g\in\CC[V,V']^0$, 
\begin{equation}\label{eq:syllep on hexbra is quasinatural}
(f\otimes g)T_{UV}-T_{U'V'}(f\otimes g)=t_{f\,\boxtimes\,g}\quad.
\end{equation}
\item For all $UVW\in\CC^{\,\boxtimes\,3}$, 
\begin{equation}\label{hexbra left factorise}
T_{U(VW)}=T_{UV}\otimes1_W+(\gamma_{VU}\otimes1_W)(1_V\otimes T_{UW})(\gamma_{UV}\otimes1_W)\quad.
\end{equation}
\item For all $UVW\in\CC^{\,\boxtimes\,3}$, 
\begin{equation}\label{eq:hexbra right factorise}
T_{(UV)W}=1_U\otimes T_{VW}+(1_U\otimes\gamma_{WV})(T_{UW}\otimes1_V)(1_U\otimes\gamma_{VW})\quad.
\end{equation}
\item For all $UV\in\CC\boxtimes\CC$,
\begin{equation}\label{eq:symmetric syllepsis on hexbra}
T_{UV}=\gamma_{VU}T_{VU}\gamma_{UV}\qquad.
\end{equation}
\end{subequations}
\end{enumerate}
Because of \eqref{eq:syllep on hexbra is homot}, we refer to $T:t\Rrightarrow0$ as a \textbf{coboundary infinitesimal syllepsis}.
\end{rem}
\begin{constr}\label{con:coboundary infinitesimal syllepsis}
Following Example \ref{ex:Poisson inf 2bra is Breen}, suppose that the symmetric strict Breen infinitesimal 2-braiding $t$ is induced by a weight 2 term $\pi^{(2)}_{n=2}$ of a 2-shifted Poisson structure which is itself coboundary. In other words, suppose there exists $\tilde{\pi}^{(2)}_{n=2}\in\Sym^2_A\left(\T_A[-1]\right)^{-1}$ such that $\dd_{\widehat{\Pol}}(\tilde{\pi}^{(2)}_{n=2})=\pi^{(2)}_{n=2}$ then we can construct $T$ as in \eqref{eqn:TfromPoissoncomponents1} but now with $\pi_{n=3}^{(2)}$ replaced by $\tilde{\pi}^{(2)}_{n=2}$. The only relation in Remark \ref{rem:coboundary 2-shifted induces syllep} which remains to be checked is the ``quasinaturality" \eqref{eq:syllep on hexbra is quasinatural},
\begin{align}
T_{M'N'}(f\otimes_Ag)=&(\iota_{\tilde{\pi}^{(2)}_{n=2}})_{M'N'}(1\otimes_A\gamma_{M',\Omega_A[1]}\otimes_A1)(\xi_{M'}f\otimes_A \xi_{N'}g)\nn\\=&(\iota_{\tilde{\pi}^{(2)}_{n=2}})_{M'N'}(1\otimes_A\gamma_{M',\Omega_A[1]}\otimes_A1)\left(\big[(1_{\Omega_A[1]}\otimes_Af)\xi_M-\partial\xi_f\big]\otimes_A\big[(1_{\Omega_A[1]}\otimes_Ag)\xi_N-\partial\xi_g\big]\right)\nn\\=&(f\otimes_Ag)T_{MN}-\nn\\&(\iota_{\pi^{(2)}_{n=2}})_{M'N'}(1\otimes_A\gamma_{M',\Omega_A[1]}\otimes_A1)\left(\xi_f\otimes_A[1_{\Omega_A[1]}\otimes_Ag]\xi_N+\left[(1_{\Omega_A[1]}\otimes_Af)\xi_M-\partial\xi_f\right]\otimes_A\xi_g\right)\nn\\=&(f\otimes_Ag)T_{MN}-t_{f\,\boxtimes\,g}\qquad,
\end{align}
where the second equality uses the pseudonaturality of $\xi$, the third equality uses the hom-truncation and the last equality uses the definition \cite[(3.25b)]{Us} of the homotopy components of $t$.
\end{constr}
\subsection{Associativity of the monoidal product and subscript index notation}\label{subsec:index notation}
The following index notation follows that of \cite[(3.32)]{Us}.
\begin{rem}\label{rem: shorthand t notation}
The following point can not be stressed enough, the monoidal product $\otimes$ will \textit{always} be associative $\otimes(\otimes\boxtimes\id_\CC)=\otimes(\id_\CC\boxtimes\otimes)$; we are deforming the associator to a nontrivial pseudonatural \textit{auto}morphism. Thus, $\otimes^n:\CC^{\boxtimes(n+1)}\rightarrow\CC$ will always be well-defined $\forall n\in\bbN$ and this permits the following index notation:
\begin{enumerate}
\item We write things such as:
\begin{subequations}
\begin{equation}
t_{12}:\equiv\otimes(t\,\boxtimes\,\Id_{\id_\CC})\quad,\quad\alpha_{1(23)4}:\equiv\alpha_{\id_\CC\,\boxtimes\,\otimes\,\boxtimes\,\id_\CC}\quad,\quad\sigma_{13}:\equiv\otimes(\Id_{\id_\CC}\boxtimes\,\sigma)_{\tau_{\CC,\CC}\,\boxtimes\,\id_\CC}\,.
\end{equation}
\item The fact that $\gamma$ is natural, involutive and satisfies the hexagon axiom means that we have, for example, 
\begin{equation}
t_{13}:\equiv\otimes\left([\gamma^{-1}\boxtimes1](1\boxtimes t)_{\tau_{\CC,\CC}\,\boxtimes\,\id_\CC}[\gamma\boxtimes1]\right)\equiv\otimes\left([1\boxtimes\gamma^{-1}](t\boxtimes1)_{\id_\CC\,\boxtimes\,\tau_{\CC,\CC}}[1\boxtimes\gamma]\right)\,.
\end{equation}
\item\label{item:strict t is primitive} If the infinitesimal 2-braiding is strict then this means, for example,
\begin{equation}
t_{1(23)}:\equiv t_{\id_\CC\,\boxtimes\,\otimes}\equiv t_{12}+t_{13}\quad.
\end{equation}
In general, for $k\neq i<j\neq k$, we would have: 
\begin{equation}\label{eq:strict t notation}
t_{k(ij)}\equiv t_{ki}+t_{kj}\qquad,\qquad t_{(ij)k}\equiv t_{ik}+t_{jk}\quad.
\end{equation}
\item\label{item:symmetric t index swap} If $t$ is symmetric then this means, for example,
\begin{equation}
t_{12}:\equiv t\equiv\gamma^{-1}t_{\tau_{\CC,\CC}}\gamma\equiv:t_{21}\quad.
\end{equation}
In general, for $k\neq i\neq j\neq k$, we would have: 
\begin{equation}\label{eq:sym t notation}
t_{ij}\equiv t_{ji}\qquad,\qquad t_{(ij)k}\equiv t_{k(ij)}\quad.
\end{equation}
\end{subequations}
\end{enumerate}
\end{rem}
This subscript index notation allows us to rewrite: the pentagonator \eqref{eq:Pentagonator} as
\begin{subequations}
\begin{equation}\label{eq:index pentagonator}
\Pi:\alpha_{234}\alpha_{1(23)4}\alpha_{123}\Rrightarrow\alpha_{12(34)}\alpha_{(12)34}\quad,
\end{equation}
the left hexagonator \eqref{eq:left hexagonator} as
\begin{equation}\label{eq:index left hexagonator}
\H^L:\alpha_{231}\sigma_{1(23)}\alpha\Rrightarrow\sigma_{13}\alpha_{213}\sigma_{12}\quad,
\end{equation}
and the right hexagonator \eqref{eq:right hexagonator} as
\begin{equation}\label{eq:index right hexagonator}
\H^R:\alpha^{-1}_{312}\sigma_{(12)3}\alpha^{-1}\Rrightarrow\sigma_{13}\alpha^{-1}_{132}\sigma_{23}\quad.
\end{equation}
\end{subequations}
\begin{constr}
\begin{subequations}\label{subeq:pre symstr pre-hexes}
Choosing: the Drinfeld series\footnote{From the Knizhnik-Zamolodchikov connection, see \cite[Theorem 20]{BRW}} as our ansatz associator $\alpha\equiv\Phi_\KZ(t_{12},t_{23})$, the braiding as $\sigma\equiv\gamma e^{\frac{h}{2}t}$, the left hexagonator as $\H^L=\gamma_{1(23)}\LL$ and the right hexagonator as $\H^R=\gamma_{(12)3}\RR$, we rewrite: the pentagonator \eqref{eq:index pentagonator} as
\begin{equation}\label{eq:ansatz pentagonator}
\Pi:\Phi(t_{23},t_{34})\Phi(t_{1(23)},t_{(23)4})\Phi(t_{12},t_{23})\Rrightarrow\Phi(t_{12},t_{2(34)})\Phi(t_{(12)3},t_{34})\quad,
\end{equation}
the left hexagonator \eqref{eq:index left hexagonator} as the \textbf{left pre-hexagonator},
\begin{equation}\label{eq:ansatz left hexagonator}
\LL:\Phi(t_{23},t_{31})e^{\frac{h}{2}t_{1(23)}}\Phi(t_{12},t_{23})\Rrightarrow e^{\frac{h}{2}t_{13}}\Phi(t_{21},t_{13})e^{\frac{h}{2}t_{12}}\quad,
\end{equation}
and the right hexagonator \eqref{eq:index right hexagonator} as\footnote{Note that $\Phi_\KZ(A,B)=\Phi_\KZ(B,A)^{-1}$, see \cite[(2.24)]{BRW}.} the \textbf{right pre-hexagonator},
\begin{equation}\label{eq:ansatz right hexagonator}
\RR:\Phi(t_{12}\,,t_{31})e^{\frac{h}{2}t_{(12)3}}\Phi(t_{23}\,,t_{12})\Rrightarrow e^{\frac{h}{2}t_{13}}\Phi(t_{32}\,,t_{13})e^{\frac{h}{2}t_{23}}\quad.
\end{equation}
\end{subequations}
\end{constr}
\begin{rem}
If the infinitesimal 2-braiding $t$ is strict and symmetric then we use \eqref{eq:strict t notation} and \eqref{eq:sym t notation}, respectively, to rewrite \eqref{subeq:pre symstr pre-hexes} as:
\begin{subequations}
\begin{alignat}{3}
\Pi:\Phi(t_{23},t_{34})\Phi(t_{12}+t_{13},t_{24}+t_{34})\Phi(t_{12},t_{23})&\Rrightarrow\Phi(t_{12},t_{23}+t_{24})\Phi(t_{13}+t_{23},t_{34})\quad&&,\label{eq:symstr pentagonator}\\
\LL:\Phi(t_{23},t_{13})e^{\frac{h}{2}(t_{12}+t_{13})}\Phi(t_{12},t_{23})&\Rrightarrow e^{\frac{h}{2}t_{13}}\Phi(t_{12},t_{13})e^{\frac{h}{2}t_{12}}\quad&&,\label{eq:symstr left pre-hex}\\
\RR:\Phi(t_{12}\,,t_{13})e^{\frac{h}{2}(t_{13}+t_{23})}\Phi(t_{23}\,,t_{12})&\Rrightarrow e^{\frac{h}{2}t_{13}}\Phi(t_{23}\,,t_{13})e^{\frac{h}{2}t_{23}}\quad&&.\label{eq:symstr right pre-hex}
\end{alignat}
\end{subequations}
\end{rem}
\begin{rem}\label{rem:pre-hex index swap not equal yet}
For a symmetric strict $t$, it might look as if we can define $\LL:=\RR_{321}$ however we will show in Remark \ref{rem:pre-hex index swap almost equal} that this fails (already at order $h^2$) unless $t$ satisfies a stronger condition (that of ``total symmetry", see Definition \ref{def: total sym}).
\end{rem}


\section{Second-order deformation quantisation}\label{sec:2nd order}
This section begins by constructing 4-term relationators, left \eqref{eqn: left 4 term} and right \eqref{eq:R=t_(t_23)}, as specific instances of the exchanger \eqref{eqn:compositioncoherences}. Assuming a symmetric strict $t$, we then construct candidate infinitesimal hexagonators \eqref{eq:inf hexagonators} from such 4-term relationators. Assuming a strict $t$, Subsubsection \ref{subsub:pent still trivial} shows that the pentagonator is still trivial at this order in $\hbar$. Again, assuming a strict $t$, Subsubsection \ref{subsub:totally symmetric} shows that CM's notion of ``total symmetry" \cite[Definition 18]{Joao} is a sufficient condition for the tetrahedra and hexahedron axioms of Definition \ref{def:bra str un mon cat} to be satisfied. Assuming a totally symmetric strict $t$, Proposition \ref{propo: h^2 Breen polytope} shows that CM's notion of ``coherency" \cite[Definition 17]{Joao} is a necessary and sufficient condition for the Breen polytope axiom \eqref{eqn:Breen axiom} to be satisfied. 

\subsection{Four-term relationators}\label{subsec:4T relationators}
We use the index notation of Remark \ref{rem: shorthand t notation} throughout the rest of this paper.
\begin{defi}
A \textbf{left 4-term relationator} is an endomodification of the following form,
\begin{equation}\label{eq:def of left 4T relationator}
\L:[t_{12}\,,t_{(12)3}]\Rrightarrow0: \otimes^2\Rightarrow\otimes^2:
\CC^{\,\boxtimes\,3}\rightarrow\CC\quad.
\end{equation}
\end{defi}
The following remark will also be used throughout the rest of this paper.
\begin{rem}\label{rem:overhat notation for mods}
Given some modification $\Xi:\xi\Rrightarrow\xi'$, we refer to Construction \ref{con:add mods} to define a new modification $\widehat{\Xi}:=\Xi-\ID_{\xi'}:\xi-\xi'\Rrightarrow0$ which obviously has the same components.
\end{rem}
\begin{propo}\label{propo:Gamma4term}
The specific instance of the exchanger 
\begin{subequations}
\begin{equation}
\begin{tikzcd}
	& {} && {} \\
	{\CC\boxtimes\CC\boxtimes\CC} && {\CC\boxtimes\CC} && {\CC} \\
	& {} && {}
	\arrow["{\otimes\,\boxtimes\,\id_\CC}", curve={height=-30pt}, from=2-1, to=2-3]
	\arrow["{\otimes\,\boxtimes\,\id_\CC}"', curve={height=30pt}, from=2-1, to=2-3]
	\arrow[""{name=0, anchor=center, inner sep=0}, "{\otimes\,\boxtimes\,\id_\CC}"{description}, from=2-1, to=2-3]
	\arrow["{\otimes}", curve={height=-30pt}, from=2-3, to=2-5]
	\arrow["{\otimes}"', curve={height=30pt}, from=2-3, to=2-5]
	\arrow[""{name=1, anchor=center, inner sep=0}, "{\otimes}"{description}, from=2-3, to=2-5]
	\arrow["{\Id_{\otimes\,\boxtimes\,\id_\CC}\,}"', shorten <=7pt, shorten >=7pt, Rightarrow, from=1-2, to=0]
	\arrow["t\,"', shorten <=7pt, shorten >=7pt, Rightarrow, from=1-4, to=1]
	\arrow["{t\,\boxtimes\,\Id_{\id_\CC}\,}"', shorten <=7pt, shorten >=7pt, Rightarrow, from=0, to=3-2]
	\arrow["\Id_\otimes\,"', shorten <=7pt, shorten >=7pt, Rightarrow, from=1, to=3-4]
\end{tikzcd}\quad,
\end{equation}
i.e.,
\begin{equation}
\ast^2_{(\Id_\otimes,t\,\boxtimes\,\Id_{\id_\CC}),(t,\Id_{\otimes\,\boxtimes\,\id_\CC})}:\otimes(t\boxtimes1)\,t_{\otimes\,\boxtimes\,\id_\CC}\Rrightarrow t*(t\boxtimes1):\otimes^2\Rightarrow\otimes^2:
\CC^{\,\boxtimes\,3}\rightarrow\CC\quad,
\end{equation}
\end{subequations}
encodes a left 4-term relationator \eqref{eq:def of left 4T relationator}.
\end{propo}
\begin{proof}
Observe that
\begin{equation}
[t_{12}\,,t_{(12)3}]:\equiv t_{12}t_{(12)3} -t_{(12)3}t_{12}\equiv\otimes(t\boxtimes1)\,t_{\otimes\,\boxtimes\,\id_\CC}-t_{\otimes\,\boxtimes\,\id_\CC}\otimes(t\boxtimes1)\equiv\otimes(t\boxtimes1)\,t_{\otimes\,\boxtimes\,\id_\CC}-t*(t\boxtimes1)\,,
\end{equation}
where the last equality comes from the condition for triviality of the exchanger in Remark \ref{rem:whiskering and exchanger}. We now use the overhat notation from Remark \ref{rem:overhat notation for mods} to define
\begin{equation}\label{eq:left 4 term as overhat asterisk}
\L:=\hat{\ast}^2_{(\Id_\otimes,t\,\boxtimes\,\Id_{\id_\CC}),(t,\Id_{\otimes\,\boxtimes\,\id_\CC})}\quad.
\end{equation}
\end{proof}
\begin{rem}
We can express the components of \eqref{eq:left 4 term as overhat asterisk} by using the formula \eqref{eqn:compositioncoherences}, 
\begin{equation}\label{eqn: left 4 term}
\L_{UVW}=\left(\hat{\ast}^2_{(\Id_\otimes,t\,\boxtimes\,\Id_{\id_\CC}),(t,\Id_{\otimes\,\boxtimes\,\id_\CC})}\right)_{UVW}=t_{t_{UV}\,\boxtimes \,1_W}\quad.    
\end{equation}
Now we can appeal to either pseudonaturality of $t$ as in \eqref{eqn: dubindex ind as homotopy} or the datum of a modification as in \eqref{eqn: mod is homot between pseudos} to describe how the left 4-term relation is deformed in this instance,
\begin{equation}
(t_{UV}\otimes1_W)t_{(UV)W}-t_{(UV)W}(t_{UV}\otimes1_W)=\partial t_{t_{UV}\,\boxtimes \,1_W}\quad.
\end{equation}
Notice that the benefit of regarding the left 4-term relationator as a specific instance of the exchanger is that, in order to show that the components \eqref{eqn: left 4 term} constitute a modification, we do \textit{not} need to show that the condition \eqref{eqn:mod single condition} holds (cf. \cite[Proposition 3.13]{Us}) given that we \textit{already} know the exchanger to be a modification in general. 
\end{rem}
Because of \eqref{eqn: left 4 term}, we can compactly write 
\begin{equation}
\L=t_{t_{12}}:[t_{12}\,,t_{(12)3}]\Rrightarrow0\quad.
\end{equation}
Of course, we also have a \textbf{right 4-term relationator}, 
\begin{equation}\label{eq:R=t_(t_23)}
\R:=t_{t_{23}}=\hat{\ast}^2_{(\Id_\otimes,\Id_{\id_\CC}\boxtimes\,t),(t,\Id_{\id_\CC\,\boxtimes\,\otimes})}:[t_{23}\,,t_{1(23)}]\Rrightarrow0\quad.
\end{equation}
\begin{lem}\label{lem:sym t reduces independence of L/R}
If the infinitesimal 2-braiding $t$ is symmetric then: 
\begin{equation}
\L=\R_{312}\,\,\,,\,\,\,\L_{213}=\R_{321}\,\,\,,\,\,\,
\R=\L_{231}\,\,\,,\,\,\,\R_{132}=\L_{321}\,\,\,,\,\,\,
\L_{132}=\R_{213}\,\,\,,\,\,\,\L_{312}=\R_{231}\,.
\end{equation}
\end{lem}
\begin{proof}
Let us demonstrate the fifth relation $\L_{132}=\R_{213}$, i.e.,
\begin{subequations}
\begin{equation}
(1_U\otimes\,\gamma_{WV})\L_{UWV}(1_U\otimes\,\gamma_{VW})=(\gamma_{VU}\otimes1_W)\R_{VUW}(\gamma_{UV}\otimes1_W)\quad.
\end{equation}
Using the fact that $\gamma$ is involutive and satisfies the hexagon axiom, we rewrite this as 
\begin{equation}
\L_{UWV}=\gamma_{V(UW)}\R_{VUW}\gamma_{(UW)V}\qquad\implies\qquad t_{t_{UW}\,\boxtimes\,1_V}=\gamma_{V(UW)}t_{1_V\,\boxtimes\,t_{UW}}\gamma_{(UW)V}\quad,
\end{equation}
\end{subequations}
but the latter equality is a special case of \eqref{intertwine homotopy}. The proofs for the rest of the above relations are similar.
\end{proof}
\begin{subequations}
\begin{lem}
If the infinitesimal 2-braiding is strict then a third 4-term relationator is given by
\begin{equation}
\L+\R=t_{t_{12}+t_{23}}:[t_{12}+t_{23}\,,t_{13}]\Rrightarrow0\quad.
\end{equation}
\end{lem}
\begin{proof}
Applying Construction \ref{con:add mods} and Item \ref{item:strict t is primitive} of Remark \ref{rem: shorthand t notation} tells us that 
\begin{equation}
\L+\R:[t_{12}\,,t_{(12)3}]+[t_{23}\,,t_{1(23)}]\equiv[t_{12}\,,t_{13}+t_{23}]+[t_{23}\,,t_{12}+t_{13}]\equiv[t_{12}+t_{23}\,,t_{13}]\Rrightarrow0\quad.
\end{equation}
\end{proof}
\end{subequations}

\subsubsection{Infinitesimal hexagonators}\label{subsub:inf hex}
\begin{constr}\label{con:inf hex}
We expand the \textbf{pre-braiding} $\gamma^{-1}\sigma:\equiv e^{\frac{h}{2}t}$ to order $h^2$ and consult \cite[Corollary XIX.6.5]{Kassel} for the order $h^2$ expansion of the Drinfeld associator series $\Phi(t_{12}\,,t_{23})$ :
\begin{subequations}
\begin{alignat}{2}
\Phi(t_{12}\,,t_{23})&\equiv1+\frac{h^2}{24}[t_{12}\,,t_{23}]+\O(h^3)\quad&&,\label{eqn:Drinfeld series up to h^2}\\
e^{\frac{h}{2}t}&\equiv1+\frac{h}{2}t+\frac{h^2}{8}t^2+\O(h^3)\quad&&.\label{eq:braiding up to h^2}
\end{alignat}
\end{subequations}
Assuming a symmetric strict infinitesimal 2-braiding, substituting the above ansatz into \eqref{eq:symstr left pre-hex} and taking into account only order $h^2$ terms, we have
\begin{subequations}
\begin{equation}
\frac{1}{24}[t_{23}\,,t_{13}]+\frac{1}{8}(t_{12}+t_{13})^2+\frac{1}{24}[t_{12}\,,t_{23}]\Rrightarrow\frac{1}{8}t_{13}^2+\frac{1}{4}t_{13}t_{12}+\frac{1}{8}t_{12}^2+\frac{1}{24}[t_{12}\,,t_{13}]\quad.
\end{equation}
Balancing terms as in Remark \ref{rem:overhat notation for mods}, we have the following constraint
\begin{equation}
\frac{1}{24}(2\L+\R):\frac{1}{24}([t_{12}+t_{23}\,,t_{13}]+[t_{12}\,,t_{(12)3}])\Rrightarrow0:\otimes^2\Rightarrow\otimes^2:\CC^{\,\boxtimes\,3}\rightarrow\CC\quad.
\end{equation}
\end{subequations}
In other words, we have a candidate left hexagonator
\begin{subequations}
\begin{equation}\label{eqn:candidate left hexagonator}
\H^L:=\frac{h^2}{24}\gamma_{1(23)}(2\L+\R)+\O(h^3)
\end{equation}
and, of course, we also have a candidate right hexagonator given by 
\begin{equation}\label{eqn:candidate right hexagonator}
\H^R:=\frac{h^2}{24}\gamma_{(12)3}(\L+2\R)+\O(h^3)\quad.
\end{equation}
\end{subequations}
We denote the \textbf{left/right infinitesimal hexagonator}, respectively, by: 
\begin{equation}\label{eq:inf hexagonators}
\hh^L:=2\L+\R\qquad,\qquad\hh^R:=\L+2\R\quad.
\end{equation}
\end{constr}
\begin{rem}\label{rem:pre-hex index swap almost equal}
Following on from Remark \ref{rem:pre-hex index swap not equal yet}, Lemma \ref{lem:sym t reduces independence of L/R} gives us, for a symmetric strict $t$, 
\begin{equation}
\hh^R_{321}=\L_{321}+2\R_{321}=\R_{132}+2\L_{213}\neq\hh^L\quad.
\end{equation}
\end{rem}

\subsection{Verifying the braided monoidal cochain 2-category axioms are satisfied}\label{subsec:verifying axioms}
We begin, in Subsubsection \ref{subsub:pent still trivial}, by showing that strictness of the infinitesimal 2-braiding $t$ implies that the pentagonator is still trivial at this order and also implies that the Drinfeld associator series $\Phi_\KZ(t_{12},t_{23})$ satisfies the triangle axiom \eqref{eq:triangle as cochain and homotopy} to all orders.
\subsubsection{The pentagonator is still trivial at this order}\label{subsub:pent still trivial}
\begin{lem}\label{lem:pent axiom upheld}
Given a strict infinitesimal 2-braiding, the ansatz associator 
\begin{equation}
\alpha:\equiv\Phi(t_{12}\,,t_{23})\equiv1+\frac{h^2}{24}[t_{12}\,,t_{23}]
\end{equation}
satisfies the pentagon axiom.
\end{lem}
\begin{proof}
Referring to \eqref{eq:symstr pentagonator},
\begin{subequations}
\begin{equation}
[t_{12}\,,t_{23}+t_{24}]+[t_{13}+t_{23}\,,t_{34}]-[t_{23}\,,t_{34}]-[t_{12}+t_{13}\,,t_{24}+t_{34}]-[t_{12}\,,t_{23}]\Rrightarrow0
\end{equation}
simply boils down to
\begin{equation}
-[t_{12}\,,t_{34}]-[t_{13}\,,t_{24}]\Rrightarrow0
\end{equation}
but the LHS is 0 for the usual reason. More explicitly, 
\begin{align}
[t_{12}\,,t_{34}]&\equiv\otimes(t\boxtimes\Id_{\otimes\,})\otimes(\Id_{\otimes\,}\boxtimes t)-\otimes(\Id_{\otimes\,}\boxtimes t) \otimes(t\boxtimes\Id_{\otimes\,})\nn\\&\equiv\otimes\left([t\boxtimes\Id_{\otimes\,}] [\Id_{\otimes\,}\boxtimes t]-[\Id_{\otimes\,}\boxtimes t] [t\boxtimes\Id_{\otimes\,}]\right)\nn\\&\equiv\otimes\left([t\boxtimes t]-[t\boxtimes t]\right)\nn\\&\equiv0
\end{align}
where the second equality comes from the functoriality of whiskering (as in Remark \ref{rem:whiskering and exchanger}) and the third equality comes from the fact that the monoidal composition $\boxtimes$ is a 3-functor (as mentioned in Remark \ref{rem:boxtimes is a strictly associative 3-functor}) hence the $(\boxtimes\,,\circ)$-exchange law is upheld. Also, we have the expression for $t_{24}t_{13}$,
\begin{equation}
\otimes\left(\left[\otimes(1\boxtimes\gamma^{-1})\boxtimes1\right]\left[([\Id_\otimes\boxtimes t][t\boxtimes\Id_\otimes])_{\id_\CC\,\boxtimes\,\tau_{\CC,\CC}\,\boxtimes\,\id_\CC}\right]\left[\otimes(1\boxtimes\gamma)\boxtimes1\right]\right)
\end{equation}
\end{subequations}
but, as before, we have $(\Id_\otimes\boxtimes t) (t\boxtimes\Id_\otimes)\equiv(t\boxtimes\Id_\otimes) (\Id_\otimes\boxtimes t)$ hence $t_{24}t_{13}\equiv t_{13}t_{24}$.
\end{proof}
Lemma \ref{lem:pent axiom upheld} allows us to choose $\Pi=\mathbf{0}$ thus the associahedron axiom \eqref{eq:associahedron} simplifies as follows.
\begin{lem}
Given a strict infinitesimal 2-braiding $t$, 
\begin{equation}
\left(\CC_{\bbK[h]/(h^3)}\,,\,\otimes\,,\,\alpha:\equiv1+\frac{h^2}{24}[t_{12}\,,t_{23}]\right)
\end{equation}
satisfies
\begin{align}
&\alpha_{1_{UV}\,\boxtimes\,\alpha_{WXY}}\alpha_{(UV)(WX)Y}(\alpha_{(UV)WX}\otimes1_Y)+\left(1_U\otimes\alpha_{VW(XY)}\right)\alpha_{U(VW)(XY)}\alpha_{\alpha_{UVW}\,\boxtimes\,1_{XY}}\label{eq:pentagonal associahedron}\\&\,\,\qquad+\left(1_U\otimes[1_V\otimes\alpha_{WXY}]\alpha_{V(WX)Y}\right)\alpha_{1_U\,\boxtimes\,\alpha_{VWX}\,\boxtimes\,1_Y}\left(\alpha_{U(VW)X}[\alpha_{UVW}\otimes1_X]\otimes1_Y\right)=0\,.\nn
\end{align}
\end{lem}
\begin{proof}
Identity pseudonatural transformations $1$ have trivial homotopy components (see Example \ref{ex:identity pseudonatural transformation}) and the homotopy components annihilate units \eqref{eqn: homotopy kills unit}.
\end{proof}
\begin{defi}\label{def:unital infinitesimal 2-braiding}
An infinitesimal 2-braiding $t$ is \textbf{unital} if: for every object $U\in\CC$, we have
\begin{subequations}
\begin{equation}\label{eq:t_U1=0=t_1U}
t_{UI}=0=t_{IU}\quad,
\end{equation}
and, for every 1-cell $f\in\CC[U,U']^0$, we have
\begin{equation}\label{eq:t_f1=0=t_1f}
t_{f\,\boxtimes\,1_I}=0=t_{1_I\,\boxtimes\,f}\quad.
\end{equation}
\end{subequations}
\end{defi}
\begin{lem}\label{lem:4-term rels are unital if t is}
If the infinitesimal 2-braiding is unital then so too are the 4-term relationators,
\begin{equation}
\L_{UVI}=\L_{UIW}=\L_{IVW}=\R_{UVI}=\R_{UIW}=\R_{IVW}=0\quad.
\end{equation}
\end{lem}
\begin{proof}
Obvious from the components \eqref{eqn: left 4 term}.
\end{proof}
\begin{lem}
A strict infinitesimal 2-braiding is unital.
\end{lem}
\begin{proof}
The monoidal product $\otimes$ is strictly unital \eqref{eq:unitality in terms of object/morphism} which allows us to write $V=(VI)$ and $g=g\otimes1_I$. We first use the cochain components of the left infinitesimal hexagon relation \eqref{eq:cochain left infhex}, 
\begin{subequations}
\begin{equation}
t_{UV}=t_{U(VI)}=t_{UV}\otimes1_I+(\gamma_{VU}\otimes1_I)(1_V\otimes t_{UI})(\gamma_{UV}\otimes1_I)=t_{UV}+\gamma_{VU}(1_V\otimes t_{UI})\gamma_{UV}\quad,
\end{equation}
thus $1_V\otimes t_{UI}=0$ hence setting $V=I$ forces $t_{UI}=0$. Now let us use the homotopy components of the left infinitesimal hexagon relation \eqref{eq:homotopy left infhex},
\begin{equation}
t_{f\,\boxtimes\,g}=t_{f\,\boxtimes\,(g\,\otimes\,1_I)}=t_{f\,\boxtimes \,g}\otimes1_I+(\gamma_{V'U'}\otimes1_I)(g\otimes t_{f\,\boxtimes \,1_I})(\gamma_{UV}\otimes1_I)=t_{f\,\boxtimes \,g}+\gamma_{V'U'}(g\otimes t_{f\,\boxtimes \,1_I})\gamma_{UV}\,.
\end{equation}
\end{subequations}
This gives us $g\otimes t_{f\,\boxtimes \,1_I}=0$ hence setting $g=1_I$ forces $t_{f\,\boxtimes \,1_I}=0$. Analogous arguments using the right infinitesimal hexagon relations \eqref{subeq:right inf hex relations explicitly} gives us $t_{IU}=0$ and $t_{1_I\,\boxtimes\,g}=0$.
\end{proof}
\begin{propo}\label{propo: SM1/2 hold}
Given a strict infinitesimal 2-braiding $t$, 
\begin{equation}
\left(\CC_{\bbK[h]/(h^3)}\,,\,\otimes\,,\,I\,,\,\alpha:\equiv1+\frac{h^2}{24}[t_{12}\,,t_{23}]\right)
\end{equation}
is a pentagonal strictly-unital monoidal $\Ch_{\bbK[h]/(h^3)}^{[-1,0]}$-category.
\end{propo}
\begin{proof}
The only remaining axiom to check is the triangle axiom \eqref{eq:triangle as cochain and homotopy}. We have 
\begin{subequations}
\begin{equation}
\alpha_{UIV}=\Phi(t_{UI},t_{IV})=\Phi(0,0)=1_{((UI)V)}=1_{(UV)}
\end{equation}
and
\begin{align}
\alpha_{f\,\boxtimes\,1_I\,\boxtimes\,g}&=1_{f\,\boxtimes\,1_I\,\boxtimes\,g}+\frac{h^2}{24}[t_{12}t_{23}-t_{23}t_{12}]_{f\,\boxtimes\,1_I\,\boxtimes\,g}\nn\\&=\frac{h^2}{24}\big[(t_{f\,\boxtimes\,1_I}\otimes g)(1_U\otimes t_{I\,\boxtimes\,V})+(t_{U'\,\boxtimes\,I}\otimes1_{V'})(f\otimes t_{1_I\,\boxtimes\,g})-(f\otimes t_{1_I\,\boxtimes\,g})(t_{U\,\boxtimes\,I}\otimes1_V)\nn\\&\,\,\qquad\qquad-(1_{U'}\otimes t_{I\,\boxtimes\,V'})(t_{f\,\boxtimes\,1_I}\otimes g)\big]\nn\\&=0\quad.
\end{align}
\end{subequations}
\end{proof}
\begin{rem}
It should be obvious from the above proof that the triangle axiom \eqref{eq:triangle as cochain and homotopy} is held to all orders (beyond $h^2$) because of: the unitality of $t$, the definition of the homotopy components of the vertical composition of pseudonatural transformations \eqref{eq:hom components of vercomp pseudos} and the fact that the associator is chosen to be the Drinfeld associator series $\alpha:=\Phi(t_{12},t_{23})$ which is a series purely in terms of $t_{12}$ and $t_{23}$. Furthermore, Lemma \ref{lem:4-term rels are unital if t is} makes it obvious that the prism axiom \eqref{eq:Prism} will also hold to all orders because the pentagonator series will be constructed out of the infinitesimal 2-braiding and the 4-term relationators (see \eqref{eq:third-order of pentagonator} for the third order term).
\end{rem}
\subsubsection{Totally symmetric infinitesimal 2-braidings}\label{subsub:totally symmetric}
Before we tackle the tetrahedra, hexahedron and Breen polytope axioms in Definition \ref{def:bra str un mon cat}, we recall the important notion of total symmetry as in Cirio and Martins\footnote{Note that, throughout their joint publications, compositions of morphisms are written from left to right.} \cite[Definition 18]{Joao}.
\begin{defi}\label{def: total sym}
A \textbf{totally symmetric} infinitesimal 2-braiding is a symmetric infinitesimal 2-braiding which satisfies
\begin{subequations}
\begin{equation}
\ast^2_{(\Id_\otimes,\gamma\,\boxtimes\,\Id_{\id_\CC}),(t,\Id_{\otimes\,\boxtimes\,\id_\CC})}=\ID_{\gamma_{12}t_{(12)3}}=\mathbf{0}
\end{equation}
or, in short,
\begin{equation}\label{left total symmetry}
t_{\gamma_{UV}\,\boxtimes \,1_W}=0\quad.
\end{equation}
\end{subequations}
\end{defi}
\begin{ex}\label{ex:our t is totally symmetric}
As in Example \ref{ex:Poisson inf 2bra is Breen}, it is clear from \cite[(3.28) and (3.29)]{Us} that those infinitesimal 2-braidings induced by 2-shifted Poisson structures are totally symmetric.
\end{ex}
\begin{rem}
Under the condition of symmetry, \eqref{left total symmetry} is equivalent to 
\begin{equation}\label{right total symmetry}
t_{1_U\,\boxtimes \,\gamma_{VW}}=0\quad.
\end{equation} 
Furthermore, in \cite[Definition 18]{Joao} conditions (27) and (29) are also redundant given that we can categorify these conditions as:
\begin{subequations}
\begin{alignat}{2}
\gamma_{23}t_{1(23)}\equiv\,&t*(1\,\boxtimes\,\gamma)\quad&&,\\
\gamma_{12}t_{(12)3}\equiv\,&t*(\gamma\boxtimes1)\quad&&,\label{eq:index left totalsym}
\end{alignat}
\end{subequations}
respectively, yet these are clear instances of the exchanger which Definition \ref{def: total sym} \emph{already} tells us are identity modifications. 
\end{rem}
\begin{lem}\label{lem:total symmetry implies Breen t}
A totally symmetric infinitesimal 2-braiding is a Breen infinitesimal 2-braiding.
\end{lem}
\begin{proof}
Obvious from \eqref{left total symmetry} and \eqref{right total symmetry}.
\end{proof}
\begin{rem}\label{rem:t_(ij)k=t_(ji)k}
Rearranging \eqref{eq:index left totalsym}, we have
\begin{subequations}
\begin{equation}
t_{(12)3}\equiv\otimes(\gamma^{-1}\boxtimes1)[t*(\gamma\boxtimes1)]\equiv\otimes(\gamma^{-1}\boxtimes1)[t*\Id_{\otimes\tau_{\CC,\CC}\,\boxtimes\,\id_\CC}]\otimes(\gamma\boxtimes1)\equiv:t_{(21)3}\quad,
\end{equation}
where the second equality comes from the condition for triviality of the exchanger in Remark \ref{rem:whiskering and exchanger}. In general, for $k\neq i\neq j\neq k$,
\begin{equation}
t_{(ij)k}\equiv t_{(ji)k}\qquad,\qquad t_{i(jk)}\equiv t_{i(kj)}\quad.
\end{equation}
\end{subequations}
\end{rem}
\begin{lem}\label{lem:totalsym further reduces independence of L/R}
Following Lemma \ref{lem:sym t reduces independence of L/R}, if the infinitesimal 2-braiding $t$ is totally symmetric then: 
\begin{equation}
\L=\R_{312}=\L_{213}=\R_{321}\,\,\,,\,\,\,
\R=\L_{231}=\R_{132}=\L_{321}\,\,\,,\,\,\,
\L_{132}=\R_{213}=\L_{312}=\R_{231}\,.
\end{equation} 
\end{lem}
\begin{proof}
Let us show that $\L=\L_{213}$, i.e.,
\begin{equation}
\L_{UVW}=(\gamma_{VU}\otimes1_W)\L_{VUW}(\gamma_{UV}\otimes1_W)\quad\implies\quad t_{t_{UV}\,\boxtimes\,1_W}=(\gamma_{VU}\otimes1_W)t_{t_{VU}\,\boxtimes\,1_W}(\gamma_{UV}\otimes1_W)\,.
\end{equation}
We can use the total symmetry \eqref{left total symmetry} together with the fact that the homotopy components of a pseudonatural transformation splits products \eqref{eqn:dubindex splits prods} to rewrite the latter expression as $t_{\gamma_{VU}t_{VU}\gamma_{UV}\,\boxtimes\,1_W}$ yet this is clearly equal to $t_{t_{UV}\,\boxtimes\,1_W}$ because of \eqref{intertwine single-index}. The proofs for the other relations work in the same way.
\end{proof}
\begin{rem}
Following on from Remark \ref{rem:pre-hex index swap almost equal}, Lemma \ref{lem:totalsym further reduces independence of L/R} gives us, for a totally symmetric $t$,
\begin{subequations}
\begin{equation}
\hh^R_{321}=\R_{132}+2\L_{213}=2\L+\R=\hh^L\quad.
\end{equation}
In fact, under this condition of total symmetry, 
\begin{equation}
\RR_{321}:\Phi(t_{32},t_{13})e^{\frac{h}{2}t_{(32)1}}\Phi(t_{21},t_{32})\Rrightarrow e^{\frac{h}{2}t_{31}}\Phi(t_{12},t_{31})e^{\frac{h}{2}t_{21}}
\end{equation}
\end{subequations}
but the domain/codomain pseudonatural automorphism is equal to that of \eqref{eq:ansatz left hexagonator} upon using \eqref{eq:sym t notation} and Remark \ref{rem:t_(ij)k=t_(ji)k}. In sum, a totally symmetric $t$ means that we can choose $\LL:=\RR_{321}$.
\end{rem}
The curious thing about Definition \ref{def: total sym} is that, whereas symmetry and strictness of $t$ are two independent properties, total symmetry and strictness of $t$ \emph{together} form a sufficient condition for ``strictness" of the 4-term relationators.
\begin{lem}[$\L$ is \textbf{primitive}]\label{lem: L is primitive}Given a totally symmetric strict infinitesimal 2-braiding, the left 4-term relationator satisfies the following ``primitive coproduct" formulae:
\begin{equation}
\L_{12(34)}=\L_{123}+\L_{124}\quad,\quad\L_{1(23)4}=\L_{124}+\L_{134}\quad,\quad\L_{(12)34}=\L_{134}+\L_{234}\,.
\end{equation}
\end{lem}
\begin{proof}
\begin{subequations}
\begin{align}
\L_{UV(WX)}=&t_{t_{UV}\,\boxtimes \,(1_W\otimes 1_X)}\nn\\=&t_{t_{UV}\,\boxtimes \,1_W}\otimes 1_X+(\gamma_{W(UV)}\otimes 1_X)(1_W\otimes t_{t_{UV}\,\boxtimes \,1_X})(\gamma_{(UV)W}\otimes 1_X)\nn\\=&\L_{UVW}\otimes 1_X+(\gamma_{W(UV)}\otimes 1_X)(1_W\otimes \L_{UV X})(\gamma_{(UV)W}\otimes 1_X)\quad,
\end{align}
for the second equality we have used the homotopy components of the left infinitesimal hexagon relation \eqref{eq:homotopy left infhex}.
\begin{align}
\L_{U(VW)X}=&t_{t_{U(VW)}\,\boxtimes \,1_X}\\=&t_{\big(t_{UV}\otimes 1_W+(\gamma_{VU}\otimes 1_W)(1_V\otimes t_{UW})(\gamma_{UV}\otimes 1_W)\big)\boxtimes \,1_X}\nn\\=&t_{UV}\otimes t_{1_W\,\boxtimes \,1_X}+(1_{(UV)}\otimes \gamma_{XW})(t_{t_{UV}\,\boxtimes \,1_X}\otimes1_W)(1_{(UV)}\otimes\gamma_{WX})\nn\\&+(\gamma_{VU}\otimes 1_{(WX)})t_{(1_V\otimes t_{UW})\boxtimes \,1_X}(\gamma_{UV}\otimes 1_{(WX)})\nn\\=&(1_{(UV)}\otimes \gamma_{XW})(\L_{UVX}\otimes 1_W)(1_{(UV)}\otimes \gamma_{WX})+(\gamma_{VU}\otimes 1_{(WX)})(1_V\otimes \L_{UWX})(\gamma_{UV}\otimes 1_{(WX)})\quad,\nn
\end{align}
for the second equality we have used the cochain components of the left infinitesimal hexagon relation \eqref{eq:cochain left infhex}. For the first two terms of the third equality we have used the homotopy components of the right infinitesimal hexagon relation \eqref{eq:homotopy right infhex}, and for the third term we have simultaneously used both the fact that the homotopy components split products \eqref{eqn:dubindex splits prods} and the total symmetry \eqref{left total symmetry}. The last equality used the fact that the homotopy components kill units \eqref{eqn: homotopy kills unit}.
\begin{align}
\L_{(UV)WX}=&t_{t_{(UV)W}\,\boxtimes \,1_X}\nn\\=&t_{\big(1_U\otimes t_{VW}+(1_U\otimes\gamma_{WV})(t_{UW}\otimes 1_V)(1_U\otimes \gamma_{VW})\big)\boxtimes \,1_X}\nn\\=&1_U\otimes \L_{VWX}+(1_U\otimes\gamma_{(WX)V})(\L_{UWX}\otimes1_V)(1_U\otimes\gamma_{V(WX)})\quad.
\end{align}
\end{subequations}
Again, the second equality uses the right infinitesimal hexagon relation \eqref{subeq:right inf hex relations explicitly} and the third equality uses it multiple times together with the total symmetry and the facts that the homotopy components split products and annihilate units.
\end{proof}
\begin{cor}\label{cor:inf hex are primitive}
$\L$ and $\R$ are primitive and so too are the infinitesimal hexagonators \eqref{eq:inf hexagonators}.
\end{cor}
This corollary is quite powerful as exemplified by the following proposition.
\begin{propo}\label{prop:SM3/4 hold}
Given a totally symmetric strict infinitesimal 2-braiding $t$,
\begin{equation}\label{eq:assbrahex structure of order 3}
\left(\CC_{\bbK[h]/(h^3)}\,,\,\otimes\,,I\,,\,\alpha:\equiv1+\frac{h^2}{24}[t_{12},t_{23}]\,,\,\sigma:\equiv\gamma e^{\frac{h}{2}t}\,,\,\H^L:=\frac{h^2}{24}\gamma_{1(23)}\hh^L\,,\,\H^R:=\frac{h^2}{24}\gamma_{(12)3}\hh^R\right)
\end{equation}
satisfies the left/right tetrahedron axiom \eqref{eq:left tetrahedron}/\eqref{eq:right tetrahedron} and the hexahedron axiom \eqref{eq:Hexahedron}.
\end{propo}
\begin{proof}
With trivial pentagonator, the left tetrahedron axiom \eqref{eq:left tetrahedron} reads as 
\begin{subequations}
\begin{align}
&(1_V\otimes\alpha_{WXU})\alpha_{V(WX)U}\sigma_{1_U\,\boxtimes \,\alpha_{VWX}}\alpha_{U(VW)X}(\alpha_{UVW}\otimes1_X)\nn\\
&+(1_V\otimes\alpha_{WXU})\H^L_{UV(WX)}\alpha_{(UV)WX}+(1_V\otimes\alpha_{WXU}\sigma_{U(WX)})\alpha_{VU(WX)}\alpha_{\sigma_{UV}\,\boxtimes \,1_W\,\boxtimes \,1_X}\nn\\&=\\
&\alpha_{VW(XU)}\H^L_{U(VW)X}(\alpha_{UVW}\otimes1_X)-\alpha_{1_V\,\boxtimes\,1_W\,\boxtimes\,\sigma_{UX}}\alpha_{(VW)UX}(\sigma_{U(VW)}\alpha_{UVW}\otimes1_X)\nn\\&+\left(1_V\otimes(1_W\otimes\sigma_{UX})\alpha_{WUX}\right)\alpha_{V(WU)X}(\H^L_{UVW}\otimes1_X)-\left(1_V\otimes\H^L_{UWX}\right)\alpha_{V(UW)X}\left(\alpha_{VUW}(\sigma_{UV}\otimes1_W)\otimes1_X\right)\nn\\&-\left(1_V\otimes(1_W\otimes\sigma_{UX})\alpha_{WUX}\right)\alpha_{1_V\,\boxtimes \,\sigma_{UW}\,\boxtimes \,1_X}\left(\alpha_{VUW}(\sigma_{UV}\otimes1_W)\otimes1_X\right)\quad.\nn
\end{align}
The order $h^2$ relation is given by
\begin{align}
\nn&\left(\gamma_{1(23)}\hh^L\right)_{UV(WX)}+(1_V\otimes\gamma_{U(WX)})[t_{12}\,,t_{23}]_{\gamma_{UV}\,\boxtimes \,1_W\,\boxtimes \,1_X}\\&=\label{eqn:h^2 of SM3.i}\\&\left(\gamma_{1(23)}\hh^L\right)_{U(VW)X}-[t_{12}\,,t_{23}]_{1_V\,\boxtimes \,1_W\,\boxtimes \,\gamma_{UX}}(\gamma_{U(VW)}\otimes1_X)\nn\\&+\left(1_{(VW)}\otimes\,\gamma_{UX}\right)\left(\left(\gamma_{1(23)}\hh^L\right)_{UVW}\otimes1_X\right)-\left(1_V\otimes\,\left(\gamma_{1(23)}\hh^L\right)_{UWX}\right)\left(\gamma_{UV}\otimes1_{(WX)}\right)\nn\\&-\left(1_{(VW)}\otimes\gamma_{UX}\right)[t_{12}\,,t_{23}]_{1_V\,\boxtimes \,\gamma_{UW}\,\boxtimes \,1_X}\left(\gamma_{UV}\otimes1_{(WX)}\right)\quad.\nn
\end{align}
Terms such as
\begin{equation}
\left(t_{12}t_{23}-t_{23}t_{12}\right)_{\gamma_{UV}\,\boxtimes \,1_W\,\boxtimes \,1_X}=\left(t_{\gamma_{UV}\,\boxtimes \,1_W}\otimes1_X\right)\left(1_{(UV)}\otimes t_{WX}\right)-\left(1_{(VU)}\otimes t_{WX}\right)\left(t_{\gamma_{UV}\,\boxtimes \,1_W}\otimes1_X\right)
\end{equation}
are trivial by the condition of total symmetry \eqref{left total symmetry} thus \eqref{eqn:h^2 of SM3.i} simplifies to
\begin{align}
\left(\gamma_{1(23)}\hh^L\right)_{UV(WX)}
=\left(\gamma_{1(23)}\hh^L\right)_{U(VW)X}&+\left(1_{(VW)}\otimes\,\gamma_{UX}\right)\left(\left(\gamma_{1(23)}\hh^L\right)_{UVW}\otimes1_X\right)\nn\\&-\left(1_V\otimes\,\left(\gamma_{1(23)}\hh^L\right)_{UWX}\right)\left(\gamma_{UV}\otimes1_{(WX)}\right)\quad.
\end{align}
This is simply
\begin{align}
\gamma_{U(VWX)}\hh^L_{UV(WX)}
=\gamma_{U(VWX)}\hh^L_{U(VW)X}&+\left(1_{(VW)}\otimes\,\gamma_{UX}\right)\left(\gamma_{U(VW)}\hh^L_{UVW}\otimes1_X\right)\nn\\&-\left(1_V\otimes\,\gamma_{U(WX)}\hh^L_{UWX}\right)\left(\gamma_{UV}\otimes1_{(WX)}\right)\quad.
\end{align}
The symmetric braiding $\gamma$ satisfies the hexagon axiom hence we can simplify even further,
\begin{equation}
\hh^L_{UV(WX)}
=\hh^L_{U(VW)X}+\hh^L_{UVW}\otimes1_X-\left(\gamma_{VU}\otimes1_{(WX)}\right)\left(1_V\otimes\,\hh^L_{UWX}\right)\left(\gamma_{UV}\otimes1_{(WX)}\right)\quad.
\end{equation}
We finally rewrite this as 
\begin{equation}
\hh^L_{12(34)}=\hh^L_{1(23)4}+\hh^L_{123}-\hh^L_{134}
\end{equation}
and apply Corollay \ref{cor:inf hex are primitive}.

Comparing \eqref{eq:left tetrahedron} with \eqref{eq:right tetrahedron} reveals that the right tetrahedron axiom would eventually reduce down to a similar equation as the above but now with $\hh^L$ replaced with $\hh^R$.

With trivial pentagonator, the hexahedron axiom \eqref{eq:Hexahedron} reads as 
\begin{footnotesize}
\begin{align}
&\alpha_{W(XU)V}(\alpha_{WXU}\otimes1_V)\H^R_{UV(WX)}(1_U\otimes\alpha_{VWX})\alpha_{U(VW)X}-(1_W\otimes\alpha^{-1}_{XUV})\H^L_{(UV)WX}(\alpha^{-1}_{UVW}\otimes1_X)\nn\\&=\\&(1_W\otimes\,\H^R_{UVX})\alpha_{WU(VX)}\alpha_{(WU)VX}(\alpha^{-1}_{WUV}\sigma_{(UV)W}\alpha^{-1}_{UVW}\otimes1_X)\nn\\&+[1_W\otimes(\sigma_{UX}\otimes1_V)\alpha^{-1}_{UXV}(1_U\otimes\sigma_{VX})]\alpha_{WU(VX)}\left[\alpha_{(WU)VX}(\H^R_{UVW}\otimes1_X)-\alpha_{\sigma_{UW}\boxtimes1_{VX}}\left(\alpha^{-1}_{UWV}(1_U\otimes\sigma_{VW})\otimes1_X\right)\right]\nn\\&+(1_W\otimes[\sigma_{UX}\otimes1_V]\alpha^{-1}_{UXV})\alpha_{1_{WU}\boxtimes\sigma_{VX}}(\sigma_{UW}\otimes1_{(VX)})\alpha_{(UW)VX}\left(\alpha^{-1}_{UWV}[1_U\otimes\sigma_{VW}]\otimes1_X\right)\nn\\&+\alpha_{1_W\boxtimes\sigma_{UX}\boxtimes1_V}(\alpha_{WUX}\otimes1_V)\alpha^{-1}_{(WU)XV}[\sigma_{UW}\otimes\sigma_{VX}]\alpha_{(UW)VX}\left[\alpha^{-1}_{UWV}[1_U\otimes\sigma_{VW}]\otimes1_X\right]\nn\\&-\alpha_{W(XU)V}\left[(1_W\otimes\sigma_{UX})\alpha_{WUX}\otimes1_V\right]\alpha^{-1}_{(WU)XV}(\sigma_{UW}\otimes\sigma_{VX})\alpha^{-1}_{UW(VX)}(1_U\otimes\alpha_{WVX})\alpha_{1_U\boxtimes\sigma_{VW}\boxtimes1_X}\nn\\&-\alpha_{W(XU)V}\left[(1_W\otimes\sigma_{UX})\alpha_{WUX}\otimes1_V\right]\alpha^{-1}_{\sigma_{UW}\boxtimes1_{XV}}(1_{(UW)}\otimes\sigma_{VX})\alpha^{-1}_{UW(VX)}[1_U\otimes\alpha_{WVX}(\sigma_{VW}\otimes1_X)]\alpha_{U(VW)X}\nn\\&+\alpha_{W(XU)V}\left[(1_W\otimes\sigma_{UX})\alpha_{WUX}(\sigma_{UW}\otimes1_X)\otimes1_V\right]\alpha^{-1}_{(UW)XV}\alpha^{-1}_{1_{UW}\boxtimes\sigma_{VX}}[1_U\otimes\alpha_{WVX}(\sigma_{VW}\otimes1_X)]\alpha_{U(VW)X}\nn\\&-\alpha_{W(XU)V}(\H^L_{UWX}\otimes1_V)\alpha^{-1}_{(UW)XV}\alpha^{-1}_{UW(XV)}[1_U\otimes(1_W\otimes\sigma_{VX})\alpha_{WVX}(\sigma_{VW}\otimes1_X)]\alpha_{U(VW)X}\nn\\&-\alpha_{W(XU)V}\left(\alpha_{WXU}\sigma_{U(WX)}\alpha_{UWX}\otimes1_V\right)\alpha^{-1}_{(UW)XV}\alpha^{-1}_{UW(XV)}(1_U\otimes\H^L_{VWX})\alpha_{U(VW)X}\qquad.\nn
\end{align}
\end{footnotesize}
We use the comment from above that total symmetry annihilates the associator homotopy components of the symmetric braiding (e.g., $\alpha_{\gamma_{UW}\,\boxtimes\,1_{VX}}=0$) to derive the order $h^2$ relation,
\begin{align}
&[\gamma_{(12)3}\hh^R]_{UV(WX)}-[\gamma_{1(23)}\hh^L]_{(UV)WX}\nn\\&=\\&(1_W\otimes[\gamma_{(12)3}\hh^R]_{UVX})(\gamma_{(UV)W}\otimes1_X)+[1_W\otimes(\gamma_{UX}\otimes1_V)(1_U\otimes\gamma_{VX})]([\gamma_{(12)3}\hh^R]_{UVW}\otimes1_X)\nn\\&-([\gamma_{1(23)}\hh^L]_{UWX}\otimes1_V)[1_U\otimes(1_W\otimes\gamma_{VX})(\gamma_{VW}\otimes1_X)]-\left(\gamma_{U(WX)}\otimes1_V\right)(1_U\otimes[\gamma_{1(23)}\hh^L]_{VWX})\qquad.\nn
\end{align}
This is simply
\begin{align}
&\gamma_{(UV)(WX)}\hh^R_{UV(WX)}-\gamma_{(UV)(WX)}\hh^L_{(UV)WX}\nn\\&=\\&\left[1_W\otimes\gamma_{(UV)X}\hh^R_{UVX}\right]\left[\gamma_{(UV)W}\otimes1_X\right]+\left[1_W\otimes\gamma_{(UV)X}\right]\left[\gamma_{(UV)W}\hh^R_{UVW}\otimes1_X\right]\nn\\&-\left[\gamma_{U(WX)}\hh^L_{UWX}\otimes1_V\right]\left[1_U\otimes\gamma_{V(WX)}\right]-\left[\gamma_{U(WX)}\otimes1_V\right]\left[1_U\otimes\gamma_{V(WX)}\hh^L_{VWX}\right]\qquad.\nn
\end{align}
Again, the hexagon axiom that $\gamma$ satisfies as a symmetric braiding means we can write
\begin{equation}
\gamma_{(UV)(WX)}=\left[1_W\otimes\gamma_{(UV)X}\right]\left[\gamma_{(UV)W}\otimes1_X\right]=\left[\gamma_{U(WX)}\otimes1_V\right]\left[1_U\otimes\gamma_{V(WX)}\right]
\end{equation}
which finally leaves us with 
\begin{align}
\hh^R_{UV(WX)}-\hh^L_{(UV)WX}=&\left(\gamma_{W(UV)}\otimes1_X\right)\left(1_W\otimes\hh^R_{UVX}\right)\left(\gamma_{(UV)W}\otimes1_X\right)+\hh^R_{UVW}\otimes1_X\nn\\&-\left(1_U\otimes\gamma_{(WX)V}\right)\left(\hh^L_{UWX}\otimes1_V\right)\left(1_U\otimes\gamma_{V(WX)}\right)-1_U\otimes\hh^L_{VWX}
\end{align}
\end{subequations}
but, once again, we know the infinitesimal hexagonators to be primitive as in Corollay \ref{cor:inf hex are primitive}.
\end{proof}

\subsubsection{Coherent infinitesimal 2-braidings and the Breen polytope}
Before we discuss the Breen polytope axiom \eqref{eqn:Breen axiom}, we quickly recall another important notion in CM; this time the notion of coherency \cite[Defintion 17]{Joao}. 
\begin{defi}\label{def: coherency}
An infinitesimal 2-braiding is \textbf{coherent} if it satisfies
\begin{subequations}
\begin{equation}
-(\gamma_{VU}\otimes1_W)\R_{VUW}(\gamma_{UV}\otimes1_W)=\L_{UVW}+\R_{UVW}=-(1_U\otimes\,\gamma_{WV})\L_{UWV}(1_U\otimes\,\gamma_{VW}),
\end{equation}
that is,
\begin{equation}
-(\gamma_{VU}\otimes1_W)t_{1_V\,\boxtimes \,t_{UW}}(\gamma_{UV}\otimes1_W)=t_{t_{UV}\,\boxtimes \,1_W}+t_{1_U\,\boxtimes \,t_{VW}}=-(1_U\otimes\,\gamma_{WV})t_{t_{UW}\,\boxtimes \,1_V}(1_U\otimes\,\gamma_{VW})\quad.
\end{equation}
\end{subequations}
\end{defi}
\begin{rem}\label{rem: coherency is natural}
As in Lemma \ref{lem:sym t reduces independence of L/R}, a symmetric infinitesimal 2-braiding would already give us
\begin{equation}
(\gamma_{VU}\otimes1_W)\R_{VUW}(\gamma_{UV}\otimes1_W)=(1_U\otimes\,\gamma_{WV})\L_{UWV}(1_U\otimes\,\gamma_{VW}),
\end{equation}
hence, in contrast with the remarks after \cite[Defintion 17]{Joao}, including symmetry as a condition to the definition of coherency would make the definition entirely natural from the categorical perspective; the Breen polytope can only enforce \emph{one} equality at order $h^2$. Another way to think of this coherency condition, assuming symmetry and strictness, is simply by looking at the relevant pseudonatural transformation which is modified to 0, i.e., 
\begin{equation}\label{eq:canonical endomod}
\L+\R+\L_{132}:[t_{12}+t_{23}\,,t_{13}]+[t_{13}\,,t_{(13)2}]\Rrightarrow0\quad,
\end{equation}
however we can use the strictness and symmetry to notice that the domain pseudonatural transformation is already 0. In sum, for a symmetric strict infinitesimal 2-braiding, coherency means that the canonical endomodification on 0 is the particular identity modification $\ID_0=\mathbf{0}$.
\end{rem}
We can rewrite \eqref{eq:canonical endomod} purely in terms of the left 4-term relationator $\L$ by using the relation $\R=\L_{231}$ from Lemma \ref{lem:sym t reduces independence of L/R}, now the condition of coherency for a symmetric strict $t$ reads as
\begin{equation}\label{eq:coherency in terms of L}
\L+\L_{231}+\L_{132}=\mathbf{0}\quad.
\end{equation}
\begin{rem}\label{rem:our t is coherent}
Using the basis expression for $\L$ in \cite[Example 3.18]{Us}, \eqref{eq:coherency in terms of L} becomes
\begin{align}
&\langle\pi^{(2)},(-1)^{|s^{-1}\und{N}_q^{~r}||s^{-1}\und{L}_{s}^{~l}|}s^{-1}\dd_{\mathrm{dR}}\langle\pi^{(2)},s^{-1}\dd_{\mathrm{dR}}\und{M}_{i}^{~j}\otimes_As^{-1}\dd_{\mathrm{dR}}\und{L}_{s}^{~l}\rangle\otimes_As^{-1}\dd_{\mathrm{dR}}\und{N}_q^{~r}\\&+s^{-1}\dd_{\mathrm{dR}}\langle\pi^{(2)},s^{-1}\dd_{\mathrm{dR}}\und{M}_{i}^{~j}\otimes_As^{-1}\dd_{\mathrm{dR}}\und{N}_q^{~r}\rangle\otimes_As^{-1}\dd_{\mathrm{dR}}\und{L}_{s}^{~l}\nn\\&+(-1)^{|s^{-1}\und{M}_{i}^{~j}|+|u_r||s^{-1}\und{L}_{s}^{~l}|}s^{-1}\dd_{\mathrm{dR}}\und{M}_{i}^{~j}\otimes_As^{-1}\dd_{\mathrm{dR}}\langle\pi^{(2)},s^{-1}\dd_{\mathrm{dR}}\und{N}_q^{~r}\otimes_As^{-1}\dd_{\mathrm{dR}}\und{L}_{s}^{~l}\rangle\rangle w_j\otimes_Au_r\otimes_Av_l=0\,.\nn
\end{align}
To understand this requirement let us recall the following operations from \cite[Chapter 3]{Poisson}: of the de Rham differential (3.25), the internal product action of multiderivations on Kähler differentials (3.29) and  of the Schouten bracket between multiderivations (3.36), keeping in mind that in our case $A$ is a CDGA. In fact, looking at Cartan's formula as in \cite[Proposition 3.6]{Poisson} suggests that the above LHS relates to the interior product $\iota_{\left[\pi^{(2)},\pi^{(2)}\right]_{\mathrm{S}}}$ but, of course, the third weight component of the Maurer-Cartan equation\footnote{Now we can see that the definition of coherency is \textit{also} entirely natural from the perspective of Derived Algebraic Geometry.} \cite[(2.27)]{Us} then implies that this is proportional to $\iota_{\dd_{\widehat{\Pol}}(\pi^{(3)})}$. But this then means that the above LHS is the output of the inner hom differential $\partial$ of the degree $-2$ morphism which is constructed analogously to \cite[(3.25a)]{Us} but with three $\xi$'s and the dual pairing $\langle\pi^{(3)},\cdot\rangle$. Finally, the truncation then guarantees that the above requirement is satisfied and indeed our concrete infinitesimal 2-braiding is coherent. 
\end{rem}
\begin{propo}\label{propo: h^2 Breen polytope}
Given a totally symmetric strict infinitesimal 2-braiding $t$, \eqref{eq:assbrahex structure of order 3} satisfies the Breen polytope axiom \eqref{eqn:Breen axiom} if and only if $t$ is coherent.
\end{propo}
\begin{proof}
The order $h^2$ relation of \eqref{eqn:Breen axiom} is given as
\begin{subequations}
\begin{align}
&-(\gamma_{VW}\otimes1_U)\left(\frac{1}{24}\gamma_{1(23)}\hh^L\right)_{UVW}+\frac{1}{2}(\gamma t)_{1_U\,\boxtimes \,\frac{1}{2}\gamma_{VW}t_{VW}}+\left(\frac{1}{24}\gamma_{1(23)}\hh^L\right)_{UWV}(1_U\otimes\gamma_{VW})\nn\\&=\\&-\left(\frac{1}{24}\gamma_{(12)3}\hh^R\right)_{VUW}(\gamma_{UV}\otimes1_W)-\frac{1}{2}(\gamma t)_{\frac{1}{2}\gamma_{UV}t_{UV}\,\boxtimes \,1_W}+(1_W\otimes\gamma_{UV})\left(\frac{1}{24}\gamma_{(12)3}\hh^R\right)_{UVW}\nn
\end{align}
which we rewrite,
\begin{align}
&(\gamma_{VW}\otimes1_U)\gamma_{U(VW)}\hh^L_{UVW}-6\gamma_{U(WV)}t_{1_U\,\boxtimes \,\gamma_{VW}t_{VW}}-\gamma_{U(WV)}\hh^L_{UWV}(1_U\otimes\gamma_{VW})\nn\\&=\\&\gamma_{(VU)W}\hh^R_{VUW}(\gamma_{UV}\otimes1_W)+6\gamma_{(VU)W}t_{\gamma_{UV}t_{UV}\,\boxtimes \,1_W}-(1_W\otimes\,\gamma_{UV})\gamma_{(UV)W}\hh^R_{UVW}\quad.\nn
\end{align}
Now we simultaneously use: the fact that the homotopy components of a pseudonatural transformation splits products \eqref{eqn:dubindex splits prods}, the total symmetry of $t$ and the hexagon axiom that $\gamma$ satisfies,
\begin{equation}
\hh^L_{UVW}-6\R_{UVW}-(1_U\otimes\,\gamma_{WV})\hh^L_{UWV}(1_U\otimes\,\gamma_{VW})=(\gamma_{VU}\otimes1_W)\hh^R_{VUW}(\gamma_{UV}\otimes1_W)+6\L_{UVW}-\hh^R_{UVW}\,.
\end{equation}
We use the definitions of the infinitesimal hexagonators \eqref{eq:inf hexagonators} and the subscript index notation,
\begin{equation}
2\L+\R-6\R-2\L_{132}-\R_{132}=\L_{213}+2\R_{213}+6\L-\L-2\R\quad.
\end{equation}
As in Lemma \ref{lem:totalsym further reduces independence of L/R}, total symmetry of $t$ gives us: $\R_{213}=\L_{132}$, $\R_{132}=\R$ and $\L_{213}=\L$. We thus rewrite the above as 
\begin{equation}
4\L+4\R+4\L_{132}=0
\end{equation}
and this is equivalent to the condition of coherency for a symmetric $t$.
\end{subequations}
\end{proof}
\subsection{The second-order coboundary symmetric syllepsis}
So far, this section has shown that a totally symmetric strict infinitesimal 2-braiding $t$ provides a second-order deformation of the symmetric strict monoidal $\Ch_\bbK^{[-1,0]}$-category $(\CC,\otimes,I,\gamma)$ to a braided pentagonal strictly-unital monoidal $\Ch_{\bbK[h]/(h^3)}^{[-1,0]}$-category \eqref{eq:assbrahex structure of order 3} if and only if $t$ is coherent. 
\begin{rem}\label{rem:h^2 coboundary syllepsis}
We now wish to show that the coboundary symmetric syllepsis $\Sigma:=h\gamma T$ from Remark \ref{rem:coboundary 2-shifted induces syllep} is actually a symmetric syllepsis on the above braided pentagonal strictly-unital monoidal $\Ch_{\bbK[h]/(h^3)}^{[-1,0]}$-category. First let us mention that \eqref{eq:syllepsis is homotopy} and \eqref{eq:syllepsis is quasinatural} still hold to order $h^2$ because of the symmetry of $t$ and because $\Sigma$ has no order $h^2$ term. The order $h^2$ relation of \eqref{eq:left factorisation} reads as 
\begin{subequations}
\begin{equation}\label{eq:h^2 left factorisation}
2\L+\R=-(\L_{231}+2\R_{231})+12t_{13}T_{12}-12T_{13}t_{12}\quad,
\end{equation}
where we used Remark \ref{rem:inverse isomodification} for the inverse $(\H^R_{231})^{-1}$. The order $h^2$ relation of \eqref{eq:right factorisation} reads as 
\begin{equation}\label{eq:h^2 right factorisation}
\L+2\R=-(2\L_{312}+\R_{312})+12t_{13}T_{23}-12T_{13}t_{23}\qquad.
\end{equation}
Lemma \ref{lem:totalsym further reduces independence of L/R} gives us: $\L_{231}=\R,\,\,\R_{312}=\L$ and $\R_{231}=\L_{312}=\L_{132}\,$, and coherency gives us $\L_{132}=-(\L+\R)$ hence \eqref{eq:h^2 left factorisation} and \eqref{eq:h^2 right factorisation} become, respectively:
\begin{equation}\label{eq:h^2 factorisation}
t_{13}T_{12}=T_{13}t_{12}\qquad,\qquad t_{13}T_{23}=T_{13}t_{23}\quad.
\end{equation}
Likewise, the order $h^2$ relation of \eqref{eq:symmetric syllepsis} reads as 
\begin{equation}\label{eq:h^2 symmetry of syllep}
[t,T]=0\quad.
\end{equation}
Using the truncation of the hom-complexes together with the fact that $T:t\Rrightarrow0$ (hence Definition \ref{def:coboundary of mod} allows us to write $\partial T=t$), we have 
\begin{equation}
t_{13}T_{12}=(\partial T_{13})T_{12}=T_{13}\partial T_{12}=T_{13}t_{12}\quad.
\end{equation}
The other relation in \eqref{eq:h^2 factorisation} and the relation \eqref{eq:h^2 symmetry of syllep} are proved similarly.
\end{subequations}
\end{rem}

\section*{Acknowledgments}
I am supported by an EPSRC doctoral studentship. I would like to thank Alexander Schenkel and Robert Laugwitz for: teaching me the general idea of the definition of a braided monoidal bicategory in \cite[Appendix C]{Schommer}, providing helpful remarks as regards the order of material in Subsection \ref{subsec:4T relationators} and motivating the belief that 3-shifted Poisson structures should induce syllepses. I would also like to thank João Faria Martins for a helpful discussion circa beginning of January.

\appendix
\section{Third-order deformation data}\label{app:Third-order data}
Throughout this appendix, we \textit{must} take $\bbK=\bbC$ because the coefficient of the third-order term of the Drinfeld associator series in \cite[Corollary XIX.6.5]{Kassel}, 
\begin{equation}\label{eq:Drinfeld 3rd order with h}
\alpha:\equiv\Phi_\KZ(t_{12},t_{23})\equiv1+\frac{h^2}{24}[t_{12},t_{23}]+\frac{h^3\zeta(3)}{8\pi^3i}\big[t_{12}+t_{23},[t_{12},t_{23}]\big]+\ldots\quad,
\end{equation}
is imaginary and irrational (and possibly transcendental). To simplify the third-order term, we take $h:=2\pi i\hbar$ hence \eqref{eq:Drinfeld 3rd order with h} becomes
\begin{equation}\label{eq:Drinfeld 3rd order with hbar}
\Phi(t_{12},t_{23})\equiv1-\frac{\hbar^2\pi^2}{6}[t_{12},t_{23}]+\hbar^3\zeta(3)\big[[t_{12},t_{23}],t_{12}+t_{23}\big]+\ldots
\end{equation}

\subsection{The pentagonator}\label{appsubsec:pentagonator}
We substitute \eqref{eq:Drinfeld 3rd order with hbar} into \eqref{eq:index pentagonator} and read off the third-order terms,
\begin{align}
\frac{1}{\zeta(3)}(\Pi)_3:\big[[t_{23},t_{34}],t_{23}+t_{34}\big]+&\big[[t_{1(23)},t_{(23)4}],t_{1(23)}+t_{(23)4}\big]+\big[[t_{12},t_{23}],t_{12}+t_{23}\big]\\&\Rrightarrow\big[[t_{12},t_{2(34)}],t_{12}+t_{2(34)}\big]+\big[[t_{(12)3},t_{34}],t_{(12)3}+t_{34}\big]\,.\nn
\end{align}
Let us assume strictness of $t$ and write\footnote{Throughout this subsection, we will make use of the fact that $$[t_{ij},t_{kl}]=0$$for $\{i,j,k,l\}=\{1,2,3,4\}$. Recall that 2 (out of 3, for a symmetric $t$) of these ``disjoint-commuting" relations was demonstrated in the proof of Lemma \ref{lem:pent axiom upheld}.}:
\begin{alignat}{3}
\big[[t_{12},t_{2(34)}],t_{12}+t_{2(34)}\big]&=\big[[t_{12},t_{23}],t_{12}+t_{23}\big]+\big[[t_{12},t_{23}],t_{24}\big]+\big[[t_{12},t_{24}],t_{12}+t_{23}+t_{24}\big]\,&&,\\
\big[[t_{(12)3},t_{34}],t_{(12)3}+t_{34}\big]&=\big[[t_{23},t_{34}],t_{23}+t_{34}\big]+\big[[t_{23},t_{34}],t_{13}\big]+\big[[t_{13},t_{34}],t_{13}+t_{23}+t_{34}\big]\,&&,\\
\big[[t_{1(23)},t_{(23)4}],t_{1(23)}+t_{(23)4}\big]&=\big[[t_{12},t_{24}]+[t_{13},t_{34}],t_{12}+t_{13}+t_{24}+t_{34}\big]\qquad&&.
\end{alignat}
Now we balance terms as usual,
\begin{equation}
\frac{1}{\zeta(3)}(\widehat{\Pi})_3:\big[t_{13},[t_{23},t_{34}]\big]-\big[t_{24},[t_{23},t_{12}]\big]+\big[[t_{13},t_{34}],t_{12}+t_{24}-t_{23}\big]+\big[[t_{12},t_{24}],t_{13}+t_{34}-t_{23}\big]\Rrightarrow0\,.
\end{equation}
With a view to four-term relationators, we make use of strictness of $t$ together with the disjoint-commuting relations to rewrite the above as
\begin{equation}
\big[t_{13},[t_{23},t_{34}]\big]-\big[t_{24},[t_{23},t_{12}]\big]+\big[t_{23},[t_{1(23)},t_{(23)4}]\big]+\big[[t_{13},t_{34}],t_{12}+t_{24}\big]+\big[[t_{12},t_{24}],t_{13}+t_{34}\big]\Rrightarrow0\,.
\end{equation}
\begin{enumerate}
\item[(i)] The first two terms give us 
\begin{align}
\big[t_{13},[t_{23},t_{34}]\big]-\big[t_{24},[t_{23},t_{12}]\big]&=\big[t_{13},[t_{23},t_{(23)4}-t_{24}]\big]-\big[t_{24},[t_{23},t_{1(23)}-t_{13}]\big]\nn\\&=\big[t_{13},\partial\L_{234}\big]-\big[t_{24},\partial\R_{123}\big]-\big[t_{13},[t_{23},t_{24}]\big]+\big[t_{24},[t_{23},t_{13}]\big]\nn\\&=\partial\left(\big[t_{13},\L_{234}\big]-\big[t_{24},\R_{123}\big]\right)\quad,
\end{align}
where we used the Jacobi identity to eliminate the last two terms of the second equality.
\item[(ii)] The third term straightforwardly gives us
\begin{equation}
\big[[t_{23},t_{1(23)}],t_{(23)4}\big]+\big[t_{1(23)},[t_{23},t_{(23)4}]\big]=\partial\left([t_{12}+t_{13},\L_{234}]-[t_{24}+t_{34},\R_{123}]\right)
\end{equation}
\item[(iii)] The fourth and the fifth term $\big[[t_{13},t_{34}],t_{12}+t_{24}\big]+\big[[t_{12},t_{24}],t_{13}+t_{34}\big]$ becomes 
\begin{align}
=&\big[t_{34},[t_{12},t_{13}]\big]+\big[t_{13},[t_{34},t_{24}]\big]-\big[t_{24},[t_{12},t_{13}]\big]-\big[t_{12},[t_{34},t_{24}]\big]\nn\\=&\big[t_{34},[t_{12},t_{(12)3}-t_{23}]\big]+\big[t_{13},[t_{34},t_{2(34)}-t_{23}]\big]-\big[t_{24},[t_{12},t_{(12)3}-t_{23}]\big]-\big[t_{12},[t_{34},t_{2(34)}-t_{23}]\big]\nn\\=&\partial\left(\big[t_{34},\L_{123}\big]+\big[t_{13},\R_{234}\big]-[t_{24},\L_{123}]-[t_{12},\R_{234}]\right)-\big[t_{34},[t_{12},t_{23}]\big]-\big[t_{13},[t_{34},t_{23}]\big]\nn\\&+\big[t_{24},[t_{12},t_{23}]\big]+\big[t_{12},[t_{34},t_{23}]\big]
\end{align}
and the remaining terms collect as
\begin{align}
&\big[t_{24},[t_{12},t_{23}]\big]+\big[t_{12},[t_{34},t_{23}]\big]-\big[t_{34},[t_{12},t_{23}]\big]-\big[t_{13},[t_{34},t_{23}]\big]\nn\\&=\big[t_{13},[t_{23},t_{34}]\big]-\big[t_{24},[t_{23},t_{12}]\big]\nn\\&=\partial\left(\big[t_{13},\L_{234}\big]-\big[t_{24},\R_{123}\big]\right)-\big[t_{13},[t_{23},t_{24}]\big]+\big[t_{24},[t_{23},t_{13}]\big]\nn\\&=\partial\left(\big[t_{13},\L_{234}\big]-\big[t_{24},\R_{123}\big]\right)\quad.
\end{align}
\end{enumerate}
Putting it all together, the pentagonator series (modulo $\hbar^4$) is given by
\begin{equation}\label{eq:third-order of pentagonator}
\Pi=\hbar^3\zeta(3)\left(\big[t_{34}-t_{24},\L_{123}\big]-[3t_{24}+t_{34},\R_{123}]+[t_{12}+3t_{13},\L_{234}]+\big[t_{13}-t_{12},\R_{234}\big]\right)
\end{equation}

\subsection{The right pre-hexagonator}
We compute the third-order term of the right pre-hexagonator \eqref{eq:symstr right pre-hex},
\begin{align}
(\RR)_3:&\zeta(3)\left(\big[[t_{12},t_{13}],t_{1(23)}\big]+\big[[t_{23},t_{12}],t_{12}+t_{23}\big]\right)-\frac{\pi^3i}{6}\left([t_{12},t_{13}]t_{(12)3}+t_{(12)3}[t_{23},t_{12}]+t_{(12)3}^3\right)\nn\\
&\Rrightarrow\zeta(3)\big[[t_{23},t_{13}],t_{(12)3}\big]-\frac{\pi^3i}{6}\left([t_{23},t_{13}]t_{23}+t_{13}[t_{23},t_{13}]+t_{13}^3+t_{23}^3+3t_{13}^2t_{23}+3t_{13}t^2_{23}\right)\,.
\end{align}
Balancing terms, we get
\begin{align}
&\zeta(3)\left(\big[[t_{12},t_{13}],t_{1(23)}\big]-\big[[t_{23},t_{13}],t_{(12)3}\big]+\big[[t_{23},t_{12}],t_{12}+t_{23}\big]\right)\nn\\&-\frac{\pi^3i}{6}\left([t_{23},t_{13}^2]+[t_{23}^2,t_{13}]+[t_{12},t_{13}]t_{(12)3}+t_{(12)3}[t_{23},t_{12}]\right)\Rrightarrow0\quad.
\end{align}
We rewrite this as
\begin{align}
&\zeta(3)\left(\big[t_{(12)3},[t_{23},t_{13}]\big]-\big[t_{1(23)},[t_{12},t_{13}]\big]+\big[t_{23},[t_{12},t_{23}]\big]-\big[t_{12},[t_{23},t_{12}]\big]\right)\nn\\&-\frac{\pi^3i}{6}\left([t_{12}+t_{23},t_{13}]t_{(12)3}+t_{(12)3}[t_{23},t_{1(23)}]\right)\Rrightarrow0\quad.
\end{align}
The terms with coefficient $-\frac{\pi^3i}{6}$ can be simply written as,
\begin{equation}
\left([t_{12},t_{(12)3}]+[t_{23},t_{1(23)}]\right)t_{(12)3}+t_{(12)3}[t_{23},t_{1(23)}]=\left(\partial\L+\partial\R\right)t_{(12)3}+t_{(12)3}\partial\R\quad.
\end{equation}
The coefficient of $\zeta(3)$ takes a bit more work,
\begin{align}
&\big[t_{(12)3},[t_{23},t_{1(23)}-t_{12}]\big]-\big[t_{1(23)},[t_{12},t_{(12)3}-t_{23}]\big]+\big[t_{23},[t_{12},t_{23}]\big]-\big[t_{12},[t_{23},t_{12}]\big]\nn\\&=[t_{(12)3},\partial\R]-\big[t_{(12)3},[t_{23},t_{12}]\big]-[t_{1(23)},\partial\L]+\big[t_{1(23)},[t_{12},t_{23}]\big]+\big[t_{23},[t_{12},t_{23}]\big]-\big[t_{12},[t_{23},t_{12}]\big]\nn\\&=[t_{(12)3},\partial\R]-[t_{1(23)},\partial\L]-\big[[t_{(12)3},t_{23}],t_{12}\big]-\big[t_{23},[t_{(12)3},t_{12}]\big]+\big[[t_{1(23)},t_{12}],t_{23}\big]+\big[t_{12},[t_{1(23)},t_{23}]\big]\nn\\&\qquad+\big[t_{23},[t_{12},t_{23}]\big]-\big[t_{12},[t_{23},t_{12}]\big]\nn\\&=[t_{(12)3},\partial\R]-[t_{1(23)},\partial\L]-\big[2t_{12},[t_{23},t_{1(23)}]\big]+\big[2t_{23},[t_{12},t_{(12)3}]\big]\nn\\&=[t_{(12)3}-2t_{12},\partial\R]+[2t_{23}-t_{1(23)},\partial\L]
\end{align}
In sum, we have the following expression for the right pre-hexagonator series $\RR$ modulo $\hbar^4$,
\begin{equation}\label{eq:3rd order right pre-hex}
\RR=-\frac{\pi^2\hbar^2}{6}(\L+2\R)+\hbar^3\left[\zeta(3)\big([t_{(12)3}-2t_{12},\R]+[2t_{23}-t_{1(23)},\L]\big)-\frac{\pi^3i}{6}\big(\left(\L+\R\right)t_{(12)3}+t_{(12)3}\R\big)\right]\,.
\end{equation}

\subsection{Higher-order coboundary syllepses}
This subsection follows the theme of Remarks \ref{rem:coboundary 2-shifted induces syllep} and \ref{rem:h^2 coboundary syllepsis}.
\begin{defi}
Given a coherent totally symmetric strict $t$ and the relevant monoidal data, 
\begin{equation}
\left(\CC[[\hbar]]\,,\,\otimes\,,\,\alpha:\equiv\Phi(t_{12},t_{23})\,,\,\sigma:\equiv\gamma e^{\pi i\hbar t}\,,\,\H^L:=\gamma_{1(23)}\LL\,,\,\H^R:=\gamma_{(12)3}\RR\right)\quad,
\end{equation}
if we have a syllepsis $\Sigma:=\gamma\Gamma:\gamma e^{\pi i\hbar t}\Rrightarrow\gamma e^{-\pi i\hbar t}$, then we call $\Gamma:e^{\pi i\hbar t}\Rrightarrow e^{-\pi i\hbar t}$ a \textbf{pre-syllepsis}.
\end{defi}
We wish to strip away instances of $\gamma$ in the factorisation axioms \eqref{eq:left factorisation}/\eqref{eq:right factorisation} that a syllepsis must satisfy. We will also do the same for the symmetry property \eqref{eq:symmetric syllepsis}. 
\begin{constr}
Let us treat the right factorisation axiom \eqref{eq:right factorisation},
\begin{equation}
\H^R=\alpha^{-1}_{312}\Sigma_{(12)3}\alpha^{-1}+(\H^L)^{-1}_{312}-\sigma^{-1}_{31}\alpha^{-1}_{132}\Sigma_{23}-\Sigma_{13}\alpha^{-1}_{132}\sigma_{23}\quad.
\end{equation}
We use the formula in Remark \ref{rem:inverse isomodification} to write the inverse of the left hexagonator \eqref{eq:index left hexagonator} as
\begin{equation}
(\H^L)^{-1}=-\alpha^{-1}\sigma_{1(23)}^{-1}\alpha^{-1}_{231}\H^L\sigma^{-1}_{12}\alpha^{-1}_{213}\sigma^{-1}_{13}\quad\implies\quad(\H^L)^{-1}_{312}=-\alpha^{-1}_{312}\sigma_{3(12)}^{-1}\alpha^{-1}\H^L_{312}\sigma^{-1}_{31}\alpha^{-1}_{132}\sigma^{-1}_{32}\,.
\end{equation}
Hence, we arrive at the following index equation
\begin{align}
\RR=\Phi(t_{12},t_{13})\Gamma_{(12)3}\Phi(t_{23},t_{12})&-\Phi(t_{12},t_{13})e^{-\pi i\hbar t_{(12)3}}\Phi(t_{23},t_{12})\RR_{213}e^{-\pi i\hbar t_{13}}\Phi(t_{23},t_{13})e^{-\pi i\hbar t_{23}}\nn\\&-e^{-\pi i\hbar t_{13}}\Phi(t_{23},t_{13})\Gamma_{23}-\Gamma_{13}\Phi(t_{23},t_{13})e^{\pi i\hbar t_{23}}\qquad.\label{eq:index right factorisation}
\end{align}
Analogous arguments show the left factorisation axiom \eqref{eq:left factorisation} to yield 
\begin{align}
\RR_{321}=\Phi(t_{23},t_{13})\Gamma_{1(23)}\Phi(t_{12},t_{23})&-\Phi(t_{23},t_{13})e^{-\pi i\hbar t_{1(23)}}\Phi(t_{12},t_{23})\RR_{231}e^{-\pi i\hbar t_{13}}\Phi(t_{12},t_{13})e^{-\pi i\hbar t_{12}}\nn\\&-e^{-\pi i\hbar t_{13}}\Phi(t_{12},t_{13})\Gamma_{12}-\Gamma_{13}\Phi(t_{12},t_{13})e^{\pi i\hbar t_{12}}\qquad.\label{eq:index left factorisation}
\end{align}
Finally, the symmetry property \eqref{eq:symmetric syllepsis} simply translates as
\begin{equation}\label{eq:index symmetric syllepsis}
\Gamma=e^{-\pi i\hbar t}\Gamma_{21}e^{\pi i\hbar t}=e^{-\pi i\hbar\ad_t}\left(\Gamma_{21}\right)\qquad.
\end{equation}
\end{constr}
\begin{rem}
Recalling the coboundary infinitesimal syllepsis $T:t\Rrightarrow0$ of Remark \ref{rem:coboundary 2-shifted induces syllep}, we consider the candidate pre-syllepsis
\begin{equation}\label{eq:candidate pre-syllepsis}
\Gamma=2\pi i\hbar\sum_{m=0}^\infty\frac{(\pi i\hbar t)^{2m}}{(2m+1)!}T\qquad.
\end{equation}
Remark \ref{rem:coboundary 2-shifted induces syllep} already tells us that $T$ satisfies:
\begin{equation}
T_{1(23)}=T_{12}+T_{13}\qquad,\qquad T_{(12)3}=T_{23}+T_{13}\qquad,\qquad T_{12}=T_{21}\quad.
\end{equation}
In particular, the symmetry of the coboundary infinitesimal syllepsis $T_{12}=T_{21}$ and that of the infinitesimal 2-braiding $t_{12}=t_{21}$ implies $\Gamma_{12}=\Gamma_{21}$ thus simplifying \eqref{eq:index symmetric syllepsis} to
\begin{equation}\label{eq:Gamma invariant under ad of t}
\Gamma=e^{-\pi i\hbar\ad_t}\left(\Gamma\right)\quad.
\end{equation}
Now \eqref{eq:h^2 symmetry of syllep} reduces the all-orders symmetry relation \eqref{eq:Gamma invariant under ad of t} to the tautology $\Gamma=\Gamma$ hence the only relevant relations going forward are the factorisation relations \eqref{eq:index right factorisation}/\eqref{eq:index left factorisation}.
\end{rem}

We now wish to show that
\begin{equation}
\Gamma=2\pi i\hbar T-\frac{\pi^3i\hbar^3}{3}t^2T
\end{equation}
satisfies the factorisation relations \eqref{eq:index right factorisation}/\eqref{eq:index left factorisation} required of a symmetric pre-syllepsis on the braided strictly-unital monoidal $\Ch_{\bbC[\hbar]/(\hbar^4)}^{[-1,0]}$-category $\left(\CC_{\bbC[\hbar]/(\hbar^4)},\otimes,I,\alpha,\Pi,\sigma,\H^L,\H^R\right)$. At order $\hbar^3$, the right factorisation relation \eqref{eq:index right factorisation} reads as
\begin{align}
&\zeta(3)\big([t_{(12)3}-2t_{12},\R]+[2t_{23}-t_{1(23)},\L]\big)-\frac{\pi^3i}{6}\big(\left(\L+\R\right)t_{(12)3}+t_{(12)3}\R\big)\\
&=-\frac{\pi^3i}{3}t^2_{(12)3}T_{(12)3}-\frac{\pi^3i}{3}[t_{12},t_{13}]T_{(12)3}-\frac{\pi^3i}{3}T_{(12)3}[t_{23},t_{12}]\nn\\&\quad-\zeta(3)\big([t_{(12)3}-2t_{12},\R_{213}]+[2t_{13}-t_{2(13)},\L_{213}]\big)+\frac{\pi^3i}{6}\big(\left(\L_{213}+\R_{213}\right)t_{(12)3}+t_{(12)3}\R_{213}\big)\nn\\&\quad-\frac{\pi^3i}{6}\big(\left(\L_{213}+2\R_{213}\right)t_{(12)3}+t_{(12)3}\left(\L_{213}+2\R_{213}\right)\big)\nn\\&\quad+
\frac{\pi^3i}{3}t^2_{23}T_{23}+\frac{\pi^3i}{3}[t_{23},t_{13}]T_{23}+\pi^3it^2_{13}T_{23}+
\frac{\pi^3i}{3}t^2_{13}T_{13}+\frac{\pi^3i}{3}T_{13}[t_{23},t_{13}]+\pi^3iT_{13}t^2_{23}\quad.\nn
\end{align}
Coherency of $t$ gives us $\R_{213}=-(\L+\R)$ and total symmetry gives us $\L_{213}=\L$ which leaves 
\begin{equation}
\frac{\pi^3i}{3}\big(\left(\L+\R\right)t_{(12)3}+t_{(12)3}\R\big)=\frac{\pi^3i}{3}\big([t_{12}+t_{23},t_{13}]T_{(12)3}+T_{(12)3}[t_{23},t_{1(23)}]\big)
\end{equation}
on the RHS. Using the strictness of $t$ and $T$ together with the $2^\mathrm{nd}$-order relations \eqref{eq:h^2 factorisation}/\eqref{eq:h^2 symmetry of syllep} guarantees the above $3^\mathrm{rd}$-order relation. The left factorisation relation \eqref{eq:index left factorisation} is proved likewise.


\end{document}